\documentclass{amsart}
\usepackage[utf8]{inputenc}

\usepackage{amssymb,comment}
\usepackage{mathrsfs}
\usepackage{stmaryrd}

\usepackage{tikz-cd}

\usepackage{mathabx}

\usepackage{lmodern}
\usepackage{csquotes}

\usepackage[T1]{fontenc}

\definecolor{darkgreen}{rgb}{0,0.5,0}
\definecolor{darkblue}{rgb}{0,0,0.8}
\definecolor{darkred}{rgb}{0.8,0,0}
\definecolor{lightblue}{rgb}{0,0.6,0.8}

\usepackage[pdfencoding=auto,colorlinks,citecolor=darkgreen,linkcolor=darkblue,urlcolor=darkred]{hyperref}

\usepackage[english]{babel}
\usepackage{url}

\usepackage{multirow}

\usepackage{enumerate}
\usepackage{colonequals}
\usetikzlibrary{matrix, calc, arrows}

\DeclareFontFamily{U}{wncy}{}
\DeclareFontShape{U}{wncy}{m}{n}{<->wncyr10}{}
\DeclareSymbolFont{mcy}{U}{wncy}{m}{n}
\DeclareMathSymbol{\Sha}{\mathord}{mcy}{"58}

\usepackage{microtype}

\theoremstyle{plain}
\newtheorem{theorem}{Theorem}[subsection]
\newtheorem{lemma}[theorem]{Lemma}
\newtheorem{corollary}[theorem]{Corollary}
\newtheorem{proposition}[theorem]{Proposition}

\newtheorem{maintheorem}{Theorem}

\theoremstyle{definition}
\newtheorem{definition}[theorem]{Definition}

\newtheorem{assumption}[theorem]{Assumption}

\newtheorem{remark}[theorem]{Remark}

\newtheorem*{assumption*}{Assumption}
\newtheorem*{claim*}{Claim}

\theoremstyle{remark}

\usepackage{cleveref}
\crefname{theorem}{Theorem}{Theorems}
\crefname{lemma}{Lemma}{Lemmata}
\crefname{corollary}{Corollary}{Corollaries}
\crefname{proposition}{Proposition}{Propositions}
\crefname{definition}{Definition}{Definitions}
\crefname{conjecture}{Conjecture}{Conjectures}
\crefname{question}{Question}{Questions}
\crefname{example}{Example}{Examples}
\crefname{algorithm}{Algorithm}{Algorithms}
\crefname{remark}{Remark}{Remarks}
\crefname{assumption}{Assumption}{Assumptions}
\crefname{maintheorem}{Theorem}{Theorems}

\def\ol#1{\overline{#1}}
\def\wt#1{\widetilde{#1}}

\def\Alphabet{A,B,C,D,E,F,G,H,I,J,K,L,M,N,O,P,Q,R,S,T,U,V,W,X,Y,Z}
\def\alphabet{a,b,c,d,e,f,g,h,i,j,k,l,m,n,o,p,q,r,s,t,u,v,w,x,y,z}
\def\endpiece{xxx}
\def\makeAlphabet[#1]{\expandafter\makeA#1,xxx,}
\def\makealphabet[#1]{\expandafter\makea#1,xxx,}
\def\makeA#1,{\def\temp{#1}\ifx\temp\endpiece\else%
	\mkbb{#1}\mkfrak{#1}\mkbf{#1}\mkcal{#1}\mkscr{#1}\mkbs{#1}\expandafter\makeA\fi}%
\def\makea#1,{\def\temp{#1}\ifx\temp\endpiece\else\mkfrak{#1}\mkbf{#1}\mkbs{#1}\expandafter\makea\fi}%
\def\mkbb#1{\expandafter\def\csname bb#1\endcsname{\mathbb{#1}}}
\def\mkfrak#1{\expandafter\def\csname fr#1\endcsname{\mathfrak{#1}}}
\def\mkbf#1{\expandafter\def\csname b#1\endcsname{\mathbf{#1}}}
\def\mkcal#1{\expandafter\def\csname c#1\endcsname{\mathcal{#1}}}
\def\mkscr#1{\expandafter\def\csname s#1\endcsname{\mathscr{#1}}}
\def\mkbs#1{\expandafter\def\csname bs#1\endcsname{{\boldsymbol{#1}}}}
\def\makeop[#1]{\xmakeop#1,xxx,}
\def\mkop#1{\expandafter\def\csname #1\endcsname{{\operatorname{#1}}}} %
\def\xmakeop#1,{\def\temp{#1}\ifx\temp\endpiece\else\mkop{#1}\expandafter\xmakeop\fi}%
\def\makeup[#1]{\xmakeup#1,xxx,}
\def\mkup#1{\expandafter\def\csname #1\endcsname{{\mathrm{#1}\,}}} %
\def\xmakeup#1,{\def\temp{#1}\ifx\temp\endpiece\else\mkup{#1}\expandafter\xmakeup\fi}%
\makeAlphabet[\Alphabet]
\makealphabet[\alphabet]
\makeop[Hom,Tor,Ext,holim,hocolim,dim,Sel,GL,PGL,SL,H,ord,Sym,Gal,Ann,ind,Aut,End,cd,Frob,Gr,sp,Tot,sm,an,Pic,NS,Br,tors,Reg,ss,new,cts,ur,alg,pd,disc,cyc,rk,cork,ab,nr,tors,Iw,res,inf,Cl,Quot,corank,ord,length,proj,Hg,im,coker,Tam,loc,cores,Char,Fitt,free,Fil,MSD,Ind]
\makeup[Spec,Proj,Spwf,Sp,Sh,Spf,Sch,id,dR,rig,cris,ur,div,ac,BK]

\newcommand{\F}{\mathbf{F}}

\newcommand{\Q}{\mathbf{Q}}
\newcommand{\Qbar}{\ol\Q}
\newcommand{\Z}{\mathbf{Z}}
\newcommand{\KS}{\mathbf{KS}}

\newcommand{\fm}{\mathfrak{m}}

\newcommand{\fX}{\mathfrak{X}}

\newcommand{\et}{\mathrm{\acute{e}t}}

\renewcommand{\epsilon}{\varepsilon}
\renewcommand{\theta}{\vartheta}
\renewcommand{\phi}{\varphi}

\newcommand{\mathup}[1]{\text{\textup{#1}}}
\renewcommand{\H} {\ensuremath{\mathup{H}}}

\newcommand{\GalQ}{G_\Q}

\newcommand{\Mod}{\mathbf{Mod}}

\newcommand{\defn}{\colonequals}
\newcommand{\defeq}{\colonequals}

\newcommand{\iso}{\simeq}
\newcommand{\isom}{\cong}
\newcommand{\isoto}{\stackrel{\sim}{\longrightarrow}}
\newcommand{\inj}{\hookrightarrow}
\newcommand{\surj}{\twoheadrightarrow}

\renewcommand{\Im}{\operatorname{Im}}

\renewcommand{\injlim}{\varinjlim}

\numberwithin{equation}{section}

\begin{document}
\title[Anticyclotomic Iwasawa theory of newforms at Eisenstein primes]{On the anticyclotomic Iwasawa theory\\ of newforms at Eisenstein primes\\ of semistable reduction}
\author{Timo Keller}
\address{(T. Keller) Institut für Mathematik, Universität Würzburg, Emil-Fischer-Strasse 30, 97074,
	Würzburg, Germany}
\address{Rijksuniversiteit Groningen, Bernoulli Institute, Bernoulliborg, Nijenborgh 9, 9747 AG Groningen, The Netherlands}
\address{Leibniz Universität Hannover, Institut für Algebra, Zahlentheorie und Diskrete Mathematik, Welfengarten 1, 30167 Hannover, Germany}
\address{Department of Mathematics, Chair of Computer Algebra, Universität Bay\-reuth, Germany}
\email{math@kellertimo.de}
\urladdr{\url{https://www.timo-keller.de}}
\thanks{TK was supported by the Deutsche Forschungsgemeinschaft (DFG),
	Geschäftszeichen STO~299/18-1, AOBJ: 667349 and by the 2021 MSCA Postdoctoral Fellowship
	01064790 -- Ex\-pli\-cit\-Rat\-Points while working on this article.}
\author{Mulun Yin}
\address{(M. Yin) University of California Santa Barbara, South Hall, Santa Barbara, CA 93106, USA}
\email{mulun@ucsb.edu}
\date{\today}


\begin{abstract}
    Let $f$ be a newform of weight $k=2r$ and level $N$ with trivial nebentypus. Let $\frp\nmid 2N$ be a maximal ideal of the ring of integers of the coefficient field of $f$ such that the self-dual twist of the mod-$\frp$ Galois representation of $f$ is reducible with constituents $\phi,\psi$. Denote a decomposition group over the rational prime $p$ below $\frp$ by $G_p$. We remove the condition $\phi|_{G_p} \neq \mathbf{1}, \omega$ from~\cite{CGLS}, and generalize their results to newforms of higher weights $2r$ with $r$ being odd. As a consequence, we prove some Iwasawa Main Conjectures and get the $p$-part of the strong BSD Conjecture for elliptic curves of analytic rank $0$ or $1$ over $\Q$ in this setting. In particular, non-trivial $p$-torsion is allowed in the Mordell--Weil group. Using Hida families, we also prove an Iwasawa Main Conjecture for newforms of weight $2$ of multiplicative reduction at Eisenstein primes. In the above situations, we also get $p$-converse to the theorems of Gross--Zagier--Kolyvagin. The $p$-converse theorems have applications to Goldfeld's conjecture in certain quadratic twist families of elliptic curves having a $3$-isogeny.
\end{abstract}

\maketitle

\tableofcontents

\section*{Introduction}
\subsection{Statement of the main results.}\label{0.1}
Let $f \in S_k(\Gamma_0(N))$ with $k = 2r$ be a newform with coefficient field $\Q(f)$ and coefficient ring $\Z[f]$. Let $F$ be a finite extension of the completion of $\Q(f)$ at a chosen prime above $p>2$, and denote by $\cO$ the ring of integers in $F$. Let $\frp$ be the maximal ideal of $\cO$ and let $\pi$ be a generator of $\frp$. Assume that $\frp$ is an \emph{Eisenstein prime} for $f$, i.e., that the mod-$\frp$ residual Galois representation $\ol\rho_f$ of $\GalQ$ associated with $f$ is reducible. This means that its semisimplification is the direct sum of two characters. After replacing $\ol\rho_f$ by its self-dual twist $\ol\rho_f(1 - r)$, these characters are $\phi,\psi = \phi^{-1}\omega$. (In the case $k = 2$, $\ol\rho_f$ is the Galois representation of $\GalQ$ acting on $A_f[\frp](\Qbar)$ where $A_f/\Q$ is the modular abelian variety attached to $f$). Note that by Serre's Modularity Conjecture, a theorem of Khare--Kisin--Wintenberger, simple RM abelian varieties over $\Q$ are attached to newforms of weight $2$, and vice versa. For elliptic curves, good Eisenstein primes are known to be ordinary.

Let $K/\Q$ be a Heegner field for $N$, i.e., an imaginary quadratic field such that all primes dividing $N$ split completely in $K$. We further assume that $D_K$ is odd and $\neq -3$, and that $p = v\ol{v}$ splits in $K$.

In~\cite{CGLS}, it is proved for $k = 2$ and $\Z[f] = \Z$ that, under the assumption $\phi|_{G_v} \ne \mathbf{1}, \omega$ (where $\mathbf{1}$ is the trivial and $\omega$ the mod-$p$ cyclotomic character), the anticyclotomic Main Conjecture and Perrin-Riou's Heegner point Main Conjecture hold. From this, the $p$-converse to Gross--Zagier--Kolyvagin's theorem and the $p$-part of BSD in the analytic rank $1$ case (under the assumption that $\phi$ is ramified at $p$ and odd or unramified at $p$ and even, which was later removed by~\cite{CGS}) easily follow.

However, this excludes in particular the case where $A_f(\Q)[\frp] \neq 0$ in the case $k = 2$. In this article, we close this gap and generalize the results to all weights $2r$ with $r$ being odd and all coefficient rings.

Our main results are the Iwasawa Main Conjectures~\cref{Gr imc} and~\cref{Heeg imc}, a $p$-converse theorem~\cref{B} to the theorem of Gross--Zagier--Kolyvagin--Logach\"ev, and the $p$-part of the BSD formula~\cref{A} in our situation when $p$ does not divide the level of $f$. For newforms of weight $2$, we obtain Iwasawa Main Conjectures at Eisenstein primes of \textit{semistable} reduction (i.e., good reduction or bad multiplicative reduction). See~\cref{C}.

\subsection{Method of proof and outline of the paper.}

First, we assume $p$ does not divide the level of $f$ from~\cref{sec:algebraic-side} through~\cref{4}. Following the strategies in~\cite{CGLS}, the proof of our Iwasawa Main Conjectures for $f$ is divided into two steps. We first compare the Iwasawa $\lambda$- and $\mu$-invariants of the algebraic side of and analytic side, then prove a one side divisibility. The starting point is again the observation that a Main Conjecture should be equivalent to an imprimitive one. More specifically, for an appropriate Selmer group $\fX_f$ for $f$ and a corresponding  Bertolini--Darmon--Prasanna $p$-adic $L$-function $\cL_f$, there is an equality for their `$S$-imprimitive' versions\[
\Char(\fX^S_f)\Lambda^{\nr}=(\cL^S_f)
\]
as ideals in $\Lambda^{\nr}$, where the $p$-adic $L$-function for $f$ lives.
Here we consider the primitive and imprimitive unramified Selmer groups (which are identified with the corresponding Greenberg Selmer groups in~\cite{CGLS} in their setting) of the characters $\phi$ and $\psi$ and of $f$ over the anticyclotomic Iwasawa algebra $\Lambda \defeq \cO\llbracket \Gamma \rrbracket$, where $\Gamma$ is the Galois group of the anticyclotomic $\Z_p$-extension $K_\infty$ of $K$. Here $\Lambda^{\nr}\coloneq\Lambda\hat{\otimes}_{\bZ_p}\bZ_p^{\nr}$, for $\bZ_p^{\nr}$ the completion of the ring of integers of the maximal unramified extension of $\bQ_p$. Denote by $\F$ the residue field of $F$. 

In Section~\ref{sec:algebraic-side}, we show that the primitive unramified Selmer group of $f$ and the characters $\phi,\psi$ appearing in the semisimplification of $\ol\rho_f$ are $\Lambda$-torsion with $\mu$-invariant $0$. In most cases, we get a simple relation between the $\lambda$-invariants\begin{equation}\label{sumlambda}
\lambda(\fX^S_f)=\lambda(\fX^S_\phi)+\lambda(\fX^S_\psi)
\end{equation} as in~\cite{CGLS}, but there is a special case when one of the characters is trivial over $G_K$, where we get\begin{equation}\label{lambdadiff}
\lambda(\fX^S_f)+1=\lambda(\fX^S_\phi)+\lambda(\fX^S_\psi).
\end{equation}
From here, we can easily get the same comparisons between the corresponding primitive Selmer groups.
We mention that this difference in $\lambda$-invariants agrees with the Iwasawa Main Conjecture for the trivial character, which takes on a different form from others\[
\Char(\fX_\mathbf{1})=(\cL_\mathbf{1}\cdot T).\] 
Since we allow one of the characters to be trivial on $G_p$, the proofs of~\eqref{sumlambda} and~\eqref{lambdadiff} requires more general ideas than those in \textit{op.\ cit.} in that the Selmer groups $\fX^S_?$ ($?=f,\theta$) may not be \textit{almost divisible} in the sense that they might contain non-trivial finite $\Lambda$-submodules. Indeed, the vanishing results of some $\H^0$ and $\H^2$ cohomology groups in \textit{op.\ cit.} heavily rely on this condition and no longer hold in our situation. We will compare these groups for $f$ to those of the characters directly and obtain an expected relation between their Selmer groups.

By Ribet's Lemma (see~\cite{Bellaiche}), we can find a sub-lattice of the lattice for the $\frp$-adic Galois representation of $f$ such that the associated mod-$\frp$ Galois representation of $f$ has no nonzero trivial subrepresentation over $G_K$. More precisely, we can fix a non-split extension\[
0\to \F(\phi) \to \ol\rho_f \to \F(\psi) \to 0\] where the first character restricts to $\omega$ on $G_p$ so in particular it's non-trivial over $G_K$. We can assume this since both the Iwasawa Main Conjectures and the BSD Conjecture is invariant under isogenies. This choice of the lattice is crucial for our computations. 

In Section~\ref{sec:analytic side}, we show an analogous formula for the $\mu$- and $\lambda$-invariants for the Katz $p$-adic $L$-functions of the characters and the BDP $L$-function of $f/K$. That is,
\begin{equation*}
	\lambda(\cL^S_f)=\lambda(\cL^S_\phi)+\lambda(\cL^S_\psi),
\end{equation*}
where $\cL^S_?\coloneq \cL_?\cdot \prod_{w\in S} \cP_w(?)$ for certain elements $\cP_w(?)$ in $\Lambda$ (see~\cref{H1 Euler}, \cref{1.5}). Here $?=f,\ \phi\text{ or }\psi$. Using the known results on the Iwasawa Main Conjecture for the characters proved by Rubin~\cite{Rubin1991}, we get the equality between algebraic and analytic Iwasawa invariants for $f$.

Hence it suffices to show that the characteristic ideal of the unramified Selmer group divides the $p$-adic $L$-function. For this, we can refer to~\cite[§§\,3, 4]{CGLS} because we chose a lattice $T_f$ such that $\ol\rho_f$ has no nonzero trivial subrepresentation. However, some of their arguments only apply to a canonical Galois stable lattice $\tilde{T}$ of $\rho_f$, which is generally different from our choice and might have a trivial subrepresentation. We will take care of this by fully utilizing the flexibility of Kolyvagin systems in~\cref{Koly}. Our modifed Kolyvagin system argument yields one-side divisibilities of the Main Conjectures, and using the comparison of Iwasawa invariants from~\cref{sec:algebraic-side} and~\cref{sec:analytic side} we can turn the divisibilities into equalities. The Main Conjectures can then be applied to yield the main results of this paper for the good reduction part. 

The first anticyclotomic Iwasawa Main Conjecture we obtain is the following, which can be seen as a special case of Greenberg's Main Conjectures.
\begin{maintheorem}[\cref{IMC}\,(IMC2)]\label{Gr imc}
	Let $f\in S_k(\Gamma_0(N))$ be a newform of weight $k=2r$ where $r$ is odd and $p>2$ an ordinary Eisenstein prime not dividing $N$, and let $K$ be an imaginary quadratic field satisfying the hypotheses in~\cref{0.1}. Let $\fX_f$ be the Pontryagin dual of the unramified Selmer groups for $f$ (see~\cref{sec:selmer-groups-of-rhof} for precise definitions) and let $\cL_f$ be the corresponding BDP $p$-adic $L$-function. Then $\fX_f$ is $\Lambda$-torsion, and 
	\[\Char(\fX_f)\Lambda^{\nr}=(\cL_f)\]
	as ideals in $\Lambda^{\nr}$.
\end{maintheorem}

The assumption that $r$ is odd is a technical assumption that is only used in the proof of~\cref{Koly} to turn the Heegner cycles into a Kolyvagin system. It can be removed in certain situations (see~\cref{rodd}) and is not considered essential. In particular, we still have access to a infinite subfamily of some Hida family for which the Main Conjectures are known, which is the key to showing a Main Conjecture for a weight $2$ form of multiplicative reduction (see~\cref{5.1}).

A second Iwasawa Main Conjecture of Perrin-Riou type, which is equivalent to the one above, is also obtained.
\begin{maintheorem}[\cref{IMC}\,(IMC1)]\label{Heeg imc}
	Let $f\in S_k(\Gamma_0(N))$ be a newform of weight $k=2r$ where $r$ is odd and $p>2$ an ordinary Eisenstein prime not dividing $N$, and let $K$ be an imaginary quadratic field satisfying the hypotheses in~\cref{0.1}. Then both $\H^1_{\cF_\Lambda}(K,\bT)$ and $\cX=\H^1_{\cF_\Lambda}(K,M_f)^\vee$ have $\Lambda$-rank one (see~\cite[section 3]{CGLS} for general definitions), and the equality \begin{equation*}
		\Char_\Lambda(\cX_{\tors})= \Char_\Lambda(\H^1_{\cF_\Lambda}(K,\bT)/\Lambda\kappa_{\infty})^2
	\end{equation*}
holds in $\Lambda$.
\end{maintheorem}

We now formulate our main results on the $p$-part of strong BSD for elliptic curves of good reduction at Eisenstein primes.

\begin{maintheorem}[\cref{BSD E}]\label[maintheorem]{A}
	Let $E/\Q$ be an elliptic curve and $p > 2$ a prime of good reduction. Assume that $E$ admits a cyclic $p$-isogeny with kernel $C=\F_p(\phi)$ for some character $\phi\colon G_\bQ\to \F_p^\times$ (equivalently, $E[p]$ is reducible). Assume that $r_\an(E) \in \{0,1\}$. 
	
	Then the $p$-part of the BSD formula holds for $E/\Q$, i.e.,
	\[
	\ord_p\Big(\frac{L^{*}(E/\Q,1)}{\Omega_E \cdot \Reg_{E/\Q}}\Big) = \ord_p\Big(\frac{\Tam(E/\Q)\#\Sha(E/\Q)[p^\infty]}{(\#E(\Q)_\tors)^2}\Big).
	\]
	Here, $L^*(E/\Q,1)$ denotes the leading Taylor coefficient of $L(E/\Q,s)$ at $s = 1$.
\end{maintheorem}

Finally, we modify Skinner's use of Hida families in~\cite{Skinner} to prove a Main Conjecture in the multiplicative reduction case for weight $2$ forms in~\cref{5}. The obstruction is again that our imprimitive Selmer groups $\fX^S_?$ may have finite $\Lambda$-submodules. In particular, the characteristic ideals are no longer identified with the Fitting ideals. The formal argument can be carried out by replacing the Selmer groups by their free parts. We would also need to modify several pieces that go into the proof into the residually reducible setting.

\begin{maintheorem}[\cref{imc mult}]\label[maintheorem]{C}
	Let $f$ be a newform of weight $2$ and $p\|N$ an odd Eisenstein prime of multiplicative reduction. Let $K$ be an imaginary quadratic field satisfying the hypotheses in~\cref{0.1}. Then the following equation holds in $\Lambda^{\nr}$:
	\[\Char(\fX_f)\Lambda^{\nr}=(\cL_f).\]
\end{maintheorem}

We remark that our Main Conjecture would yield both rank $0$ and rank $1$ BSD formulae at multiplicative primes provided the results in~\cite{CGS} are extended to higher weight modular forms.

\subsection{Examples}

Here are the elliptic curves $E$ over $\Q$ in the~\cite{lmfdb} with $E(\Q)_\tors \isom \Z/p$, rank $0$ or $1$, and $p$ a prime of good reduction: \href{https://www.lmfdb.org/EllipticCurve/Q/?bad_quantifier=exclude&bad_primes=3&rank=0..1&torsion=%5B3%5D}{$p=3$}, \href{https://www.lmfdb.org/EllipticCurve/Q/?rank=0-1&bad_quantifier=exclude&bad_primes=5&torsion=%5B5%5D}{$p=5$}, and \href{https://www.lmfdb.org/EllipticCurve/Q/?rank=0-1&bad_quantifier=exclude&bad_primes=7&torsion=%5B7%5D}{$p = 7$}.

Here are concrete examples for each of these three infinite families, an example of smallest conductor in each case. Our main result implies the $p$-part of strong BSD for all their twists that have analytic rank $0$ or $1$, of which there exist infinitely many.
Case $p = 3$: \href{https://www.lmfdb.org/EllipticCurve/Q/19/a/3}{19a3} has torsion subgroup $\Z/3$ and rank $1$.
Case $p = 5$: \href{https://www.lmfdb.org/EllipticCurve/Q/11/a/3}{11a3} ($= X_1(11)$) has torsion subgroup $\Z/5$ and rank $0$.
Case $p = 7$: \href{https://www.lmfdb.org/EllipticCurve/Q/26/b/2}{26b2} has torsion subgroup $\Z/7$ and rank $1$.

There are also examples of higher dimension: The descent computations for absolutely simple modular abelian surfaces in~\cite{KellerStoll2022} and~\cite[§\,6]{KellerStoll2023} are not necessary anymore for odd primes of good reduction. For example, the genus $2$ curve $X\colon y^2 = 5x^6 + 10x^5 - 13x^4 - 30x^3 + 9x^2 + 20x - 8$ has absolutely simple semistable RM Jacobian $J$ of squarefree level $5 \cdot 13$ and rank $0$ associated to the newform \href{https://www.lmfdb.org/ModularForm/GL2/Q/holomorphic/65/2/a/c/}{65.2.a.c}. Note that this curves is not covered by \emph{op.\ cit.} The algorithms from \emph{op.\ cit.} prove the prime-to-$3$ part of strong BSD for $J/\Q$ with $\#\Sha(J/\Q)_\an = 2$. Our results allow us to also prove the $3$-part: 
Let the good ordinary prime $3$ split as $\frp\frp'$ in $\Z[f]=\Z[\sqrt{3}]$ such that $\rho_\frp$ is irreducible and $\rho_{\frp'}$ is reducible with a trivial $1$-dimensional subrepresentation coming from the $3$-torsion. The results and algorithms of \emph{op.\ cit.} prove that the Heegner index $I_K$ for $D_K = -131$ (so $3$ is split in $\cO_K$) equals $\frp'$ times some ideal lying above $2$, and that the Tamagawa product considered as an ideal in $\Z[f]$ equals $\frp'$. Hence~\cite[Theorem~5.2.2]{KellerStoll2023} proves that $\Sha(J/\Q)[\frp] = 0$. Our main theorem implies~\cite[eq.~(5.5)]{CGLS} with $p$-torsion replaced by $\frp$-torsion throughout, showing that $v_{\frp'}(\#\Sha(J/K)[\frp'^\infty]) = 2v_{\frp'}(I_K) - 2v_{\frp'}(\Tam(J/\Q)) = 2 - 2 = 0$. Since $\frp' \nmid 2$, we have $\Sha(J/K)[\frp'] \iso \Sha(J/\Q)[\frp'] \oplus \Sha(J^K/\Q)[\frp']$, hence we get $\Sha(J/\Q)[\frp'] = 0$, so strong BSD holds for $J/\Q$ with $\Sha(J/\Q) \isom \Z/2$. (Note that the argument with~\cite[eq.~(5.5)]{CGLS} also provides evidence for the conjecture that the Tamagawa product divides the Heegner index.)

\subsection{A $p$-converse theorem and applications to Goldfeld's conjecture}

As in the good ordinary case, we show the equivalence between~\cref{C} and a Heegner Point Main Conjecture (\cref{multHg}). As a consequence, we obtain the following.

\begin{maintheorem}[\cref{p converse}, \cref{p converse mult}]\label[maintheorem]{B}
	Let $A/\Q$ be a simple RM abelian variety associated to a newform $f$ and $\frP \mid p > 2$ a prime ideal of its endomorphism ring $\cO$ of good ordinary or bad multiplicative reduction such that $\rho_{A,\frP}$ is reducible.
	
	Then
	\[
	\corank_\cO \Sel_{\frP^\infty}(A/\Q) = r \in \{0,1\} \implies \rk_\cO A(\Q) = r_\an(f) = r,
	\]
	and $\Sha(A/\Q)[\frP^\infty]$ is finite.
\end{maintheorem}

Our $3$-converse theorem has several applications to Goldfeld's conjecture, improving the bounds for the proportion of rank~$1$ twists when being combined with results~\cite{BKLOS,ShnidmanIMRN,BruinFlynnShnidman} on the average $3$-Selmer rank. See~\cref{Goldfeld conjecture}.

\subsection{Ordering of the characters.}
Recall the short exact sequence relating the residual representation $\ol\rho_f$ to the characters\[
0\to \F(\phi) \to \ol\rho_f \to \F(\psi) \to 0,\]
where we assumed that the first character $\phi$ restricts to $\omega$ on $G_p$ and that the sequence is non-split. From the invariance of the Main Conjectures and the BSD formulae, we could as well fix another ordering where the first character restricts to $\mathbf{1}$ on $G_p$. As is already noted in~\cref{Koly}, the use Ribet's lemma is somewhat limited, and some results for a different ordering should also be necessary. In fact, our original approach assumed that locally, the trivial character is a subrepresentation, regardless of whether the sequence is split or not. This made our computations in Section~\ref{sec:algebraic-side} easier. However, we had to heavily extend the results on the Kolyvagin system from~\cite[§\,3]{CGLS}, or rather~\cite{How2004}, to allow non-trivial torsion group $\H^0(K,\ol\rho_f)$. As our results might be useful for the future, we included them in the Appendix~\ref{sec:appendix}. We also deduce a version of the anticyclotomic control theorem which allows non-trivial torsion in Appendix~\ref{appB}. It should be mentioned that a direct proof of everything with torsion is possible, but we do not include it in this paper for simplicity.

\subsection{Relation to previous works.}

Results in this paper in the good reduction case can be seen as direct generalizations of those in~\cite{CGLS} to higher weight modular forms while removing the technical condition $\theta|_{G_p}\ne \mathbf{1},\omega$. Thus it would automatically generalize the results which require the main theorem in \textit{op.\ cit.} as an input, for example~\cite[Theorem A]{CGS}. We also obtain generalizations of the structure theorems in~\cite{How2004} as well as the anticyclotomic control theorem in~\cite{JSW2017} in the residually reducible case which allow torsion. Since our characters are more general than those in~\cite{CGLS}, the results on the proportion of ranks of quadratic twists of an elliptic curve over $\bQ$ with a rational $3$-isogeny are also improved, following~\cite{BKLOS}. 

Finally, we obtain new results for newforms with multiplicative reduction at Eisenstein primes which include a proof of an Iwasawa Main Conjecture. These results are analogous to those in~\cite{Skinner}, but the residually reducible case were untouched before.

\subsection{Future work}
As is mentioned before, a generalization of the results in~\cite{CGS} to higher weight modular forms would yield BSD formulas for elliptic curves at multiplicative reduction. This is what we will examine in later work.

In our subsequent paper~\cite{KY24}, we have obtained the anticyclotomic Main Conjectures for elliptic curves at Eisenstein primes of potentially good ordinary reduction. As a consequence, we obtain the $p$-converse theorems in that setting and the applications to Goldfeld's Conjecture (see~\cref{Goldfeld conjecture}) are much improved in some situations. We would like to study potentially multiplicative case and potentially supersingular case in future work, as both see very interesting applications in arithmetic statistics.

\subsection{Notation.}

Pontryagin duals are denoted by $(-)^\vee$. The mod-$p$ cyclotomic character is denoted by $\omega$ and the $p$-adic cyclotomic character by $\chi$. For any subextension $L/K$ of $\bar{K}/K$, denote by $G_L$ the absolute Galois group of $L$.

\subsection{Acknowledgements.} We thank Debanjana Kundu and Otmar Venjakob for helpful discussions and Ari Shnidman for pointing out the application of our $p$-converse theorem to Goldfeld's conjecture. We also thank MY's advisor Francesc Castella for his guidance throughout this project. This work is part of MY's forthcoming Ph.D.\ thesis. 
\section{Algebraic side}\label{sec:algebraic-side}

Let $f \in S_{k}(\Gamma_0(N))^\new$ be a newform of weight $k = 2r \geq 2$ and level $N$ with trivial nebentypus. Fix an odd prime $p \nmid N$. Let $\Z[f]$ be the coefficient ring of $f$ with field of fractions $\Q(f)$. Let $F$ be a finite extension of a completion of $\bQ(f)$ at a chosen prime above $p$ with ring of integers $\cO$. Let $\frp=(\pi)$ be the maximal ideal of $\cO$ and let $\F/\frp$ be the residue field of $\cO$. 
We use the Pontryagin dual $(-)^\vee = \Hom_{\cO,\cts}(-, F/\cO)$ for locally compact $\cO$-modules.

Recall that $\omega$ is the mod-$p$ cyclotomic character.
Denote by
\[
    \rho_f\colon \Gal(\Qbar/\Q) \to \Aut_F(V_f(1-r)) \isom \GL_2(F)
\]
the \emph{self-dual Tate twist} of the $p$-adic Galois representation attached to $f$.
It has determinant $\chi$.

Let $K \subset \Qbar$ be an imaginary quadratic field in which $p = v\ol{v}$ splits, with $v$ the prime of $K$ above $p$ induced by $\iota_p$. We also fix an embedding $\iota_\infty\colon \Qbar \inj \bC$.

We assume that $K$ satisfies the following \textit{Heegner hypothesis}:\begin{equation*}
\text{every prime } \ell\mid N \text{ splits in } K
\end{equation*}

Let $G_K \defeq \Gal(\Qbar/K) \subset \GalQ \defn \Gal(\Qbar/\Q)$ be the absolute Galois group of $K$, and for each place $w$ of $K$ let $I_w \subset G_w \subset G_K$ be the corresponding inertia and decomposition groups, unique up to conjugation. Let $\Frob_w \in G_w/I_w$ be the arithmetic Frobenius. For the prime $v \mid p$, we assume $G_v$ is chosen so that it is identified with $\Gal(\Qbar_p/\Q_p)$ via $\iota_p$.

Let $\Gamma \defn \Gal(K_\infty/K)$ be the Galois group of the anticyclotomic $\Z_p$-extension $K_\infty$ of $K$, and let $\Lambda \defn \cO\llbracket\Gamma\rrbracket$ be the anticyclotomic Iwasawa algebra over $\cO$. We shall often identify $\Lambda$ with the power series ring $\cO\llbracket T\rrbracket$ by setting $T = \gamma - 1$ for a fixed topological generator $\gamma \in \Gamma$. 

We record the following well-known properties of the anticyclotomic $\Z_p$-extension $K_\infty/K$ of an imaginary quadratic field $K$, i.e., the $\Z_p$-extension on which complex conjugation acts as inversion. Readers who are only interested in the main results of this paper can directly go to~\cref{1.1}.
\begin{lemma}[splitting of primes in $K_\infty$] \label[lemma]{local behavior of anticyclotomic Zp extension at ell}
	Let $w \mid \ell \neq p$ be a place of $K$. If $\ell$ is unramified and split in $K$, there are only finitely many primes of $K_\infty$ lying above $\ell$. If $\ell$ is ramified or inert in $K$, $\ell$ splits completely in $K_\infty$.
\end{lemma}
Note that this is in contrast to the \emph{cyclotomic} $\Z_p$-extension, in which no primes are infinitely split.

\begin{lemma}[local behavior of $K_\infty$ above $p$] \label[lemma]{local behavior of anticyclotomic Zp extension at p}
Let $K$ be an imaginary quadratic field and $p = v\ol{v}$ a rational prime completely split in $K$. Let $K_\infty/K$ be the anticyclotomic $\Z_p$-extension.
\begin{enumerate}[(i)]
	\item $\Gamma$ is the $\Z_p$-quotient of $\varprojlim_n \Gal(K_n/K)$, where $K_n$ is the ring class field of the order $\cO_{p^n} \colonequals \Z + p^n\cO_K$. The unramified part of $\Gamma$ at $v \mid p$ is cyclic of order dividing $h_K = \#\Gal(K_0/K)$, and $K_{\infty,v}/K_{1,v}$ is totally ramified at $v$.
	
	\item The completion $K_{\infty,v}$ of $K_\infty$ at $v \mid p$ is a ramified $\Z_p$-extension of $K_v = \Q_p$ different from $K_{v,cyc}$, the cyclotomic $\Z_p$-extension of $K_v$. In particular, $K_{\infty,v}\cap K_{v,cyc}$ is a finite extension of $K_v$.
	
	\item Let $v \mid p$ be a place of $K$. The image of the compositions $G_v \inj G_K \surj \Gamma$ and $I_v \inj G_v \surj \Gamma$ are open. If $p \nmid h_K$, the composition $I_v \inj G_v \surj \Gamma$ is surjective.
\end{enumerate}
\end{lemma}
\begin{proof}
\begin{enumerate}[(i)]
\item Note that $K_0$ is the Hilbert class field of $K$, so by class field theory $\Frob_v \in \Gal(K_0/K)$ has order equal to the order of $v$ in the class group of $K$. All primes above $p$ are totally ramified in $K_n/K_0$, and one has canonical isomorphisms $\Gal(K_n/K) \iso \Pic(\cO_n) \iso (\cO_K/p^n)^\times/(\Z/p^n)^\times$ from~\cite[equation~(7.27)]{CoxPrimes}.

\item The first claim follows from~\cite[Theorem~3.2]{LannuzelNguyen2000}; condition~(ii) there is satisfied for $K$. Thus 
\begin{center}$\Gal((K_{\infty,v}\cap K_{v,\cyc})/K_v)\isom\Gal(K_{\infty,v}/K_v)/\Gal(K_{\infty,v}/(K_{\infty,v}\cap K_{v,\cyc}))$
\end{center} 
must be a non-trivial quotient of $\Gal(K_{\infty,v}/K_v)\isom \bZ_p$ by an open subgroup, and hence it must be finite.

\item By the (ii), $G_v \inj G_{K_v}$ is open. The remaining statements follow from the (i).\qedhere

\end{enumerate}
\end{proof}

\begin{corollary}[the localization is an isomorphism away from $p$] \label[corollary]{restriction away from p}
Let $A$ be a $p^\infty$-torsion $G_{K_w}$-module. If $\ell \neq p$ is unramified and split with $w \mid \ell$ in $K$, then the restriction morphism $\H^1(K_{\infty,w}/K_w, A) \to \H^1(I_w,A)^{\Gamma/I_w}$ is an isomorphism.
\end{corollary}
\begin{proof}
See~\cite[Remark~3.1 and Lemma~3.2]{PollackWeston2011}. Note that the kernel and cokernel of the restriction morphism are the first and second Galois cohomology groups of a $p$-primary Galois module and that the profinite degree of the absolute Galois group $G_{\infty,w}$ of $K_{\infty,w}$ is prime to $p$.
\end{proof}

\begin{lemma}[the dual of the Iwasawa algebra] \label[lemma]{Lambda dual}
The Pontryagin dual $\Lambda^\vee$ of the Iwasawa algebra is isomorphic to $(F/\cO)[\injlim_n (\Gamma/\Gamma_n)^\vee] \iso (\injlim_nF/\cO)[(\Gamma/\Gamma_n)^\vee]$.
\end{lemma}
\begin{proof}
This follows from the exactness property of the Pontryagin dual, which interchanges projective and injective limits, and the definition of the Iwasawa algebra as a projective limit.
\end{proof}

\begin{corollary}[Galois cohomology of Iwasawa algebra equals Galois cohomology over the $\Z_p$-extension] \label[corollary]{cohomology of Mtheta}
One has $\H^i(K,\Lambda^\vee) \iso \H^i(K_\infty,F/\cO)$. For the twist of $\Lambda^\vee$ by a character $\theta\colon G_K \to \F^\times \inj \cO^\times$, one has $\H^i(K,\cO(\theta) \otimes_{\cO} \Lambda^\vee) \iso \H^i(K_\infty, F/\cO(\theta))$. If $K_\theta$ is obtained from $K$ by adjoining the values of $\theta$, then $K_\theta$ is linearly disjoint from $K_\infty$ and one has $\H^i(K_\infty, F/\cO(\theta)) \iso \H^i(K_\infty K_\theta, F/\cO)^{\Delta_\theta}$.
\end{corollary}
\begin{proof}
	See~\cite[proof of Theorem~1.2.2]{CGLS}. Note that their proof works for any $\theta$.
\end{proof}

\subsection{Local cohomology groups of characters.}\label{1.1}

Recall that $\F = \cO/\frp$ is the residue field of $\cO$, a finite field of characteristic $p$. Let $\theta\colon G_K \to \F^\times$ be a character with conductor divisible only by primes that are split in $K$. Since the case where $\theta|_{G_p}\ne \mathbf{1},\omega$ is already dealt with in~\cite{CGLS}, we focus on the complementary cases. Via the Teichmüller lift $\F^\times \inj F^\times$, we shall also view $\theta$ as taking values in $\cO^\times$. Set
\[
    M_\theta \defn \cO(\theta) \otimes_{\cO} \Lambda^\vee.
\]
 The module $M_\theta$ is equipped with a $G_K$-action via $\theta \otimes \Psi^{-1}$, where $\Psi\colon G_K \to \Lambda^\times$ is the character arising from the projection $G_K \surj \Gamma$. One has a $G_K$-equivariant isomorphism $M_\theta \isoto \Hom_{\cO,\cts}(\Lambda, \cO(\theta))$.

Note that one has an exact sequence
\begin{equation} \label{eq:kernel of T multiplication on Mtheta}
    0 \to (F/\cO)(\theta) \to M_\theta \stackrel{\cdot T}{\to} M_\theta \to 0.
\end{equation}

In this section, we study the local cohomology of $M_\theta$ at various primes $w$ of $K$.

\subsubsection{$w \nmid p$ split in K}

\begin{lemma} \label[lemma]{H1 Euler}
Let $\gamma_w$ be the image of $\Frob_w$ in the decomposition group of $w \mid \ell$. Let $\cP_w(\theta) = P_w(\ell^{-1}\gamma_w) \in \Lambda$ with $P_w \colonequals \det(1 - \Frob_wX \mid F(\theta)_{I_w})$ the Euler factor at $w$ of the $L$-function of $\theta$.

The module $\H^1(K_w, M_\theta)^\vee$ is $\Lambda$-torsion with characteristic ideal $(\cP_w(\theta))$. In particular, it has $\mu$-invariant $0$.
\end{lemma}
\begin{proof}
See~\cite[Lemma~1.1.1]{CGLS}. Note that their proof works for any $\theta$.
\end{proof}

\subsubsection{$w \mid p$ in $K$}

\begin{lemma} \label[lemma]{Iwasawa module free}
Let $X$ be a finitely generated $\Lambda$-module.

\begin{enumerate}[(i)]
	\item If 
	\begin{itemize}
		\item $X[T] = 0$ and
		\item $X/TX$ is a free $\cO$-module of rank $r$,
		\end{itemize}
		then $X$ is a free $\Lambda$-module of rank $r$ (and conversely).
	
	\item If
	\begin{itemize}
		\item $X[T] = 0$ and
		\item $\rk_\cO X/TX=r$,
		\end{itemize}
		then $\rk_\Lambda X=r$.
\end{enumerate}
\end{lemma}
\begin{proof}
\begin{enumerate}[(i)]
	\item This is~\cite[Lemma~1.1.2]{CGLS} (see also~\cite[Pro\-po\-si\-tion~5.3.19\,(ii)]{NSW2.3}).	
	\item Consider the structure theorem\[
	X\sim \Lambda^s\oplus\bigoplus_{i}\Lambda/(f_i)\oplus\bigoplus_j \Lambda/(\pi^{m_j})\]
	for some distinguished polynomials $f_i$.
	That $X[T]=0$ implies there is no $f_i$. Then $X/TX$ would have $\cO$-rank $s$ because $(\Lambda/(\pi^{m_j}))/T$ is finite for any $j$. Hence $r=s$.
\qedhere\end{enumerate}
\end{proof}

We now study the case $\theta|_{G_v} \in \{\mathbf{1}, \omega\}$ excluded in~\cite{CGLS}. Note that in each of these cases one property in~\cite{CGLS} fails.

\begin{proposition} \label[proposition]{H1ofMtheta}
Assume that $\theta|_{G_v} \in  \{\mathbf{1}, \omega\}$ (all other cases have been considered in the analogous~\cite[Proposition~1.1.3]{CGLS}). Then:
\begin{enumerate}[(i)]
    \item The restriction morphism
        \[
            r_w\colon \H^1(K_w, M_\theta) \to \H^1(I_w, M_\theta)^{G_w/I_w}
        \]
        is surjective.
        
        \begin{itemize}
            \item If $\theta|_{G_w} = \mathbf{1}$, its kernel is a cofree $\cO$-module of rank $1$. In particular, $ker(r_w)^\vee$ has $\Lambda$-rank $0$ and it has no non-trivial finite $\Lambda$-submodule because all non-trivial subgroups of $\cO$ are infinite. 
            
            \item If $\theta|_{G_w} = \omega$, it is an isomorphism.
        \end{itemize}
        
    \item 
    \begin{itemize}
        \item If $\theta|_{G_w} = \mathbf{1}$, $\H^1(K_w, M_\theta)$ is $\Lambda$-cofree of rank $1$.
        
        \item If $\theta|_{G_w} = \omega$, $X \colonequals \H^1(K_w, M_\theta)^\vee$ satisfies $\pd_\Lambda X \leq 1$, (hence) has no non-trivial finite ($\Lambda$-)submodule, and is generated by $2$ elements with $\Lambda$-rank equal to $1$.
    \end{itemize}
\end{enumerate}
\end{proposition}
\begin{proof}
Note that $K_w = \Q_p$ and that $K_{w,\infty}/K_w$ is an infinitely ramified $\Z_p$-extension by~\cref{local behavior of anticyclotomic Zp extension at p} not containing $\mu_p$ because $p > 2$ and $([K_w(\mu_p):K_w], p) = 1$.

(i): Since $\cd(K_w) = 1$ and $G_w/I_w \iso \hat{\Z}\cdot\Frob_w$, $r_w$ is surjective and $\ker(r_w) \iso M_\theta^{I_w}/(\Frob_w - 1)M_\theta^{I_w}$ by~\cite[Lemma~I.3.2\,(i)]{RubinEulerSystems}.
Let us now determine the $\ker(r_w)$ by determining its Pontryagin dual as a $\Lambda$-module. Dualizing
\[
    0 \to M_\theta^{G_w} \to M_\theta^{I_w} \xrightarrow{\Frob_w-1} M_\theta^{I_w} \to M_\theta^{I_w}/(\Frob_w-1)M_\theta^{I_w} \to 0
\]
and noting that Pontryagin duality transforms $G_w$-invariants into the quotient by $T$, $\ker(r_w)^\vee$ sits in an exact sequence
\begin{equation} \label{eq:rw}
    0 \to \ker(r_w)^\vee \to (M_\theta^\vee)_{I_w} \to (M_\theta^\vee)_{I_w} \to M_\theta^\vee/T \to 0.
\end{equation}

If $\theta|_{G_w} \neq \mathbf{1}$, $M_\theta^{G_w} = 0$, so $(M_\theta^\vee)_{I_w} \to (M_\theta^\vee)_{I_w}$ is a surjection of finitely generated modules over the Noetherian algebra $\Lambda$, hence an isomorphism.

If $\theta|_{G_w} = \mathbf{1}$, the exact sequence~\eqref{eq:rw} reads
\[
    0 \to \ker(r_w)^\vee \to \Lambda_{I_w} \to \Lambda_{I_w} \to \Lambda/T \to 0,
\]
i.e.,
\[
    0 \to \ker(r_w)^\vee \to \cO[\Gamma/I_w] \to \cO[\Gamma/I_w] \to \cO \to 0.
\]
Since $I_w \subseteq G_w$ is of finite index by~\cref{local behavior of anticyclotomic Zp extension at p}, $\Gamma/I_w \isom \Z/p^m$ with $p^m \mid h_K$. Hence all modules are $\cO$ torsion-free and $\rk_{\cO}\ker(r_w)^\vee = 1$. 
We now prove (ii). Let $X$ be the $\Lambda$-module $\H^1(K_w, M_\theta)^\vee$. It is finitely generated because it is finite mod $(p,T)$. We determine $X[T]$ and $X/TX$ if $\theta|_{G_w} \in \{\mathbf{1}, \omega\}$.

\textbf{Claim 0: $\H^i(K_w, M_\theta) = \H^i(K_{w,\infty}, F/\cO(\theta))$ and $\H^2(K_w, M_\theta) = 0$ independent of $\theta$.} 

The first statement is Shapiro's lemma. Note that $K_{w,\infty}/K_w$ is a (infinitely ramified) $\Z_p$-extension by~\cref{local behavior of anticyclotomic Zp extension at p}.  For the second statement, use the first and that from~\cite{NekovarSelmerComplexes}
\[
    \H^2(K_w, M_\theta) = \H^2(K_{w,\infty}, F/\cO) = \H^0_\Iw(K_{w,\infty}/K_w, \cO(\chi\theta^{-1})) = 0,
\]
where the $0$-th Iwasawa cohomology is $0$ by~\cite[Remark 5.11, Lemma 5.12]{SchneiderVenjakob}. 

\textbf{Claim 1:  $\H^1(K_w, M_\mathbf{1})/T = \H^1(K_w, M_\omega)/T = 0$.}

From Claim 0, there is an isomorphism
\[
    \H^1(K_w,M_\theta)/T \isoto \H^2(K_w, F/\cO(\theta)).
\]
From~\eqref{eq:Fpomega}, one gets
\[
    0 \to \H^1(K_w, F/\cO(\theta))/\frp \to \H^2(K_w, \F(\theta)) \to \H^2(K_w, F/\cO(\theta))[\frp] \to 0
\]
and
\[
    0 \to \H^2(K_w, F/\cO(\theta))/\frp \to \H^3(K_w, \F(\theta)) \to \H^3(K_w, F/\cO(\theta))[\frp] \to 0.
\]
In the first sequence, $\H^2(K_w, \F(\theta)) \iso \H^0(K_w, \F(\omega\theta^{-1}))^\vee$ by Tate's local duality. In the second sequence, the $\H^3$ terms are $0$ since $\cd(K_w) = 2$. Hence $\H^2(K_w, F/\cO(\theta))$ is $p$-divisible and finitely generated, hence isomorphic to $(F/\cO)^r$ for some $r$. If $\theta = \mathbf{1}$, $r = 0$, since the middle term in the first sequence is $0$, and so $\H^1(K_w, M_\mathbf{1})/T = 0$. If $\theta = \omega$, $\H^2(K_w, F/\cO(\omega))\isom \H^0(K_w,\bZ_p(\omega))^\vee$ by Tate's local duality (extended to $\cO$-modules using that projective limits are dual to injective ones) and the latter is $0$.

\textbf{Claim 2a: $\H^1(K_w,M_\mathbf{1})[T] = F/\cO$.}

 Equivalently, $(X/TX)^\vee = 0$ for the Pontryagin dual. The long exact cohomology sequence over $K_w$ for~\eqref{eq:kernel of T multiplication on Mtheta} gives an exact sequence
\begin{equation} \label{eq:H1Kw}
    0 \to \H^0(K_w, M_\theta)/T \to \H^1(K_w, F/\cO(\theta)) \to X^\vee[T] \to 0.
\end{equation}
By Claim~0, $\H^0(K_w, M_\theta)/T = (F/\cO)/T = F/\cO$, hence $(\H^0(K_w, M_\theta)/T)^\vee = \H^0(K_w, M_\theta)^\vee[T] = \cO$.

If $\theta|_{G_w} = \mathbf{1}$, the middle term $\H^1(K_w, F/\cO(\theta))$ is canonically isomorphic to $\Gal(K_w^{\ab,p}/K_w)^\vee = (K_w^\times)^{\wedge,p} = \Z_p^2$ by local class field theory since $\mu_p \not\subset K_w^\times$. Hence the Pontryagin dual of the exact sequence~\eqref{eq:H1Kw} reads
\[
    0 \to X/TX \to \cO^2 \to \cO \to 0,
\]
hence $\rk_{\cO} X/TX = 1$ and $X/TX$ is $\cO$-free, hence $X/TX \isom \cO$ and so $X^\vee[T] \isom F/\cO$.

Since $X[T] = 0$ from Claim~1, apply~\cref{Iwasawa module free} to $X \colonequals \H^1(K_w, M_\theta)^\vee$ to get (ii) in the case $\theta|_{G_w} = \mathbf{1}$.

\textbf{Claim 2b: $\H^1(K_w,M_\omega)[T] \isom F/\cO \oplus \cO/\frp^n$ for some $n > 0$.} 

If $\theta|_{G_w} = \omega$, we look at the sequence~\eqref{eq:H1Kw} again:
\[
    \H^0(K_w, M_\theta)/T = (F/\cO(\omega))^{G_{K_{w,\infty}}}/T = 0/T = 0
\]
because $\mu_p \not\subset K_{w,\infty}$: $\Gal(K_w(\mu_p)/K_w) \iso (\Z/p)^\times$ cannot be a quotient of the $\Z_p$-extension $K_{w,\infty}/K_w$ if $p > 2$ by~\cref{local behavior of anticyclotomic Zp extension at p}\,(iii). Hence
\[
    \H^1(K_w, F/\cO(\omega)) \isoto X^\vee[T] = (X/TX)^\vee.
\]
The long exact Galois cohomology sequence of
\begin{equation} \label{eq:Fpomega}
    0 \to \F(\omega) \to F/\cO(\omega) \stackrel{\pi}{\to} F/\cO(\omega) \to 0
\end{equation}
gives the short exact sequence
\[
    0 \to \H^0(K_w, F/\cO(\omega))/\frp \to \H^1(K_w, \F(\omega)) \to \H^1(K_w, F/\cO(\omega))[\frp] \to 0.
\]
We have already seen $\H^0(K_w, F/\cO(\omega)) = 0$, hence one has an isomorphism \[\H^1(K_w, \F(\omega)) \isoto \H^1(K_w, F/\cO(\omega))[\frp].\] One can compute the (finite) $\F$-dimension of $\H^1(K_w, \F(\omega))$ using the local Euler--Poincar\'{e} characteristic formula, or by noting that $\H^1(K_w, \F(\omega))$ is dual to $\H^1(K_w,\F)$ by Tate's local duality, and the latter is an $\F$-vector space of dimension $\dim_{\F}(K_w^\times)/p = 2$ by local class field theory, noting that $\mu_p \not\subset K_w^\times = \Q_p^\times$ since $p > 2$.

The Galois cohomology of~\eqref{eq:Fpomega} also gives an exact sequence
\[
    0 \to \H^1(K_w, F/\cO(\omega))/\frp \to \H^2(K_w, \F(\omega)) \to \H^2(K_w, F/\cO(\omega))[\frp] \to 0.
\]
The last term evaluates by Tate's local duality theorem to
\[
    \H^2(K_w, F/\cO(\omega))[\frp] \iso (\H^0(K_w, \cO(\chi\omega^{-1}))/\frp)^\vee = 0
\]
with $\chi\colon G_{K_w} \to \cO^\times$ the $p$-adic cyclotomic character and $\chi\omega^{-1}$ its projection $\chi^1$ to the principal units $1 + p\Z_p$ (note that $p > 2$). Hence
\begin{equation} \label{eq:H1KwQpZpomegamodp}
    \H^1(K_w, F/\cO(\omega))/\frp \isoto \H^2(K_w, \F(\omega)) = \H^0(K_w, \F)^\vee = \F
\end{equation}
by Tate's local duality.

Hence we know $\H^1(K_w, F/\cO(\omega))$ mod $\frp$ (namely $\F$) and its $\frp$-torsion (namely $\F^2$). As it is a cofinitely generated $\cO$-module (since its $\frp$-torsion is finite dimensional), writing it as $(F/\cO)^r \oplus A$ with $A$ a finite $\cO$-module, $A/\frp \isom \F$ from~\eqref{eq:H1KwQpZpomegamodp}, hence $A \isom \cO/\frp^n$ for some $n > 0$, and so $\F^2 \isom \H^1(K_w, \F(\omega)) = \H^1(K_w, F/\cO(\omega))[\frp] = \F^{r+1}$. Hence $r = 1$ and $\H^1(K_w, F/\cO(\omega)) \isom F/\cO \oplus \cO/\frp^n$ with $n > 0$, agreeing with the local Euler--Poincar\'{e} characteristic formula $0 - 2 + 1 = \chi(K_w, \F(\omega)) = - [K_w:\Q_p] = -1$.

\textbf{Determination of $\H^1(K_w,M_\omega)$ as a $\Lambda$-module.}

Hence we know from~\eqref{eq:H1Kw} that $X/TX \isom \cO \oplus \cO/\frp^n$ with $n > 0$ in the case $\theta|_{G_w} = \omega$, so by~\cref{Iwasawa module free} $X$ is not $\Lambda$-free and is generated by $\leq 2$ elements over $\Lambda$, and that $X[T] = 0$ from Claim~1. When $X[T]=0$, $X$ is $\Lambda$-torsion free, and there is a presentation
\[\begin{tikzcd}
    0 \ar[r]& C     \ar[r]\ar[d,twoheadrightarrow]& \Lambda^2   \ar[r]\ar[d,twoheadrightarrow]& X   \ar[r]\ar[d,twoheadrightarrow]& 0 \\
    0 \ar[r]& C/T   \ar[r]\ar[d,"\isom"]& \Lambda^2/T \ar[r]\ar[d,"\isom"]& X/T \ar[r]\ar[d,"\isom"]& 0 \\
    0 \ar[r]    & \frp^n\cO   \ar[r]& \cO^2 \ar[r]& \cO \oplus \cO/\frp^n \ar[r]& 0 \\
\end{tikzcd}\]
of $X$ as a $\Lambda$-module. Since $C \subset \Lambda^2$ is $\Lambda$-torsion free and $C/T \isom \frp^n\cO$, $C \isom \Lambda$ by~\cref{Iwasawa module free}. Hence the presentation implies that $\rk_\Lambda X = 1$ and $\pd_\Lambda X \leq 1$, hence $X$ has no finite $\Lambda$-submodule by~\cite[Proposition~(5.3.19)\,(i)]{NSW2.3}.
\end{proof}

\subsection{Selmer groups of characters} \label{ssec:Selmer groups of characters}

Let $\theta\colon G_K \to \F_\frp^\times$ be a character whose conductor is divisible only by primes split in $K$ (i.e., unramified over $\Q$ and have degree $1$).

Let $\Sigma$ be a finite set of places of $K$ containing $\infty$ and the primes dividing $p$ or the conductor of $\theta$ and such that every finite place in $\Sigma$ is split in $K$. Denote the maximal extension of $K$ unramified outside $\Sigma$ by $K^\Sigma$. Recall that $p$ splits as $v\ol{v}$ in $K$.

\begin{definition} \label[definition]{def:Selmer groups of characters}
\begin{enumerate}
	\item The \emph{Greenberg Selmer group} of $\theta$ is
	{\small
		\[
		\H^1_{\cF_\Gr}(K, M_\theta) \colonequals \ker\Big\{\H^1(K^\Sigma/K, M_\theta) \xrightarrow{\res_\Gr} \prod_{w\in\Sigma, w \nmid p}\H^1(K_w,M_\theta) \times \H^1(K_{\ol v}, M_\theta)\Big\}
		.\]}
	
	\item For $S \colonequals \Sigma \setminus \{v,\ol v,\infty\}$, let the \emph{$S$-imprimitive Selmer group} of $\theta$ be
	\[
	\H^1_{\cF_\Gr^S}(K, M_\theta) \colonequals \ker\Big\{\H^1(K^\Sigma/K, M_\theta) \to \H^1(K_{\ol v}, M_\theta)\Big\},
	\]
	Replacing $M_\theta$ by $M_\theta[\frp]$ in the above definitions, we obtain the \emph{residual Greenberg Selmer groups} $\H^1_{\cF_\Gr}(K, M_\theta[\frp])$ and $\H^1_{\cF_\Gr^S}(K, M_\theta[\frp])$.
	
	\item The \emph{Selmer group with unramified local conditions away from $v$}, also called \textit{unramified Selmer group} for short, of $\theta$ is
	{\small
		\[
		\H^1_{\cF_\nr}(K, M_\theta) \colonequals \ker\Big\{\H^1(K^\Sigma/K, M_\theta) \xrightarrow{\res_{\nr,\theta}} \prod_{w\in\Sigma, w \nmid p}\H^1(I_w,M_\theta)^{G_w/I_w} \times \H^1(I_{\ol v}, M_\theta)^{G_{\ol v}/I_{\ol v}}\Big\}.
		\]}
	
	\item For $S \colonequals \Sigma \setminus \{v,\ol v,\infty\}$, let the \emph{$S$-imprimitive unramified Selmer group} of $\theta$ be
	\[
	\H^1_{\cF_\nr^S}(K, M_\theta) \colonequals \ker\Big\{\H^1(K^\Sigma/K, M_\theta) \to \H^1(I_{\ol v}, M_\theta)^{G_{\ol v}/I_{\ol v}}\Big\},
	\]
	Replacing $M_\theta$ by $M_\theta[\frp]$ in the above definitions, we obtain the \emph{residual unramified Selmer groups} $\H^1_{\cF_\nr}(K, M_\theta[\frp])$ and $\H^1_{\cF_\nr^S}(K, M_\theta[\frp])$.
\end{enumerate}

\end{definition}

Note these Selmer groups are $\Lambda$-cofinitely generated (by finiteness of the $\frp$-Selmer groups and Nakayama's lemma), and the Selmer groups are independent of the choice of $\Sigma$ as reflected in their notation.

The local conditions for these \textit{unramified Selmer groups} are relaxed (namely, the full group) at $v$ and $\cL_v = \ker(r_w)$  for $w \in \Sigma \setminus \{v\}$ where $r_w\colon \H^1(K_w, M_\theta) \to \H^1(I_w, M_\theta)^{G_w/I_w}$ is the restriction in~\cref{H1ofMtheta}. Note that $\H^1(K^\Sigma/K, M_\theta) = \Sel^\Sigma(K, M_\theta)$ with the $\Sigma$-relaxed Selmer group and $\H^1_{\cF_\nr}(K, M_\theta) = \Sel^\Sigma_{\Sigma \setminus \{v\}}(K, M_\theta)$ with the $\Sigma \setminus \{v\}$-strict Selmer group by~\cite[Lemma~I.5.3]{RubinEulerSystems}. Analogously, $\H^1_{\cF_\nr^S}(K, M_\theta) = \Sel^\Sigma_{\{\ol v\}}(K, M_\theta)$ is the $\ol v$-strict Selmer group. The $S$-imprimitive Selmer group has less local conditions than the Selmer group, omitting the local conditions at $S$. Hence one has an inclusion $\H^1_{\cF_\nr}(K, M_\theta) \subseteq \H^1_{\cF_\nr^S}(K, M_\theta)$.

Note that because $K_w = \Q_p$ and $[\Q_p(\mu_p) : \Q_p] = p-1 > 1$, $\omega|_{G_w} \neq 1$. Hence also $\omega|_{G_K} \neq 1$. We need this because we use several times that $\omega$ has no invariants over $G_w$ and $G_K$.

\begin{theorem} \label[theorem]{Rubin Hida}
Assume that the conductor of $\theta$ is only divisible by primes that split in $K$. Then $\H^1_{\cF_\nr}(K, M_\theta)^\vee$ is a finitely generated $\Lambda$-torsion module with $\mu$-invariant $0$.
\end{theorem}
\begin{proof}
	The first part of this proposition is essentially proved in~\cite{Rubin1991} (with~\cite{Rubin1994} to remove the assumption on $h_K$). As explained in~\cite[Theorem 1.2.2]{CGLS}, the dual Selmer groups  $\H^1_{\cF_\nr}(K,M_\theta)^\vee$ are readily identified with the $\theta$-isotypic components of $\cX_\infty$ (which are shown to be $\Lambda$-torsion in~\cite[Remark (ii) after Theorem 4.1]{Rubin1991}) where $\cX_\infty$ is the Galois group of the maximal abelian pro-$p$ extension of $K_\infty K_\theta$ unramified outside $v$, $K_\theta$ being the fixed field of $\ker(\theta)$. 
	
	That $\mu$-invariants is $0$ follows from the vanishing result of~\cite{Hida2010} and the Iwasawa Main Conjectures in~\cite{Rubin1991} together with~\cite{CW78} (or more generally, \cite[III.1.10]{deShalit} to remove the assumption on $p\ne 3$, that $p$ is an anomalous prime and that $i\ne 0$ which corresponds to the case where $\theta|_{G_K}=\mathbf{1}$). More precisely, if $\theta|_{G_K}\ne \mathbf{1}$, then the Main Conjecture for $\theta$ identifies Char$(\fX_\theta)$ with $(\cL_\theta)$ and hence the algebraic $\mu$-invariant equals the analytic one which vanishes. If $\theta|_{G_K}=\mathbf{1}$, the Main Conjecture identifies Char$(\fX_\mathbf{1})$ with $(\cL_\mathbf{1}\cdot T)$ so again both $\mu$-invariants vanish. 
\end{proof}

\begin{remark} \label[remark]{remark on Selmer groups}
\begin{enumerate}[(i)]
	\item Recall that from~\cref{restriction away from p}, the local conditions of our Selmer group at $w\nmid p$ agree with the Greenberg's local conditions. This is also discussed in~\cite[Theorem 1.2.2]{CGLS}. As a consequence, our Selmer group can be alternatively defined as 
	\[
	\H^1_{\cF_\nr}(K, M_\theta) \colonequals \ker\Big\{\H^1(K^\Sigma/K, M_\theta) \xrightarrow{\res} \prod_{w\in\Sigma, w \nmid p}\H^1(K_w,M_\theta) \times \H^1(I_{\ol v}, M_\theta)^{G_{\ol v}/I_{\ol v}}\Big\}.
	\]
	\item 
	We remark here that~\cref{Rubin Hida} is crucial in the proof of the fact that the global-to-local map defining the above Selmer group (as well as the imprimitive one) is surjective.
\end{enumerate}
\end{remark}
\begin{proof}
Regarding (ii):
We want to show that the conditions~(1)--(4) of~\cite[Proposition~A.2]{PollackWeston2011} are satisfied for $\H^1(K, M_\theta) \iso \H^1(K_\infty, F/\cO(\theta))$ by Shapiro's lemma(see~\cref{cohomology of Mtheta}). The conditions are for $K/\Q$ an imaginary quadratic number field with $p = v\ol{v}$ and local conditions $\cL_w$:
\begin{enumerate}[(1)]
	\item[(1)] No place $w \mid p$ splits completely in $K_\infty$.
	
	\item[(2)] The Selmer group is $\Lambda$-cotorsion.
	
	\item[(3)] $\H^0(K_\infty,\Hom(\cO(\theta), F/\cO(1)))$ is finite.
	
	\item[(4)] $r_v + r_{\ol{v}} \defeq \corank_\Lambda \cL_v + \corank_\Lambda \cL_{\ol v} = \rk_\cO \cO(\theta) \equalscolon [K:\Q]d - d \equalscolon \delta(K,V)$.
\end{enumerate}
Condition~(1) is satisfied for the anticyclotomic $\Z_p$-extension by~\cref{local behavior of anticyclotomic Zp extension at p}. By~\cref{Rubin Hida}, $\frX_\theta$ is $\Lambda$-torsion; this is condition~(2). Denoting by $\chi$ the $p$-adic cyclotomic character. Then the Cartier dual $\Hom(\cO(\theta),F/\cO(1)) = F/\cO(\chi\theta^{-1})$ has finite $G_{K_\infty}$-invariants for any finite order character $\theta$. (This is equivalent to showing that $\chi\theta^{-1}|_{G_{K_\infty}} \neq \mathbf{1}$. Note that $\ker(\chi)=K_{\cyc}$, the cyclotomic $\bZ_p$-extension of $K$, is linearly disjoint from $K_{\cyc}$, so $\chi(G_{K_\infty})=\chi(G_{K_\infty}G_{K_{\cyc}})=\chi(G_K)$. Hence $\chi|_{G_{K_\infty}}\colon G_{K_\infty} \to \Z_p^\times$ has open image. But multiplying it by the character $\theta^{-1}$ with finite image does not change this property.)
This means that condition~(3) is satisfied. Condition~(4) is satisfied since one has $r_v = 1, r_{\ol{v}} = 0$ (from~\cref{H1ofMtheta}), $[K:\Q] = 2$, $d = 1$ (we have characters) and $\delta(K, F(\theta)) = 1$ (there is exactly one archimedean place up to equivalence, and it is complex, and $d = 1$).

Surjectivity for the $S$-imprimitive Selmer group immediately follows. \qedhere
\end{proof}

Consider the (finitely generated) $\Lambda$-modules
\begin{align*}
	\frX_\theta^S &\colonequals \H^1_{\cF_\nr^S}(K, M_\theta)^\vee,\\
	\frX_\theta &\colonequals \H^1_{\cF_\nr}(K, M_\theta)^\vee.
\end{align*}
We now determine $\lambda(\frX_\theta^S)$ in terms of $\H^1_{\cF_\nr^S}(K, M_\theta[\frp])$.

\begin{lemma}[relation between the $p$-torsion of the $S$-imprimitive Selmer group and the residual one] \label[lemma]{eq:move p torsion}
One has an isomorphism 
\[
    \H^1_{\cF_\nr^S}(K, M_\theta[\frp]) \iso \H^1_{\cF_\nr^S}(K, M_\theta)[\frp].
\]

\end{lemma} \label[lemma]{1.2.4}
Note that one has a surjection $\Sel^\Sigma(K, M_\theta[\frp]) \surj \Sel^\Sigma(K, M_\theta)[\frp]$ by~\cite[Lemma~I.5.4]{RubinEulerSystems}.
\begin{proof}
The long exact cohomology sequence of $0 \to M_\theta[\frp] \to M_\theta \to M_\theta \to 0$ (note that $M_\theta$ is $\pi$-divisible since $\Lambda^\vee$ is) yields a short exact sequence
\[
    0 \to \H^0(K^\Sigma/K, M_\theta)/\pi \to \H^1(K^\Sigma/K, M_\theta[\frp]) \to \H^1(K^\Sigma/K, M_\theta)[\frp] \to 0,
\]
with an analogous short exact sequence with $K^\Sigma/K$ replaced by $K_{\ol v}$. This allows us to first study the $p$-torsion of the Greenberg Selmer groups, which recovers~\cite[Lemma 1.2.4]{CGLS} for arbitrary characters (their assumption $\theta|_{G_{\ol v}} \neq 1$ is not essential). Indeed, by the snake lemma one has a commutative diagram with exact rows 
{\small
	\begin{equation}\tag{Gr,$\theta$}\label{move p Gr}
\begin{tikzcd}
    0 \ar[r]& \ker(r_{\ol v}^{\Gr,0}) \ar[r]\ar[d]& \H^1_{\cF_\Gr^S}(K, M_\theta[\frp]) \ar[r]\ar[d]& \H^1_{\cF_\Gr^S}(K, M_\theta)[\frp] \ar[d]& \\
   0 \ar[r]& \H^0(K^\Sigma/K, M_\theta)/\frp \ar[r]\ar[d,"r_{\ol v}^{\Gr,0}"]& \H^1(K^\Sigma/K, M_\theta[\frp]) \ar[r]\ar[d,"r_{\ol v}^{\Gr,1}"]& \H^1(K^\Sigma/K, M_\theta)[\frp] \ar[r]\ar[d]& 0\\
   0 \ar[r]& \H^0(K_{\ol v}, M_\theta)/\frp \ar[r]& \H^1(K_{\ol v}, M_\theta[\frp]) \ar[r]& \H^1(K_{\ol v}, M_\theta)[\frp] \ar[r]& 0
\end{tikzcd}
\end{equation}}
We have to study the kernel and the cokernel of \[r_{\ol v}^{\Gr,0}\colon \H^0(K^\Sigma/K, M_\theta)/\frp \to \H^0(K_{\ol v}, M_\theta)/\frp.\] But its domain and codomain are $0$: $\H^0(K^\Sigma/K, M_\theta) \iso \H^0(K^\Sigma_\infty/K_\infty, F/\cO(\theta))$ and $\H^0(K_{\ol v}, M_\theta) \iso \H^0(K_{\ol v,\infty}, F/\cO(\theta))$, both isomorphisms by Shapiro's lemma, which are both $\pi$-divisible: If $\theta = \mathbf{1}$, they are $F/\cO$; and if $\theta \ne \mathbf{1}$ but $\theta|_{G_{\ol v}} = \mathbf{1}$, they are $0$ and $F/\cO$. Otherwise $0 = \H^0(K^\Sigma_\infty/K_\infty, \F(\theta)) \isoto \H^0(K^\Sigma_\infty/K_\infty,F/\cO)[\frp]$ (and similar for the local $\H^0$) shows that they are torsion $p$-groups with trivial $\pi$-torsion, hence $0$. In particular, the middle row gives the identification $\H^1(K^\Sigma/K, M_\theta[\frp])=\H^1(K^\Sigma/K, M_\theta)[\frp]$.

We then study the imprimitive Selmer groups with unramified local conditions. There is a similar diagram 

{\small
	\begin{equation}\tag{nr,$\theta$}\begin{tikzcd}
		0 \ar[r]& \ker(r_{\ol v}^{\nr,0}) \ar[r]\ar[d]& \H^1_{\cF_{\nr}^S}(K, M_\theta[\frp]) \ar[r]\ar[d]& \H^1_{\cF_{\nr}^S}(K, M_\theta)[\frp] \ar[d]& \\
		0 \ar[r]& \H^0(K^\Sigma/K, M_\theta)/\frp \ar[r]\ar[d,"r_{\ol v}^{\nr,0}"]& \H^1(K^\Sigma/K, M_\theta[\frp]) \ar[r]\ar[d,"r_{\ol v}^{\nr,1}"]& \H^1(K^\Sigma/K, M_\theta)[\frp] \ar[r]\ar[d]& 0\\
		0 \ar[r]& \ker(f_\theta) \ar[r]& \frac{\H^1(K_{\ol v}, M_\theta[\frp])}{\ker(\res_{M_\theta[\frp]})} \ar[r,"f_\theta"]& \frac{\H^1(K_{\ol v}, M_\theta)}{\ker(\res_{M_\theta})}[\frp] 
	\end{tikzcd}\end{equation}}
where again we need to study the domain and codomain of $r_{\ol v}^{\nr,0}$. But the domain remains the same, so still $\ker(r_{\ol v}^{\nr,0}) = 0$. We claim that the codomain $\ker(f_\theta) = 0$ as well.

From here and onward, we use $\tilde{\omega}$ (resp.\ $\tilde{\mathbf{1}}$) to denote a character $\theta$ of $G_K$ whose restriction to $G_{\ol v}$ is $\omega$ (resp.\ $\mathbf{1}$) which may or may not be $\omega$ (resp.\ $\mathbf{1}$) on $G_K$.

When $\theta|_{G_{\ol v}}=\omega$, \cref{H1ofMtheta}\,(i) shows that $\ker(\res_{M_{\tilde{\omega}}})=0$. Actually, the same argument shows that $\ker(\res_{M_{\tilde{\omega}}[\frp]})=0$ as well, since both $M_{\tilde{\omega}}^{G_{\ol v}}$ and $M_{\tilde{\omega}}[\frp]^{G_{\ol v}}$ are $0$. Therefore, if $\theta|_{G_{\ol v}}=\omega$, the above two diagrams agree, so $\ker(f_{\tilde{\omega}})=\H^0(K_{\ol v}, M_{\tilde{\omega}})/\frp=0$.

When $\theta|_{G_{\ol v}}=\mathbf{1}$, however, we need to study the map $f_{\tilde{\mathbf{1}}}$ in the bottom row more carefully. Note that in general for $C$ a $\frp$-divisible submodule of $A$ we have an isomorphism
\begin{equation}\label{quotient p}
	A/C[\frp] \iso A[\frp]/C[\frp].
\end{equation}
Notice that from~\cref{H1ofMtheta}\,(i), $\ker(\res_{M_{\tilde{\mathbf{1}}}})=F/\cO$. A similar argument shows that for $M_{\tilde{\mathbf{1}}}[\frp]$, the exact sequence~\eqref{eq:rw} reads 
\[
0 \to \ker(\res_{M_{\tilde{\mathbf{1}}}[\frp]})^\vee \to \F[\Gamma/I_w] \to \F[\Gamma/I_w] \to \F \to 0,
\]
so $\ker(\res_{M_{\tilde{\mathbf{1}}}[\frp]})=\F$.

Note also that from the last row in the first diagram~\eqref{move p Gr}, since $\H^0(K_{\ol v}, M_{\tilde{\mathbf{1}}})/\frp=0$, there is an identification $\H^1(K_{\ol v}, M_{\tilde{\mathbf{1}}}[\frp])=\H^1(K_{\ol v}, M_{\tilde{\mathbf{1}}})[\frp]$. Therefore the domain of $f_{\tilde{\mathbf{1}}}$ is identified with $\frac{\H^1(K_{\ol v}, M_{\tilde{\mathbf{1}}})[\frp]}{\ker(\res_{M_{\tilde{\mathbf{1}}}[\frp]})}$. Let $A=\H^1(K_{\ol v}, M_{\tilde{\mathbf{1}}})$ and $C=\ker(\res_{M_{\tilde{\mathbf{1}}}})=F/\cO$. The map $f_{\tilde{\mathbf{1}}}$ then looks like
\[
A[\frp]/\F \xrightarrow{f_{\tilde{\mathbf{1}}}} A/(F/\cO)[\frp]=A[\frp]/(F/\cO)[\frp]=A[\frp]/\F,
\]
where the identification in the middle comes from~\eqref{quotient p}. However, $f_{\tilde{\mathbf{1}}}$ is induced by the natural map $ \H^1(K_{\ol v}, M_{\tilde{\mathbf{1}}}[\frp]) \isoto \H^1(K_{\ol v}, M_{\tilde{\mathbf{1}}})[\frp]:=A[\frp]$, so $f_{\tilde{\mathbf{1}}}$ must itself be an isomorphism on the quotients. Hence $\ker(f_{\tilde{\mathbf{1}}})=0$ as well.

Hence in both cases $\theta|_{G_{\ol v}}=\omega$ and $\mathbf{1}$, $\H^1_{\cF_\nr^S}(K, M_\theta[\frp]) \to \H^1_{\cF_\nr^S}(K, M_\theta)[\frp]$ is an isomorphism by the snake lemma applied to the second diagram.
\end{proof}

\begin{proposition}\label{Lambda invariants}
The $S$-imprimitive Selmer group $\frX_\theta^S$ is $\Lambda$-torsion with $\mu$-invariant $0$ and $\lambda$-invariant

\[
    \lambda(\frX_\theta^S) = \lambda(\frX_\theta) + \sum_{w \in \Sigma, w \nmid p}\lambda(\cP_w(\theta)).
\]
Moreover, $\H^1_{\cF_\nr^S}(K, M_\theta[\frp])$ is finite and if $\theta|_{G_{\ol v}}\ne\omega$, $\dim_\F\H^1_{\cF_\nr^S}(K, M_{\theta}[\frp])=\lambda(\frX_{\theta}^S)$.
\end{proposition}

\begin{proof}
When $\theta|_{G_{\ol v}}\ne\mathbf{1},\omega$, this is~\cite[Proposition 1.2.5]{CGLS}. Now we obtain from~\cref{remark on Selmer groups} that the restriction map defining $\H^1_{\cF_\nr}(K, M_\theta)$ using unramified local conditions is surjective, so the sequence
{\small
\begin{equation} \label{eq:Def unr}
    0 \to \H^1_{\cF_\nr}(K, M_\theta) \to \H^1(K^\Sigma/K, M_\theta) \to \prod_{w\in\Sigma,w\nmid p}\H^1(K_w,M_\theta) \times \H^1(I_{\ol v}, M_\theta)^{G_{\ol v}/I_{\ol v}} \to 0
\end{equation}}
is exact. From the definitions, this readily yields the exact sequence
\begin{equation} \label{eq:Gr to imp}
    0 \to \H^1_{\cF_\nr}(K, M_\theta) \to \H^1_{\cF_\nr^S}(K, M_\theta) \to \prod_{w\in S}\H^1(K_w,M_\theta) \to 0,
\end{equation}
which combined with~\cref{Rubin Hida} and~\cref{H1 Euler} gives the first part of the proposition.

Essentially the same proof in~\textit{loc.\ cit.} shows $\H^2(K^\Sigma/K,M_\theta)$ is 0. More precisely, we adopt the arguments from Greenberg in~\cite{Greenberg1989}, as we now explain.

First note that because  $\Gal(K^\Sigma/K)$ has $p$-cohomological dimension $2$, $\H^2(K^\Sigma/K,M_\theta)$ is cofree. Proposition 3 in \textit{op.\ cit.}\ yields the following equation of $\Lambda$-coranks:
\begin{equation*}
	\corank_\Lambda(\H^1(K^\Sigma/K,M_\theta))-\corank_\Lambda(\H^2(K^\Sigma/K,M_\theta))=1.
\end{equation*}
Since $\corank_\Lambda(\H^1(K^\Sigma/K,M_\theta))\leq 1$(it follows from the exact sequence~\eqref{eq:Def unr}, together with~\cref{Rubin Hida}, \cref{H1 Euler} and~\cref{H1ofMtheta}), this gives the cotorsionness of $\H^2(K^\Sigma/K,M_\theta)$, so $\H^2(K^\Sigma/K,M_\theta)=0$.

Then~\cite[Proposition 5]{Greenberg1989} shows that $\H^1(K^\Sigma/K,M_{\tilde{{\mathbf{1}}}})^\vee$ has no finite $\Lambda$-sub\-mo\-dules. From~\cref{H1ofMtheta}, we see that $P^S:=(\frac{\H^1(K_{\ol v},M_{\tilde{{\mathbf{1}}}})}{\ker(\res_{{M_{\tilde{\mathbf{1}}}}})})^\vee$ fits into an exact sequence\[
0\to P^S \to \Lambda \to \cO \to 0,\]
so $P^S$ is a free $\Lambda$-module of rank $1$. Then~\eqref{eq:Def unr} and~\eqref{eq:Gr to imp} readily yields \[
 \fX^S_{\tilde{\mathbf{1}}}\cong \frac{\H^1(K^\Sigma/K,M_{\tilde{\mathbf{1}}})^\vee}{P^S}
 \]
 as $\Lambda$-modules, and by~\cite[Lemma 2.6]{GV00} we conclude that $\fX^S_{\tilde{\mathbf{1}}}$ has no finite $\Lambda$-submodules.

 Since $\frX_{\tilde{\mathbf{1}}}^S$ is torsion with $\mu$-invariant $0$ by~\cref{Rubin Hida}, the finiteness of $\H^1_{\cF_\nr^S}(K, M_{\tilde{\mathbf{1}}})[\frp]$ (and hence of $\H^1_{\cF_\nr^S}(K, M_{\tilde{\mathbf{1}}}[\frp])$  by~\cref{eq:move p torsion}) follows from the structure theorem. It then follows that $\H^1_{\cF_\nr^S}(K, M_{\tilde{\mathbf{1}}})$ is divisible, and in particular
\begin{equation*}
\H^1_{\cF_\nr^S}(K, M_{\tilde{\mathbf{1}}})\isom (F/\cO)^{\lambda(\frX_{\tilde{\mathbf{1}}}^S)},
\end{equation*} which together with~\cref{eq:move p torsion} gives the formula for $\lambda$-invariant.

The same argument shows that $\H^1_{\cF_\nr^S}(K, M_{\tilde{\omega}}[\frp])$ is finite, except that $\fX^S_{\tilde{\omega}}$ might have some finite submodules (also from~\cref{H1ofMtheta}) so we don't obtain a simple formula about the $\lambda$-invariant. The computation will be postponed to~\cref{sec:selmer-groups-of-rhof}.
\end{proof}

\begin{corollary}\label[corollary]{[p]exact}
	\begin{enumerate}
		\item Assume $\theta|_{G_{\ol v}} \neq \omega$. Then $\H^2(K^\Sigma/K,M_\theta[\frp]) = 0$.
	
		\item When $\theta|_{G_{\ol v}}=\mathbf{1}$, one has an exact sequence
		\[
		0 \to \H^1_{\cF_{\nr}^S}(K, M_\mathbf{1}[\frp]) \to \H^1(K^\Sigma/K,M_\mathbf{1}[\frp]) \to \H^1(K_{\ol v}, M_\mathbf{1}[\frp])/\ker(\res_{M_{\tilde{\mathbf{1}}}[\frp]}) \to 0.
		\]
	\end{enumerate}
\end{corollary}

\begin{proof}
	(i) When $\theta|_{G_{\ol v}} \neq \mathbf{1},\omega$, this is proved in~\cite{CGLS}. Assume $\theta|_{G_{\ol v}}=\mathbf{1}$.
	
	The cohomology long exact sequence induced by multiplication by $p$ on $M_{\tilde{\mathbf{1}}}$ yields an isomorphism
	\begin{equation}\label{H1/p=H2}
		\frac{\H^1(K^\Sigma/K, M_{\tilde{\mathbf{1}}})}{p\H^1(K^\Sigma/K, M_{\tilde{\mathbf{1}}})}\cong \H^2(K^\Sigma/K, M_{\tilde{\mathbf{1}}}[\frp]).
	\end{equation}
	On the other hand, from the exactness of~\eqref{eq:Def unr} we deduce the exact sequence
	\begin{equation}\label{ImpSES1}
		0 \to \H^1_{\cF_\nr^S}(K,M_{\tilde{\mathbf{1}}}) \to \H^1(K^\Sigma/K, M_{\tilde{\mathbf{1}}}) \to \H^1(I_{\ol v}, M_{\tilde{\mathbf{1}}})^{G_{\ol v}/I_{\ol v}}  \to 0.
	\end{equation}
	Since $\H^1(I_{\ol v}, M_{\tilde{\mathbf{1}}})^{G_{\ol v}/I_{\ol v}}$ is a quotient object of $\H^1(K_{\ol v}, M_{\tilde{\mathbf{1}}})$, the dual of $\H^1(I_{\ol v}, M_{\tilde{\mathbf{1}}})^{G_{\ol v}/I_{\ol v}}$ is a subobject of $\H^1(K_{\ol v}, M_{\tilde{\mathbf{1}}})^\vee$. By~\cref{H1ofMtheta}\,(ii) the latter is free of $\Lambda$-rank 1, so the dual of $\H^1(I_{\ol v}, M_{\tilde{\mathbf{1}}})^{G_{\ol v}/I_{\ol v}}$ has no $p$-torsion, too. Since we showed in~\cref{Lambda invariants} that $\H^1_{\cF_\nr^S}(K, M_{\tilde{\mathbf{1}}})$ is divisible, it follows from~\eqref{ImpSES1} that $\H^1(K^\Sigma/K, M_{\tilde{\mathbf{1}}})^\vee$ has no $p$-torsion, and so
	\begin{equation*}
		\H^2(K^\Sigma/K, M_{\tilde{\mathbf{1}}}[\frp])\cong 0
	\end{equation*}
	by~\eqref{H1/p=H2}.

	(ii) By~\cref{remark on Selmer groups}\,(ii), the global to local map defining the imprimitive unramified Selmer group $\H^1_{\cF^S_{\nr}}(K,M_{\tilde{\mathbf{1}}})$ is surjective. By~\cref{Lambda invariants}, $\H^1_{\cF^S_{\nr}}(K, M_{\tilde{\mathbf{1}}})$ is divisible. Now with our~\cref{eq:move p torsion}, the proof of the second claim in Corollary 1.2.6 in \textit{op.\ cit.} applies almost \textit{verbatim} to our case. This finishes the proof.
\end{proof}

\subsection{Local cohomology groups of $\rho_f$}\label{sec:local-cohomology-groups-of-rhof} Recall that $f\in S_{k}(\Gamma_0(N))^{\new}$ is a newform of weight $k=2r$.
Let $T \colonequals T_f$ be the self-dual twist of the $\frp$-adic Galois representation of $f$, $V \colonequals V_f \colonequals T_f \otimes_\cO F$, and $W \colonequals W_f \colonequals V_f/T_f$.
Let $M_f$ be the $G_K$-module
\[
    M_f \colonequals T \otimes_{\cO} \Lambda^\vee.
\]
Note that one has an exact sequence
\[
    0 \to W_f \to M_f \stackrel{\cdot T}{\to} M_f \to 0.
\]

Recall that by $\tilde{\mathbf{1}}$ ({resp.\ $\tilde{\omega}$}) we mean a character $\theta$ of $G_K$ with $\theta|_{G_p}=\mathbf{1}$ ({resp.\ $\omega$}) which may or may not be $\mathbf{1}$ ({resp.\ $\omega$}) on $G_K$. Compared to~\cite{CGLS}, we then have two new cases to deal with:
\begin{enumerate}
\item	$0\to\F(\tilde{\mathbf{1}})\to \ol\rho_f\to \F(\tilde{\omega})\to 0$
	
\item	$0\to\F(\tilde{\omega})\to \ol\rho_f\to \F(\tilde{\mathbf{1}})\to 0$
\end{enumerate}
We will see in section~\ref{sec:selmer-groups-of-rhof} that the choice of the ordering of the characters appearing in the decomposition in $\ol\rho_f$ will not matter. For our convenience, assume we are in case~(ii). This assumption is in effect throughout section~\ref{sec:local-cohomology-groups-of-rhof} and section~\ref{sec:selmer-groups-of-rhof}.

\begin{proposition} \label[proposition]{omega first non-split}
	Let $\rho$ be an irreducible $2$-dimensional $p$-adic Galois representation whose residual representation $\ol\rho$ has semisimplification the direct sum of two $1$-dimensional characters. Then there is a lattice such that the associated Galois representation is a non-split extension of those two characters in a given ordering.
\end{proposition}
\begin{proof}
	See~\cite[Proposition~1.4]{Bellaiche}
\end{proof}

Note that $V_f$ is an irreducible $G_K$-module (this follows from the Tate conjecture for abelian varieties if $k = 2$ and in for newforms of general weight from~\cite[Theorem~2.3]{RibetNebentypus}), but becomes reducible when restricted to $G_{K_v}$ because it is crystalline with distinct Hodge--Tate weights and characteristic polynomial having a unit root~\cite[Theorem~8.3.6]{BrinonConrad}. Hence by isogeny invariance of the BSD Conjecture, we can assume
$\ol\rho_f|_{G_K}$ non-split with no trivial subrepresentation and that there is a short exact sequence \[0 \to \F(\wt\phi) \to \ol\rho_f|_{G_K} \to \F(\wt\psi) \to 0\] with $\phi = \wt\phi|_{G_{K_v}} = \omega$.

\begin{theorem} \label[theorem]{rhof crystalline}
	The $p$-adic Galois representation $\rho_f|_{G_{K_w}}$ is crystalline with Hodge--Tate weights $\{r,1-r\}$.
\end{theorem}
\begin{proof}
	This is because $p \nmid N$~\cite[Theorem~1.2.4\,(ii)]{Scholl1990} (note that the condition $p \geq k$ there can be removed by~\cite{Tsuji1999}). It has Hodge--Tate weights $\{r,1-r\}$ since we have Hodge--Tate weights $\{0,k-1\}$ before taking the self-dual twist by $1-r$~\cite[4.2.0]{Scholl1990}.
\end{proof}

We first compute some $\H^0$ cohomology groups.

\begin{proposition}\label[proposition]{splitGkinfty}
	Let $p > 2$ be a prime. Let $K/\Q$ be an imaginary quadratic field with $p = v\ol{v}$ split in $K$ and $K_\infty|K$ be the anticyclotomic $\Z_p$-extension.
	
	Assume the reduction $\ol \rho_f = V_f/T_f[\frp]$ of $V_f$ a $2$-dimensional residual Galois representation of $G_\Q$ which is a non-split extension \[0 \to \F(\phi) \to \ol\rho_f \to \F(\psi) \to 0\] with $\phi = \epsilon\omega^{k-1}$ and $\psi = \epsilon^{-1}$ and $\epsilon$ unramified at $p$. Then $\H^0(G_{K_\infty},\ol \rho_f) = 0$ (i.e., $\ol \rho_f$ remains non-split over $G_{K_{\infty}}$).
\end{proposition}
\begin{proof}
		Since $\ol\rho_f$ is non-split, it defines a nonzero extension class in \[\Ext_{\F[G_K]}^1(\F(\epsilon^{-1}), \F(\epsilon\omega^{k-1})) \iso \Ext_{\F[G_K]}^1(\F, \F(\epsilon^2\omega^{k-1})) \iso \H^1(K, \F(\epsilon^2\omega^{k-1}));\] here we use that twisting by $\epsilon$ induces an isomorphism on $\Ext^1$. Thus it suffices to show that $\res\colon \H^1(K,\F(\epsilon^2\omega^{k-1})) \to \H^1(K_\infty, \F(\epsilon^2\omega^{k-1}))$ is injective.
		
		Let $L = K(\epsilon^2)$. It is an extension of $K$ of degree dividing $(\#\F-1)/2$, hence coprime to $p$. It is also unramified at $p$. Since $K(\omega^{k-1})/K$ is totally ramified at $p$ as $p > 2$ and $K \neq \Q(\sqrt{-3}),\Q(\sqrt{-1})$, one has $K(\omega) \cap L = K$. Note that $K(\omega^{k-1}) \supsetneq K$ because otherwise $2 \mid p-1 = \ord(\omega) \mid (k-1)$, but $k$ is even.
		
		Now consider the commutative diagram of inflation-restriction sequences
		{\small
			\[
			\begin{tikzcd}
				0 \ar[r] & \H^1(K_\infty|K,\F(\epsilon^2\omega^{k-1})^{G_{K_\infty}}) \ar[r,"\inf"] \ar[d,hook] & \H^1(K,\F(\epsilon^2\omega^{k-1})) \ar[r,"\res"] \ar[d,hook] & \H^1(K_\infty,\F(\epsilon^2\omega^{k-1})) \ar[d]\\
				0 \ar[r] & \H^1(LK_\infty|L,\F(\epsilon^2\omega^{k-1})^{G_{LK_\infty}}) \ar[r,"\inf"]  & \H^1(L,\F(\epsilon^2\omega^{k-1})) \ar[r,"\res"] & \H^1(LK_\infty,\F(\epsilon^2\omega^{k-1}))
			\end{tikzcd}
			\]}
		The middle vertical arrow is injective by restriction-corestriction because $p \nmid [L:K]$. The lower left group is $0$ because $\F(\epsilon^2\omega^{k-1})^{G_{LK_\infty}} = \F(\omega^{k-1})^{G_{LK_\infty}} = 0$ since no non-trivial $p$-th root of unity is contained in $K_\infty$ because it is a $\Z_p$-extension of $K$ and $\mu_p(K) = \{1\}$ since $p > 2$ and $K \neq \Q(\sqrt{-3}), \Q(\sqrt{-1})$; note that $\omega^{k-1} \neq \mathbf{1}$. Hence the lower restriction morphism is injective, hence by commutativity and exactness of the diagram the top restriction is injective.
\end{proof}

For simplicity, in the local cohomology groups, we write the characters as $\mathbf{1}$ and $\omega$ instead of $\tilde{\mathbf{1}}$ and $\tilde{\omega}$ since we do not care about their global behavior.

\begin{proposition} \label[proposition]{H0 of V/T}
	Let $w \mid p$.	Assume the Galois representation $\ol\rho_{f}$ does not have a trivial subrepresentation. Assume that for $w \mid p$, $\ol\rho_f|_{G_{K_w}}$ is an extension of the trivial representation $\F$ by $\F(\omega)$. 
	\begin{enumerate}[(i)]
		\item $\H^0(K, V_f/T_f) = 0$.
		
		\item $\H^0(K_w, V_f/T_f)$ is finite and cyclic, i.e., isomorphic to $\cO/\frp^a$ with $a \geq 0$.

		\item $\H^0(K_w, M_f)$ is finite and cyclic, i.e., isomorphic to $\cO/\frp^b$ with $b  \geq 0$.
	\end{enumerate}
\end{proposition}
\begin{proof}
	\begin{enumerate}[(i)]
		\item Follows by assumption on $T_f$ by~\cref{omega first non-split}.
		
		\item Note that we have $V_f/T_f=M_f[T]$, hence $\H^0(?,V_f/T_f) \subseteq \H^0(?_\infty,V_f/T_f)[T]$ for $? \in \{K,K_w\}$, so (ii) follows from (iii).

		\item Note that $\rho_f|_{G_{K_w}}$ is crystalline with Hodge--Tate weights $\{r,1-r\}$ by~\cref{rhof crystalline}.
		By~\cite[Theorem~8.3.6]{BrinonConrad}, $\rho_f|_{G_w}$ is reducible
		\[
			0 \to F(\sigma) \to \rho_f|_{G_w} \to F(\tau) \to 0
		\]
		with characters $\sigma,\tau$ having Hodge--Tate weights $\{r,1-r\}$. By~\cite[Proposition 8.3.4]{BrinonConrad}, $\{\sigma,\tau\} = \{\epsilon\chi^r, \epsilon^{-1}\chi^{1-r}\}$ with $\epsilon$ an unramified character of $G_{K_w}$.	But for $m \neq 0$, $\epsilon\chi^m$ is non-trivial on $G_{K_{\infty,w}}$ by~\cref{local behavior of anticyclotomic Zp extension at p}\,(ii). Hence $\H^0(K_{\infty,w}, F(\epsilon\chi^m)) = 0$, so $\H^0(K_{\infty,w}, F/\cO(\epsilon\chi^m)) \subseteq F/\cO$ is finite and cyclic whenever $r>1$. (Note that $\H^0(K_{\infty,w}, F(\epsilon\chi^m)) = 0$ implies $\H^0(K_{\infty,w}, F/\cO(\epsilon\chi^m))$ is finite by~\cite[Lemma~2.7]{KobayashiOta}). Hence in order to show $\H^0(K_{\infty,w},V_f/T_f)$ is cyclic, it remains to show that for one of $? = \phi,\psi$, the invariants $\H^0(K_{\infty,w}, F/\cO(?)) = 0$.
		
		Note that at least one $\H^0(K_w, \F(\phi)), \H^0(K_w,\F(\psi))$ is $0$ (since one of them is $\omega$ when restricted to $G_w$). Since $\ord(\omega|_{G_{K_\infty}}) = p-1$ (since $p-1$ is coprime to $p$), also at least one $\H^0(K_{\infty,w}, F/\cO(\phi)), \H^0(K_{\infty,w}, F/\cO(\psi))$ is $0$. Hence $\H^0(K_{\infty,w}, V_f/T_f) \subseteq F/\cO$ and the finiteness result in the above paragraph shows that we cannot have equality.
		
	Now if $r=1$, $\rho_f$ is the $\frp$-adic Galois representation attached to an abelian variety $A/K_w$ with good reduction. Since $A/K_w$ has good reduction, its N\'eron model $\sA/\cO_{K_w}$ is an abelian scheme. It satisfies the N\'eron mapping property $\sA(\cO_{K_w}) \isoto A(K_w)$. Since $\sA$ is an abelian scheme, $\sA[p^n]$ is a finite flat group scheme over $\cO_{K_w}$ for all $n \geq 1$. Now for any finite flat group scheme $G$ over $\cO_{K_{n,w}}$, one has an exact sequence
		\[
			0 = G^0(\cO_{K_{n,w}}) \to G(\cO_{K_{n,w}}) \to G^{\et}(\cO_{K_{n,w}})
		\]
		coming from the connected-\'etale sequence. Since $\cO_{K_{n,w}}$ is a Henselian local ring, $G^{\et}(\cO_{K_{n,w}}) \iso G^{\et}(\bF_n)$ for the finite residue field $\F_n$ of $\cO_{K_{n,w}}$. Note that since $K_{\infty,w}/K_w$ is infinitely ramified, the $G^{\et}(\cO_{K_{n,w}})$ will eventually stabilize for $n \to \infty$ if $G = \sA[p^n]$ since an abelian variety has only finitely many points over a finite field. Hence $A[p^\infty](K_{\infty,w}) = \H^0(K_w, M_f)$ is finite.
		
		For future use, we mention that the same argument gives the finiteness of $\H^0(K_w, \Fil^-(M_f))$ if one replaces $G$ by $G^{\et}$ throughout.  \qedhere			
	\end{enumerate}
\end{proof}

\begin{lemma}\label[lemma]{corankH1Mf}
Let $w \mid p$ in $K$. Then $\H^1(K_w, M_f)$ has $\Lambda$-corank $2$.
\end{lemma}

\begin{proof}
Let $X \colonequals \H^1(K_w,M_f)^\vee$. Consider the exact sequences
\begin{equation}\label{Mf exact}
	0\to V_f/T_f\to M_f\xrightarrow{\text{$\cdot T$ }} M_f\to 0
\end{equation}
and
\begin{equation}\label{V/T exact}
	0\to \ol\rho_f\to V_f/T_f\xrightarrow{\text{$\cdot \pi$ }} V_f/T_f \to 0.
\end{equation}
We aim to study the $\cO$-rank of $X/T$ and mimic our argument in~\cref{H1ofMtheta}\,(ii). We will show that $X[T]=0$ and the $\cO$-rank is $2$, then the $\Lambda$-rank of $X$ is $2$ by the topological Nakayama lemma as in~\cref{Iwasawa module free}. 
	
We would like to study the following short exact sequence coming from~\eqref{Mf exact} \begin{equation}\label{V/TtoMf}
	0 \to \H^0(K_w,M_f)/T	\to \H^1(K_w,V_f/T_f) \to X^\vee[T] \to 0,
\end{equation}  
where the first term is finite by~\cref{H0 of V/T}\,(i).
	
By~\eqref{V/T exact}, there is an exact sequence 
\begin{equation}\label{H1ofV/T/p}
	0\to \H^1(K_w,V_f/T_f)/\frp \to \H^2(K_w,\ol\rho_f)\to \H^2(K_w, V_f/T_f)[\frp]\to 0.
\end{equation}
We claim that the last term in the above sequence is $0$. Indeed, $\H^2(K_w,V_f/T_f)=0$ because it's dual to $\H^0(K_w,T_f)$ which must vanish because $\H^0(K_w,V_f/T_f)$ is finite by~\cref{H0 of V/T}\,(ii).
Note that there is a sequence\[
0\to \H^0(K_w,\F(\omega))=0 \to \H^0(K_w,\ol\rho_f) \to \H^0(K_w,\F(\mathbf{1}))=\F,
\]
so $\H^0(K_w,\ol\rho_f)$ has $\F$-dimension $0$ or $1$. From local duality and self-duality of $\ol\rho_f$, $\H^2(K_w,\ol\rho_f)=\H^0(K_w,\ol\rho_f)^\vee$ has the same dimension. Now we consider two cases.

\textbf{Case I:}
$\dim_\F(\H^2(K_w,\ol\rho_f))=\dim_\F(\H^0(K_w,\ol\rho_f))=0$. 

In this case, from the local Euler characteristic formula~\cite[Theorem~7.3.1]{NSW2.3}, $\H^1(K_w,\ol\rho_f)$ has $\F$-dimension $2$. By~\eqref{V/T exact}, there is an identification
\begin{equation}\label{H1ofV/T[p]}
	\H^1(K_w,V_f/T_f)[\frp] \iso \H^1(K_w,\ol\rho_f)/(\H^0(K_w,V_f/T_f)/\frp).
\end{equation} Since in this case $\H^0(K_w,V_f/T_f)=0$,  from~\eqref{H1ofV/T[p]} we see that $\H^1(K_w,V_f/T_f)[\frp]$ has $\F$-dimension $2$. Therefore $\H^1(K_w,V_f/T_f)/\frp=0$, so $\H^1(K_w,V_f/T_f)$ is $\cO$-cofree of rank $2$. Hence $X/T$ is also $\cO$-free of rank $2$.

\textbf{Case II:} $\dim_\F(\H^2(K_w,\ol\rho_f))=\dim_\F(\H^0(K_w,\ol\rho_f))=1$. 

From~\eqref{H1ofV/T/p}, we see that $\H^1(K_w,V_f/T_f)/\frp$ also has $\F$-dimension $1$. 
Again by the local Euler characteristic formula, $\H^1(K_w,\ol\rho_f)$ has $\F$-dimension $4$. Since in this case $\H^0(K_w,V_f/T_f)$ is non-trivial, we see from~\cref{H0 of V/T}\,(ii) that $\H^0(K_w,V_f/T_f)/\frp=\F$. So by~\eqref{H1ofV/T[p]}, $\H^1(K_w,V_f/T_f)[\frp]$ has $\F$-dimension $3$. Hence $\rk_\cO(X/T)=\corank_\cO(X^\vee[T])=\corank_\cO\H^1(K_w,V_f/T_f)=2$.

Now we show that $X[T]=0$. Indeed, the same arguments in~\cref{H1ofMtheta} Claim~0 shows $\H^2(K_w,M_f)=0$. We therefore have an identification $X^\vee/T=\H^1(K_w,M_f)/T\isom\H^2(K_w,V_f/T_f)=0$. Now by~\cref{Iwasawa module free} ((i) for Case I and (ii) for Case II), $X$ has $\Lambda$-rank $2$.
	\end{proof}

Similarly as in~\cref{H1ofMtheta}, we define the restriction map for $M_f$ by\[
\res_{M_f}\colon \H^1(K_w,M_f) \to \H^1(I_w,M_f)^{G_w/I_w}.\] Define $\res_{M_f[\frp]}$ by replacing $M_f$ with $M_f[\frp]$ in the above definition.

\begin{lemma}\label[lemma]{kerresf}
	$\ker(\res_{M_f})$ is finite and cyclic. Moreover, when $\H^0(K_w,M_f)$ is non-trivial, $\ker(\res_{M_f[\frp]})=\F$ and when $\H^0(K_w,M_f)=0$, $\ker(\res_{M_f[\frp]})=0$. 
\end{lemma}
\begin{proof}
	Similar to~\eqref{eq:rw}, there are two exact sequences for $M_f$ and $M_f[\frp]$ respectively:
	\begin{equation} \label{eq:rf}
		0 \to \ker(\res_{M_f})^\vee \to (M_f^\vee)_{I_w} \to (M_f^\vee)_{I_w} \to M_f^\vee/T \to 0.
	\end{equation}
and
\begin{equation} \label{eq:rfp}
	0 \to \ker(\res_{M_f[\frp]})^\vee \to (M_f[\frp]^\vee)_{I_w} \to (M_f[\frp]^\vee)_{I_w} \to M_f[\frp]^\vee/T \to 0.
\end{equation}
Note that the last term in~\eqref{eq:rf} is nothing but the dual of $\H^0(K_w,M_f)$, which is finite by~\cref{H0 of V/T}\,(iii). Therefore $\H^0(K_w,M_f)^\vee=\H^0(K_w,M_f)=\cO/\frp^n$ for some $n \geq 0$. 

Note that $(M_f^\vee)_{I_w}$ is $\Lambda$-torsion because $I_w$ is non-trivial by~\cref{local behavior of anticyclotomic Zp extension at p}(iii), so $\ker(\res_{M_f})$ is also finite because characteristic ideals are multiplicative in exact sequences of torsion $\Lambda$-modules. Moreover, if $\H^0(K_w,M_f)=0$, by Noetherian property of $\Lambda$ the $\Lambda$-module surjection \[(M_f^\vee)_{I_w} \to (M_f^\vee)_{I_w}\] must be an isomorphism so $\ker(\res_{M_f})=0$. Before proving it's always cyclic, we first prove the second half of this proposition.
  
First assume $\H^0(K_w,M_f)$ is non-trivial and hence cyclic, then $M_f[\frp]^\vee/T=\H^0(K_w,M_f[\frp])^\vee=(\H^0(K_w,M_f)[\frp])^\vee=\F$. Since $(M_f[\frp]^\vee)_{I_w}=(\H^0(I_w,M_f[\frp]))^\vee=(\H^0(I_{w,\infty},\ol \rho_f))^\vee$ is finite (the second equality follows from Shapiro's lemma), it follows from sequence~\eqref{eq:rfp} that $\ker(\res_{M_f[\frp]})=\F$.

Now if $\H^0(K_w,M_f)=0$, then $M_f[\frp]^\vee/T=0$ so again from sequence~\eqref{eq:rfp} we see that $\ker(\res_{M_f[\frp]})=0$.

Finally, we show $\ker(\res_{M_f})[\frp]\subseteq \F$, from which it follows immediate that $\ker(\res_{M_f})$ is cyclic because we have showed it's finite.

Consider the diagram{\small
	\begin{equation*}\begin{tikzcd}
			0\ar[r] &\ker(\res)\ar[r]\ar[d]& \ker(\res_{M_f[\frp]}) \ar[r]\ar[d]& \ker(\res_{M_f})[\frp]\ar[d] & \\
			0 \ar[r]& \H^0(K_w, M_f)/\frp \ar[r]\ar[d,"\res"]& \H^1(K_w,M_f[\frp]) \ar[r]\ar[d]& \H^1(K_w, M_f)[\frp] \ar[r]\ar[d]& 0 \\
			0 \ar[r]& \H^0(I_w,M_f)/\frp \ar[r]\ar[d]& \H^1(I_w,M_f[\frp]) \ar[r]& \H^1(I_w,M_f)[\frp]\ar[r]& 0 \\
			&\coker(\res)&&&
\end{tikzcd}\end{equation*}}

When $\H^0(K_w,M_f)=0$, we have seen that $\ker(\res_{M_f})=0$.

Now assume $\H^0(K_w,M_f)$ is non-trivial so $\H^0(K_w,M_f)/\frp=\F$. In this case also $\ker(\res_{M_f[\frp]})=\F$. Since again $\H^0(I_w,M_f)[\frp]=\H^0(I_w,M_f[\frp])=\H^0(I_{w,\infty},\ol\rho_f)$, we have an exact sequence\[ 0\to \H^0(I_{w,\infty},\F(\omega))=0  \to\H^0(I_w,M_f[\frp])\to \H^0(I_{w,\infty},\F(\mathbf{1}))=\F.
\]
Here $\H^0(I_{w,\infty},\F(\omega))=0$ because $\omega$ is ramified and $I_w/I_{w,\infty}$ is a pro-$p$ group while the trivializing extension of $\omega$ has order dividing $p-1$. Hence $\H^0(I_w,M_f)[\frp]\subseteq\F$. It follows that $\H^0(I_w, M_f)\subseteq F/\cO$ and $\H^0(I_w,M_f)/\frp\subseteq \F$.  

Now if $\H^0(I_w,M_f)/\frp=0$, from the above diagram we see that $\coker(\res)=0$ and the snake lemma yields a surjection\[
\ker(\res_{M_f[\frp]})=\F\to\ker(\res_{M_f})[\frp],\] so $\ker(\res_{M_f})\subseteq \F$. 

If $\H^0(I_w,M_f)/\frp=\F$, then $\ker(\res)=\coker(\res)\subseteq\F$. If both are $0$, $\ker(\res_{M_f})[\frp]=\ker(\res_{M_f[\frp]})=\F$. If both are $\F$, then the snake lemma yields an injection\[
\ker(\res_{M_f})[\frp]\to\coker(\res)=\F,\] so again $\ker(\res_{M_f})[\frp]\subseteq \F$.
\end{proof}

\subsection{Selmer groups of $\rho_f$}\label{sec:selmer-groups-of-rhof}

Similarly as for that of $M_\theta$, we define an \emph{unramified Selmer group} for $f$ by
{\footnotesize\[
   \H^1_{\cF_{\nr}}(K,M_f)= \ker\Big\{\H^1(K^\Sigma/K, M_f) \xrightarrow{\res_{\nr,f}} \prod_{w\in\Sigma, w \nmid p}\H^1(I_w,M_f)^{G_w/I_w} \times \H^1(I_{\ol v}, M_f)^{G_{\ol v}/I_{\ol v}}\Big\},
\]}
and an \emph{$S$-imprimitive unramified Selmer group}, where $S=\Sigma
-\{v,\ol v,\infty\}$, by
\[
   \H^1_{\cF_{\nr}^S}(K,M_f)= \ker\Big\{\H^1(K^\Sigma/K, M_f) \xrightarrow{\res^S_{\nr,f}} \H^1(I_{\ol v}, M_f)^{G_{\ol v}/I_{\ol v}}\Big\}.
\]

The residual Selmer groups $\H^1_{\cF_{\nr}}(K,M_f[\frp])$ and 
$\H^1_{\cF_{\nr}^S}(K,M_f[\frp])$ are defined in the same manner. 

One can similarly define the Greenberg Selmer groups for $f$ by imposing triviality at $\ol v$. 

By Ribet's Lemma~\cite{Bellaiche}, from the decomposition $\ol\rho_f^\ss \isom \bF(\phi) \oplus \bF(\psi)$ one has an exact sequence
\[
0\to \bF(\phi) \to \ol\rho_f \to \bF(\psi) \to 0,
\]
which further induces and exact sequence of $\Gal(K^\Sigma/K)$-modules
\[
    0 \to M_\phi[\frp] \to M_f[\frp] \to M_\psi[\frp] \to 0,
\]
where $\phi \cdot \psi=\omega$. The cases left out in~\cite{CGLS} are when one of $\phi|_{G_p},\psi|_{G_p} $ is trivial. Hence we assume this is the case.

Note that by Perrin-Riou's formula in~\cite[Proposition~2.9]{KobayashiOta}, the order of $\phi$ and $\psi$ does not matter: Conjecture~1.2\,(1) in \emph{op.\ cit.} is true (the Selmer group is $\Lambda$-cotorsion), so Conjecture~1.2\,(2) (the anticyclotomic Main Conjecture for our modular forms)\footnote{see also~\cref{IMCwithtorsion}.} is independent of the lattice $T_f$. Therefore we are free to assume that $\phi|_{G_p}=\omega$ and $\psi|_{G_p}=\mathbf{1}$.  We then write the global character $\phi$ (resp.\ $\psi$) as $\tilde{\omega}$ (resp.\ $\tilde{\mathbf{1}}$). We remark here that the BSD Conjecture is also independent of the lattice (see for example~\cite[I.7]{milne2006}).

From now on, by essentially choosing a lattice $T_f$, we fix the following ordering of the characters appearing in the above sequence
\begin{equation}\label{char to f}
0 \to M_{\tilde{\omega}}[\frp] \to M_f[\frp] \to M_{\tilde{\mathbf{1}}}[\frp] \to 0.
\end{equation}

We remark here that the second character $\tilde{\mathbf{1}}$ may be non-trivial or trivial on $G_K$. Therefore,  $\H^0(K^\Sigma/K, M_{\tilde{\mathbf{1}}}[\frp])$ may or may not be $0$. Since the semisimplification $\ol\rho_f^\ss$ is independent from the choice of a lattice $T_f$, we must study both cases. However, from~\cite[Theorem 5.2]{BP19}, we can always choose a non-split extension of $\ol\rho_f^\ss$ so that $\bF(\mathbf{1})$ (the global trivial character) is never a subrepresentation, so we could always assume that $\H^0(K^\Sigma/K, M_f)=0$ (equivalently, $\H^0(K,\ol\rho_f)=0$), i.e., there is no global torsion for $f$ (it should be mentioned that their assumption (Good Eisen) is not necessary for our use). We will do so in the rest of this paper. We also include some partial results without assuming global non-torsionness in the appendix for interested readers.

Let $\fX_f$ and $\fX_f^S$ denote Pontryagin dual of the primitive and $S$-imprimitive unramified Selmer groups of $f$. 
This section is devoted to proving the following relation between the $\lambda$-invariants of the imprimitive Selmer groups under the assumption that there is no global torsion:
\begin{theorem}\label[theorem]{imprimlambda}
	If $\phi|_{G_p}=\omega$ (so $\psi|_{G_p}=\mathbf{1}$) and $\phi|_{G_K}\ne\omega$ (so $\psi|_{G_K}\ne\mathbf{1}$). Then
	\begin{equation*}
			\lambda(\frX^S_f) = \lambda(\frX^S_\phi) + \lambda(\frX^S_\psi).
	\end{equation*}
If $\phi|_{G_K}=\omega$ and $\psi|_{G_K}=\mathbf{1}$, 
\begin{equation*}
	\lambda(\frX^S_f) + 1 = \lambda(\frX^S_\phi) + \lambda(\frX^S_\psi).
\end{equation*}
Furthermore, $\fX^S_f$ and $\fX_f$ are both $\Lambda$-torsion with $\mu$-invariants $0$.
\end{theorem}

\begin{proof}[Proof of the theorem:]
First notice that the sequence~\eqref{char to f} yields a commutative diagram	{\small\[\begin{tikzcd}
		  \H^1_{\cF_\Gr^S}(K, M_{\tilde{\omega}}[\frp]) \ar[r]\ar[d]& \H^1_{\cF_\Gr^S}(K, M_f[\frp])\ar[r]\ar[d]&\H^1_{\cF_\Gr^S}(K, M_{\tilde{\mathbf{1}}}[\frp]) \ar[d]  \\
		  \H^1(K^\Sigma/K, M_{\tilde{\omega}}[\frp]) \ar[r]\ar[d]& \H^1(K^\Sigma/K, M_f[\frp]) \ar[r]\ar[d]& \H^1(K^\Sigma/K, M_{\tilde{\mathbf{1}}}[\frp]) \ar[d]\\
		 \H^1(K_{\ol v}, M_{\tilde{\omega}}[\frp])\ar[r]& \H^1(K_{\ol v}, M_f[\frp]) \ar[r]&\H^1(K_{\ol v}, M_{\tilde{\mathbf{1}}}[\frp])
	\end{tikzcd}\]}
with exact rows. It then follows from~\cref{Lambda invariants} that $\H^1_{\cF_\Gr^S}(K, M_{\tilde{\mathbf{1}}}[\frp])$ and $\H^1_{\cF_\Gr^S}(K, M_{\tilde{\omega}}[\frp])$ are finite (since Greenberg's Selmer groups are smaller than unramified ones). Now  $\H^1_{\cF_\Gr^S}(K, M_f[\frp])$ is finite and there is an exact sequence\[
0\to\H^1_{\cF_\Gr^S}(K, M_f[\frp])\to\H^1_{\cF_\ur^S}(K, M_f[\frp])\to \ker(\res_{M_f[\frp]})(\subset\F),
\]
from which it follows that $\H^1_{\cF_\ur^S}(K, M_f[\frp])$ is finite.
	
	We have the following 2 diagrams for $f$ with exact rows and columns similar to those of the characters in Section~\ref{ssec:Selmer groups of characters} and~\cref{eq:move p torsion}:
	
	{\small
		\begin{equation}\tag{Gr,$f$}\label{move p Gr f}
			\begin{tikzcd}
				0 \ar[r]& \ker(r_{\ol v}^{\Gr,0}) \ar[r]\ar[d]& \H^1_{\cF_\Gr^S}(K, M_f[\frp]) \ar[r]\ar[d]& \H^1_{\cF_\Gr^S}(K, M_f)[\frp] \ar[d]& \\
				0 \ar[r]& \H^0(K^\Sigma/K, M_f)/\frp \ar[r]\ar[d,"r_{\ol v}^{\Gr,0}"]& \H^1(K^\Sigma/K, M_f[\frp]) \ar[r]\ar[d,"r_{\ol v}^{\Gr,1}"]& \H^1(K^\Sigma/K, M_f)[\frp] \ar[r]\ar[d]& 0\\
				0 \ar[r]& \H^0(K_{\ol v}, M_f)/\frp \ar[r]& \H^1(K_{\ol v}, M_f[\frp]) \ar[r]& \H^1(K_{\ol v}, M_f)[\frp] \ar[r]& 0
			\end{tikzcd}
	\end{equation}}
	
	{\small
		\begin{equation}\tag{nr,$f$}\begin{tikzcd}
				0 \ar[r]& \ker(r_{\ol v}^{\nr,0}) \ar[r]\ar[d]& \H^1_{\cF_{\nr}^S}(K, M_f[\frp]) \ar[r]\ar[d]& \H^1_{\cF_{\nr}^S}(K, M_f)[\frp] \ar[d]& \\
				0 \ar[r]& \H^0(K^\Sigma/K, M_f)/\frp \ar[r]\ar[d,"r_{\ol v}^{\nr,0}"]& \H^1(K^\Sigma/K, M_f[\frp]) \ar[r]\ar[d,"r_{\ol v}^{\nr,1}"]& \H^1(K^\Sigma/K, M_f)[\frp] \ar[r]\ar[d]& 0\\
				0 \ar[r]& \ker(f) \ar[r]& \frac{\H^1(K_{\ol v}, M_f[\frp])}{\ker(\res_{M_f[\frp]})} \ar[r,"f"]& \frac{\H^1(K_{\ol v}, M_f)}{\ker(\res_{M_f})}[\frp] 
	\end{tikzcd}\end{equation}}

Since $\H^0(K^\Sigma/K, M_f)=0$ and as we shall see, $\ker(f)\subseteq \F$, $\H^1_{\cF_{\nr}^S}(K, M_f)[\frp]$ must also be finite. Then from the structure theorem it follows that $\fX^S_f$ (and hence $\fX_f$) are $\Lambda$-torsion with $\mu$-invariants $0$. This proves the last part of the theorem.

\begin{remark}\label[remark]{surj for Self}
	As in~\cref{remark on Selmer groups}, the global-to-local map defining the Selmer groups (both the Greenberg's one and the unramified one, both primitive and imprimitive) for $f$ is also surjective in our case, due to~\cite[Proposition~A.2]{PollackWeston2011}.
	
	(1) is already proved.
	
	(2) is a proved above (both primitive and imprimitive).
	
	To see (3), since $\rho_f$ is self-dual, the cohomology group we consider is identified with $\H^0(K_\infty,V_f/T_f)$. By the comments before~\cite[Lemma~3.3.3]{CGLS}, we have $\rho_f(G_{K_\infty}) = \rho_f(G_\Q)$, so we need to show that $\H^0(\Q,V_f/T_f)$ is finite. Assume this is infinite, then $V_f/T_f \isom (F/\cO)^2$ has at least a $G_\Q$-stable copy of $F/\cO$. Since we could recover $T_f = \varprojlim V_f/T_f[p^n]$ from $V_f/T_f$, this implies $T_f$ has a $G_\Q$-stable copy of $\cO$. However, this contradicts the irreducibility of $T_f$ as a $G_\Q$ module, a result of Ribet in~\cite{RibetNebentypus}.
	
	(4) holds by~\cref{corankH1Mf} (for Greenberg's Selmer group) and~\cref{kerresf} (for the unramified Selmer group).
\end{remark}

To finish the proof of the theorem, it remains to compare the $\lambda$-invariants. Recall that by~\cref{kerresf}, $\ker(\res_{M_f})$ is finite. We now have $3$ different cases to consider: 
\begin{enumerate}
		\item[I] $\H^0(K_w,M_f)$ is non-trivial and $\ker(\res_{M_f})=0$
		\item[II] Both $\H^0(K_w,M_f)$ and $\ker(\res_{M_f})$ are non-trivial 
		\item[III] $\H^0(K_w,M_f)=0$ (so also $\ker(\res_{M_f})=\ker(\res_{M_f[\frp]})=0$)
\end{enumerate}For simplicity, we first assume $\ker(\res_{M_f})=0$. In this case, the map $f$ is surjective.
	
\textbf{Case I:}	
	
Note that $\H^0(K_{\ol v},M_f)$ is cyclic and non-trivial, so by~\cref{H0 of V/T}\,(iii) one has $\H^0(K_{\ol v},M_f)/\frp=\F$. Therefore $\ker(f)=0$. Also $\H^0(K^\Sigma/K, M_f)/\frp=0$ by~\cref{splitGkinfty}.

To apply snake lemma, we add a fourth row to the unramified diagrams and abbreviate them as follows (here $?=\tilde{\mathbf{1}},f$ or $\tilde{\omega}$):
{\small
	\begin{equation}\tag{nr,$?$}\label{M?tolambda}\begin{tikzcd}
			0 \ar[r]& 0 \ar[r]\ar[d]& \Sel(M_?[\frp]) \ar[r]\ar[d]& \Sel(M_?)[\frp] \ar[d]& \\
			0 \ar[r]& \H^0(K^\Sigma/K, M_?)/\frp=0 \ar[r]\ar[d,"r_?^0"]& \H^1(K^\Sigma/K, M_?[\frp]) \ar[r]\ar[d,"r_{?[\frp]}"]& \H^1(K^\Sigma/K, M_?)[\frp] \ar[r]\ar[d,"r_{?}\text{[$\frp$]}"]& 0\\
			0 \ar[r]& \ker(f_?)=0 \ar[r]\ar[d]& \frac{\H^1(K_{\ol v}, M_?[\frp])}{\ker(\res_{M_?[\frp]})} \ar[r,"f_?"]\ar[d]& \frac{\H^1(K_{\ol v}, M_?)}{\ker(\res_{M_?})}[\frp]\ar[r,dashed]\ar[d]&0 \\
			& 0 \ar[r]& \coker_{?[\frp]} \ar[r]& \coker_?[\frp] \ar[r,dashed] &0
\end{tikzcd}\end{equation}}
Here dashed lines are used because in the case $\ker(\res_{M_f})\ne 0$ they may not be exact. We also remark that the cokernels appearing in the above diagrams are finite since all groups have finite $\cO$-rank.

Note that by the structure theorem of finitely generated $\Lambda$-modules (since $\mu(\Sel(M_?)) = 0$) the $\lambda$-invariants of $\Sel(M_?)$ can be computed as
\[
	\lambda(\frX^S_?)=\dim_\F \Sel(M_?)[\frp]-\dim_\F \Sel(M_?)/\frp,
\]
whereas $\coker_?[\frp]=\im(\Sel(M_?)[\frp] \to \Sel(M_?)/\frp)$ by~\cref{remark on Selmer groups} and~\cref{surj for Self}, so in particular
{\small\begin{align*}
	&\dim_\F \Sel(M_?)/\frp \\=&\dim_\F \ker(\Sel(M_?)/\frp \to \H^1(K^\Sigma/K,M_?)/\frp)+\dim_\F \im\big(\Sel(M_?)/\frp \to \H^1(K^\Sigma/K,M_?)/\frp\big)\\=&\dim_\F \coker_?[\frp]+\dim_\F \ker\Big(\H^1(K^\Sigma/K,M_?)/\frp\to \frac{\H^1(K_{\ol v}, M_?)}{\ker(\res_{M_?})}/\frp\Big).
	\end{align*}}
Denoting the last kernel by $K_?$, we get  \[\lambda(\frX^S_?)=\dim_\F \Sel(M_?)[\frp]-\dim_\F \coker_?[\frp]-\dim_\F K_?.\]
	
Now by~\cref{[p]exact}, the last two terms are $0$ when $?=\tilde{\mathbf{1}}$.
Hence, if we use $d(\cdot)$ to represent the $\F$-dimension of an $\F$-vector space, we get the following 
\begin{lemma}\label[lemma]{Seltolambda} We have
\begin{enumerate}
	\item $d(\Sel(M_{\tilde{\mathbf{1}}}[\frp]))=\lambda(\frX^S_{\tilde{\mathbf{1}}}),$
	\item $d(\Sel(M_f[\frp]))=\lambda(\frX^S_f)+d(\coker_{f[\frp]})+d(K_f),$
	\item $d(\Sel(M_{\tilde{\omega}}[\frp]))=\lambda(\frX^S_{\tilde{\omega}})+d(\coker_{\tilde{\omega}[\frp]})+d(K_{\tilde{\omega}}).$
\end{enumerate}
\end{lemma}
The next step is to relate the left hand sides of the above equations. Unlike in~\cite{CGLS}, they do not simply satisfy an additive relation. Instead, we have
\begin{lemma}\label[lemma]{SelfSelomega}  
If $\tilde{\mathbf{1}}\ne\mathbf{1}$ (i.e., the local trivial character is not globally trivial), we have\[
	d(\Sel(M_f[\frp]))=d(\Sel(M_{\tilde{\omega}}[\frp]))+d(K)+d(\coker_{f[\frp]})-d(\coker_{\tilde{\omega}[\frp]})-d(C)+1,\] 
	where $K$ (resp.\ $C$) is the kernel (resp.\ cokernel) of the map\[
	\coker(\H^1(K^\Sigma/K,M_{\tilde{\omega}}[\frp])\to \H^1(K^\Sigma/K,M_f[\frp])) \to \coker\Big(\frac{\H^1(K_{\ol v},M_{\tilde{\omega}}[\frp])}{\ker(\res_{M_{\tilde{\omega}}[\frp]})}\to \frac{\H^1(K_{\ol v},M_f[\frp])}{\ker(\res_{M_f[\frp]})}\Big).\] 
If $\tilde{\mathbf{1}}=\mathbf{1}$, we have\[
d(\Sel(M_f[\frp]))=d(\Sel(M_{\tilde{\omega}}[\frp]))+d(K)+d(\coker_{f[\frp]})-d(\coker_{\tilde{\omega}[\frp]})-d(C).\]
\end{lemma}
\begin{proof}
	This is just generalized snake lemma:
	{\small\[\begin{tikzcd}
		0(\ar[r,dashed]&	\F) \ar[r]& \H^1(K^\Sigma/K, M_{\tilde{\omega}}[\frp]) \ar[r,"g"]\ar[d]& \H^1(K^\Sigma/K, M_f[\frp]) \ar[r]\ar[d]& \coker(g) \ar[r]\ar[d]& 0\\
			0 \ar[r]&\F \ar[r]& \frac{\H^1(K_{\ol v}, M_{\tilde{\omega}}[\frp])}{\ker(\res_{M_{\tilde{\omega}}[\frp]})} \ar[r,"g_{\ol v}"]\ar[d]& \frac{\H^1(K_{\ol v}, M_f[\frp])}{\ker(\res_{M_f[\frp]})} \ar[r]\ar[d]&\coker(g_{\ol v})  \ar[r]\ar[d] &0\\&
			& \coker_{\tilde{\omega}[\frp]} \ar[r]& \coker_{f[\frp]}\ar[r]&C \ar[r]&0 
		\end{tikzcd}\]}
	Here the map $g$ is injective if $\tilde{\mathbf{1}}\ne\mathbf{1}$ and it has kernel $\F$ if $\tilde{\mathbf{1}}=\mathbf{1}$.\footnote{This is the only place where global trivialness of the character makes a difference on the algebraic side. Similar difference also appears if we choose the other ordering of the characters, since it determines whether $f$ has global torsion, as the (local) trivial character is a subrepresentation in that case. This difference corresponds to a different Iwasawa Main Conjecture for the (global) trivial character. See~\cref{anacomp}.}
	
	To see why the top two rows are exact on the left, first notice that $\H^0(K^\Sigma/K, M_?[\frp])=0$ for all $?=\tilde{\omega},f$ and $\tilde{\mathbf{1}}$ if $\tilde{\mathbf{1}}$ is globally non-trivial (it's $\F$ otherwise). This gives exactness of the first row. Now $\H^0(K_{\ol v}, M_{\tilde{\mathbf{1}}}[\frp])=\F$ and $\H^0(K_{\ol v}, M_f[\frp])=\H^0(K_{\ol v}, M_f)[\frp]=\F$ because we are in Case I.
	
	Thus the map\[\H^1(K_{\ol v},M_{\tilde{\omega}}[\frp])\to \H^1(K_{\ol v},M_f[\frp])\] is injective. Adding the local conditions $\ker(\res_{M_{\tilde{\omega}}[\frp]})=0,\ker(\res_{M_f[\frp]})=\F$,  we get an exact sequence\[
	0\to \F \to\frac{\H^1(K_{\ol v}, M_{\tilde{\omega}}[\frp])}{\ker(\res_{M_{\tilde{\omega}}[\frp]})} \xrightarrow{g_{\ol v}} \frac{\H^1(K_{\ol v}, M_f[\frp])}{\ker(\res_{M_f[\frp]})}. \qedhere\]
\end{proof}
For later use, we remark that we also get from the above argument that\[
\coker(g_{\ol v}) \cong \coker( \H^1(K_{\ol v},M_{\tilde{\omega}}[\frp]) \to \H^1(K_{\ol v},M_f[\frp])).\]

 To finish the proof of~\cref{imprimlambda}, we now relate $K_f$ to $K_{\tilde{\omega}}$. The argument in~\cref{[p]exact} (or more precisely, \cite[Proposition~3, Proposition~4]{Greenberg1989} together with~\cref{H1ofMtheta} and~\cref{corankH1Mf}) implies the vanishing of $\H^2(K^\Sigma/K,M_?)$. We can then identify $K_f$ (resp.\ $K_{\tilde{\omega}}$) with $K^{(2)}_f$ (resp.\ $K^{(2)}_{\tilde{\omega}}$) in the following two diagrams (noting that $\ker(\res_{M_f})=\ker(\res_{M_{\tilde{\omega}}})=0$) 
 
 	{\small\[\begin{tikzcd}
 		0\ar[r]& K_f \ar[r,"\cong"]\ar[d]& K^{(2)}_f\ar[r]\ar[d]&0 \ar[r]\ar[d]&0 \\
 		0 \ar[r]& \H^1(K^\Sigma/K, M_f)/\frp \ar[r]\ar[d]& \H^2(K^\Sigma/K, M_f[\frp]) \ar[r]\ar[d]& \H^2(K^\Sigma/K, M_f)[\frp]=0 \ar[r]\ar[d] &0\\
 		0 \ar[r]& \H^1(K_{\ol v}, M_f)/\frp \ar[r]& \H^2(K_{\ol v}, M_f[\frp]) \ar[r]& \H^2(K_{\ol v}, M_f)[\frp] \ar[r]& 0
 	\end{tikzcd}\]}
 
 {\small\[\begin{tikzcd}
 		0\ar[r]& K_{\tilde{\omega}} \ar[r,"\cong"]\ar[d]& K^{(2)}_{\tilde{\omega}}\ar[r]\ar[d]&0 \ar[r]\ar[d]&0 \\
 		0 \ar[r]& \H^1(K^\Sigma/K, M_{\tilde{\omega}})/\frp \ar[r]\ar[d]& \H^2(K^\Sigma/K, M_{\tilde{\omega}}[\frp]) \ar[r]\ar[d]& \H^2(K^\Sigma/K, M_{\tilde{\omega}})[\frp]=0 \ar[r]\ar[d] &0\\
 		0 \ar[r]& \H^1(K_{\ol v}, M_{\tilde{\omega}})/\frp \ar[r]& \H^2(K_{\ol v}, M_{\tilde{\omega}}[\frp]) \ar[r]& \H^2(K_{\ol v}, M_{\tilde{\omega}})[\frp] \ar[r]& 0
 	\end{tikzcd}\]}
 
It therefore suffices to relate $K^{(2)}_f$ to $K^{(2)}_{\tilde{\omega}}$.

\cref{[p]exact} also gives the vanishing of $\H^2(K^\Sigma/K,M_{\tilde{\mathbf{1}}}[\frp])$. By local duality and Shapiro's lemma, we get $\H^2(K_{\ol v},M_{\tilde{\mathbf{1}}}[\frp])=\H^0(K_{\ol v,\infty},\F(\omega))=0$ as well. We then obtain the following diagram

{\small
	\begin{equation}\tag{$K^{(2)}_f, K^{(2)}_{\tilde{\omega}}$}\begin{tikzcd}
			0 \ar[r]& K^{(2)}_{\tilde{\omega}\to f} \ar[r]\ar[d]& K^{(2)}_{\tilde{\omega}} \ar[r]\ar[d]& K^{(2)}_f \ar[d]& \\
			0 \ar[r]& \ker(h) \ar[r]\ar[d]& \H^2(K^\Sigma/K, M_{\tilde{\omega}}[\frp]) \ar[r,"h"]\ar[d]& \H^2(K^\Sigma/K, M_f[\frp]) \ar[r]\ar[d]& \H^2(K^\Sigma/K,M_{\tilde{\mathbf{1}}}[\frp])=0\\
			0 \ar[r]& \ker(h_{\ol v}) \ar[r]\ar[d,dashed]& \H^2(K_{\ol v}, M_{\tilde{\omega}}[\frp]) \ar[r,"h_{\ol v}"]& \H^2(K_{\ol v}, M_f[\frp])\ar[r]&\H^2(K_{\ol v},M_{\tilde{\mathbf{1}}}[\frp])=0 \\
			&0&&&
\end{tikzcd}\end{equation}}
from which we get 
\begin{equation}\label{K2omega}
	d(K^{(2)}_{\tilde{\omega}})=d(K^{(2)}_f)+d(K^{(2)}_{\tilde{\omega}\to f}).
\end{equation} 

The middle left vertical map is surjective because its domain and codomain are quotients of those of $\H^1(K^\Sigma/K,M_{\tilde{\mathbf{1}}}[\frp])\to \H^1(K_{\ol v},M_{\tilde{\mathbf{1}}}[\frp])$ which is surjective. 

Next, consider the diagram {\small
	\begin{equation}\tag{Sel,$\tilde{\mathbf{1}}$}\begin{tikzcd}
			0 \ar[r]& K \ar[r]\ar[d]& \Sel(M_{\tilde{\mathbf{1}}}[\frp]) \ar[r]\ar[d]& \ker(j) \ar[d]& \\
			0 \ar[r]& 	\coker(g) \ar[r]\ar[d]& \H^1(K^\Sigma/K, M_{\tilde{\mathbf{1}}}[\frp]) \ar[r]\ar[d]& \ker(h)\ar[r]\ar[d,"j"]& 0\\
			0 \ar[r]& \coker(g_{\ol v}) \ar[r,"i"]\ar[d]& \frac{\H^1(K_{\ol v}, M_{\tilde{\mathbf{1}}}[\frp])}{\ker(\res_{M_{\tilde{\mathbf{1}}}[\frp]})} \ar[r]\ar[d]& \coker(i) \ar[r]\ar[d]&0 \\
			&  C\ar[r]& 0 \ar[r]& 0 \ar[r]&0
		\end{tikzcd}\end{equation}}
	where $g,g_{\ol v}$ are the maps in~\cref{SelfSelomega}. Note that by~\cref{[p]exact}, the middle vertical map is surjective, so is the right vertical map. We thus conclude that
\begin{equation}\label{Sel1tok}	
	d(\Sel(M_{\tilde{\mathbf{1}}[\frp]}))=d(K)+d(\ker(j))-d(C).
\end{equation}
	
	The above diagram together with the diagram{\small
		\begin{equation*}\begin{tikzcd}
				0 \ar[r]& 0 \ar[r]\ar[d]& \ker(\res_{M_{\tilde{\mathbf{1}}}})=\F\ar[r]\ar[d]& \ker(k) \ar[d]& \\
				0 \ar[r]& 	\coker(\H^1(K_{\ol v},M_{\tilde{\omega}}[\frp]) \to \H^1(K_{\ol v},M_f[\frp])) \ar[r]\ar[d]& \H^1(K_{\ol v}, M_{\tilde{\mathbf{1}}}[\frp]) \ar[r]\ar[d]& \ker(h_{\ol v})\ar[r]\ar[d,"k"]& 0\\
				0 \ar[r]& \coker(g_{\ol v}) \ar[r,"i"]\ar[d]& \frac{\H^1(K_{\ol v}, M_{\tilde{\mathbf{1}}}[\frp])}{\ker(\res_{M_{\tilde{\mathbf{1}}}[\frp]})} \ar[r]\ar[d]& \coker(i) \ar[r]\ar[d]&0 \\
				&  0\ar[r]& 0 \ar[r]& 0 \ar[r]&0
	\end{tikzcd}\end{equation*}}
then yields the following diagram{\small
	\begin{equation*}\begin{tikzcd}
			0 \ar[r]&  K^{(2)}_{\tilde{\omega}\to f}\ar[r,"="]\ar[d]& K^{(2)}_{\tilde{\omega}\to f}\ar[r]\ar[d]& 0 \ar[d]& \\
			0 \ar[r]& 	\ker(j) \ar[r]\ar[d]& \ker(h) \ar[r]\ar[d]& \coker(i) \ar[r]\ar[d,"="]& 0\\
			0 \ar[r]& \ker(k)=\F \ar[r]\ar[d]& \ker(h_{\ol v}) \ar[r]\ar[d]& \coker(i) \ar[r]\ar[d]&0 \\
			&  0\ar[r]& 0 \ar[r]& 0 \ar[r]&0
\end{tikzcd}\end{equation*}}
 from which we see 
 \begin{equation}\label{K2omega'}
 	d(\ker(j))=d(K^{(2)}_{\tilde{\omega}\to f})+1.
 \end{equation}
	
	Combining equations~\eqref{K2omega},~\eqref{Sel1tok},~\eqref{K2omega'} with~\cref{Seltolambda} and~\cref{SelfSelomega}, we get \[	\lambda(\frX^S_f) = \lambda(\frX^S_\phi) + \lambda(\frX^S_\psi)\] if $\phi=\tilde{\omega}\ne\omega$ (equivalently, if $\psi=\tilde{\mathbf{1}}\ne\mathbf{1}$); and \[	\lambda(\frX^S_f) +1= \lambda(\frX^S_\phi) + \lambda(\frX^S_\psi)\] if $\phi=\tilde{\omega}=\omega$ (equivalently, if $\psi=\tilde{\mathbf{1}}=\mathbf{1}).$

\textbf{Case II: $\ker(\res_{M_f})=\cO/\pi^n$ for some $n\geq 1$}

In this case, it is necessary that $\ker(\res_{M_f})[\frp]=\F=\ker(\res_{M_f}[\frp])$. Since $\ker(\res_{M_f})$ is no longer divisible, the natural injective map \begin{equation*}
	\frac{\H^1(K_{\ol v},M_f)[\frp]}{\ker(\res_{M_f})[\frp]}\to  	\frac{\H^1(K_{\ol v},M_f)}{\ker(\res_{M_f})}[\frp]
\end{equation*}
has cokernel $\ker(\ker(\res_{M_f})/\frp \to \H^1(K_{\ol v},M_f)/\frp)$ which might be $0$ or $\F$. Denote this cokernel by $C_f$.

The map\[
\frac{\H^1(K_{\ol v}, M_f[\frp])}{\ker(\res_{M_f[\frp]})} \to \frac{\H^1(K_{\ol v}, M_f)}{\ker(\res_{M_f})}[\frp]\] factors through the above map and has kernel $\F$. The diagram~\eqref{M?tolambda} now becomes
{\small
	\begin{equation}\tag{nr,$?=f$}\label{nr?=f}\begin{tikzcd}
			0 \ar[r]& 0 \ar[r]\ar[d]& \Sel(M_f[\frp]) \ar[r]\ar[d]& \Sel(M_f)[\frp] \ar[d]& &\\
			0 \ar[r]& 0 \ar[r]\ar[d,"r_f^0"]& \H^1(K^\Sigma/K, M_f[\frp]) \ar[r]\ar[d,"r_{f[\frp]}"]& \H^1(K^\Sigma/K, M_f)[\frp] \ar[r]\ar[d,"r_{f}\text{[$\frp$]}"]& 0&\\
			0 \ar[r]& \ker(f_f)=\F \ar[r]\ar[d]& \frac{\H^1(K_{\ol v}, M_f[\frp])}{\ker(\res_{M_f[\frp]})} \ar[r,"f_f"]\ar[d]& \frac{\H^1(K_{\ol v}, M_f)}{\ker(\res_{M_f})}[\frp]\ar[r]\ar[d]&C_f\ar[r]&0 \\
			& \F \ar[r]& \coker_{f[\frp]} \ar[r]& \coker_f[\frp] \ar[r] &C_f\ar[r] &0
\end{tikzcd}\end{equation}}

Therefore as in the Case I we still get
\[
\lambda(\frX^S_f)=\dim_\F \Sel(M_f)[\frp]-\dim_\F \coker_f[\frp]-\dim_\F K_f,
\]
and hence
\begin{equation}\label{finitekererror}
	d(\Sel(M_f[\frp]))=\lambda(\frX^S_f)+d(C_f)+d(\coker_{f[\frp]})+d(K_f)-1
\end{equation}
by the generalized snake lemma.
 
Finally, we compare $K_f$ to $K_{\tilde{\omega}}$ in this case. Recall that in Case I we showed that 
\[d(K_{\tilde{\omega}})=d(\ker(\H^1(K^\Sigma/K,M_f)/\frp \to \H^1(K_{\ol v},M_f)/\frp))+d(K^{(2)}_{\tilde{\omega}\to f})\](here the first kernel coincides with $K_f$ in that case), it suffices to compare $K^I_f \defeq \ker(\H^1(K,M_f)/\frp \to \H^1(K_{\ol v},M_{\tilde{f}})/\frp)$ to $K_f$. This is done by considering the following diagram
{\small\begin{equation*}\begin{tikzcd}
		&0 \ar[r]&  0\ar[r]\ar[d]& K^I_f\ar[r]\ar[d]& K_f \ar[d]& \\
	&	0 \ar[r]& 	0 \ar[r]\ar[d]& \H^1(K^\Sigma/K,M_f)/\frp \ar[r,"="]\ar[d]& \H^1(K^\Sigma/K,M_f)\frp \ar[r]\ar[d]& 0\\
		0\ar[r]&C_f \ar[r]& \ker(\res_{M_f})/\frp=\F \ar[r]\ar[d]& \H^1(K_{\ol v},M_f)/\frp \ar[r]\ar[d]& \frac{\H^1(K_{\ol v},M_f)}{\ker(\res_{M_f})}/\frp \ar[r]\ar[d]&0 \\
		0\ar[r]&C_f\ar[r]&  \F \ar[r]& 0 \ar[r]& 0 \ar[r]&0
\end{tikzcd}\end{equation*}}
If $C_f=\F$, $K_f=K^I_f$ and~\eqref{finitekererror} becomes exactly as in Case I, so the result follows by the same computation. If $C_f=0$, $K_f=K^I_f+\F$ and~\eqref{finitekererror} becomes \[d(\Sel(M_f[\frp]))=\lambda(\frX^S_f)+d(\coker_{f[\frp]})+d(K_f)-1,\] and $d(K_{\tilde{\omega}})=d(K_f)+d(K^{(2)}_{\tilde{\omega}\to f})-1$. The same result follows from this observation.

\textbf{Case III: $\H^0(K_w,M_f)=0$}

In this case, it is necessary that $\ker(\res_{M_f})[\frp]=\F=\ker(\res_{M_f}[\frp])$ by~\cref{kerresf}. We only study what can be different from Case I.

The diagram~\eqref{nr?=f} now becomes
{\small
	\begin{equation*}\begin{tikzcd}
			0 \ar[r]& 0 \ar[r]\ar[d]& \Sel(M_f[\frp]) \ar[r]\ar[d]& \Sel(M_f)[\frp] \ar[d]& \\
			0 \ar[r]& \H^0(K^\Sigma/K, M_f)/\frp=0 \ar[r]\ar[d,"r_f^0"]& \H^1(K^\Sigma/K, M_f[\frp]) \ar[r]\ar[d,"r_{f[\frp]}"]& \H^1(K^\Sigma/K, M_f)[\frp] \ar[r]\ar[d,"r_{f}\text{[$\frp$]}"]& 0\\
			0 \ar[r]& \ker(f_f)=\H^0(K_{\ol v},M_f)/\frp=0 \ar[r]\ar[d]& \H^1(K_{\ol v}, M_f[\frp]) \ar[r,"f_f"]\ar[d]& \H^1(K_{\ol v}, M_f)[\frp]\ar[r]\ar[d]&0 \\
			& 0 \ar[r]& \coker_{f[\frp]} \ar[r]& \coker_f[\frp] \ar[r] &0
\end{tikzcd}\end{equation*}}
which agrees with Case I.

$\H^0(K_{\ol v},M_f)=0$ also implies $\H^0(K_{\ol v},M_f[\frp])=0$, so we get an exact sequence\[
0\to \H^0(K_{\ol v},M_{\tilde{\mathbf{1}}}[\frp])=\F\to \H^1(K_{\ol v},M_{\tilde{\omega}}[\frp]) \to \H^1(K_{\ol v},M_f[\frp]).\]
This is exactly the final ingredient in the proof of~\cref{SelfSelomega}. Now exactly the same arguments as in Case I give the desired relation.
\end{proof}

\subsection{Comparison I: Algebraic Iwasawa invariants}\label{1.5}

We're now ready to state the main result of the algebraic side. 

Let 
\[
\cP_w(f) \colonequals P_w(\ell^{-1}\gamma_w) \in \Lambda
\]
with $V_f = T_f \otimes \Q_\ell$ and $P_w = \det(1 - \Frob_w X \mid V_{f,I_w})$.

Recall that we have fixed the ordering of the characters appearing in the following decomposition:
\[0\to\F(\phi(=\tilde{\omega}))\to \ol\rho_f\to \F(\psi(=\tilde{\mathbf{1}}))\to 0.\]

\begin{theorem}\label{algmain}
Assume that $\phi|_{G_p} = \omega$ and $\phi|_{G_K} \ne \omega$. 

Then the module $\frX_f$ is $\Lambda$-torsion with $\mu(\frX_f) = 0$ and
\[
    \lambda(\frX_f) = \lambda(\frX_\phi) + \lambda(\frX_\psi) + \sum_{w \in S}\big\{\lambda(\cP_w(\phi)) + \lambda(\cP_w(\psi)) - \lambda(\cP_w(f))\big\}.
\]
In the case where $\phi|_{G_K} = \omega$, the same results hold except that the relation between $\Lambda$-invariants now becomes\[
\lambda(\frX_f) +1 = \lambda(\frX_\phi) + \lambda(\frX_\psi) + \sum_{w \in S}\big\{\lambda(\cP_w(\phi)) + \lambda(\cP_w(\psi)) - \lambda(\cP_w(f))\big\}.
\]

\end{theorem}
\begin{proof}
	With our~\cref{imprimlambda} and~\cref{surj for Self}, the proof is the same as that of~\cite[Theorem 1.5.1]{CGLS}.
\end{proof}

\section{Analytic side}\label{sec:analytic side}
Let $f \in S_k(\Gamma_0(N))$ with $k = 2r$ be a newform, $p\nmid 2N$ and $K$ an imaginary quadratic field satisfying the Heegner hypothesis. As on the algebraic side, we assume $p=v\ol v$ is split in $K$. We further assume that the discriminant $D_K$ of $K$ is odd and $D_K\ne -3$. Recall that $\Lambda=\cO\llbracket \Gamma\rrbracket$ denotes the anticyclotomic Iwasawa algebra, and set $\Lambda^\nr=\Lambda\hat{\otimes}_{\bZ_p}\bZ_p^{\nr}$, for $\bZ_p^{\nr}$ the completion of the ring of integers of the maximal unramified extension of $\bQ_p$. This is where the BDP $p$-adic $L$-function lives.

In this section, we continue to study the case $\ol \rho_f^{\mathrm{ss}}=\F(\phi)\oplus \F(\psi)$. Note that the cases when $\phi|_{G_p}=\mathbf{1}$ (resp.\ $\omega$) but $\phi|_{G_K}\ne \mathbf{1}$ (resp.\ $\omega$) are already studied in~\cite{CGLS}. We therefore only state the theorems we need and explain how to extend the results to the case where $\phi|_{G_K}=\mathbf{1}$.

\subsection{$p$-adic $L$-functions}

\subsubsection{The Bertolini--Darmon--Prasanna $p$-adic $L$-functions.}

Thanks to the Heegner hypothesis, we can fix an integral ideal $\frN \subset \cO_K$ with
\begin{equation*} 
    \cO_K/\frN \iso \Z/N.
\end{equation*}

\begin{proposition}[$p$-adic interpolation property]
	There exists an element $\cL_f\in\Lambda^{\nr}$ characterized by the following interpolation property: If $\hat{\xi} \in \frX_{p^\infty}$ is the $p$-adic avatar of a Hecke character $\xi$ of infinity type $(n,-n)$ with $n \geq 0$ and $p$-power conductor, then
	{\small\[
	\sL_f(\hat{\xi}) = \frac{\Omega_p^{4n}}{\Omega_K^{4n}}\cdot \frac{4\Gamma(n+\frac{k}{2})\Gamma(n-\frac{k}{2}+1)\xi^{-1}(\frN^{-1})}{(2\pi)^{2n+1}(\sqrt{D_K})^{2n-1}}\cdot (1-a_p(f)p^{-r}\xi_{\ol \frp}(p)+\xi_{\ol \frp}(p^2)p^{-1})^2\cdot L\big(f/K,\xi,\tfrac{k}{2}\big).  
	\]}
Here and in the following $p$-adic units may be ignored in the expressions.
\end{proposition}
\begin{proof}
	This is proved in~\cite[Theorem 2.1.1]{CGLS}, using results from~\cite{CastellaHsieh}. We remark here that the assumption $p \nmid (2k-1)!$ can be directly removed from \textit{op.\ cit.} so for arbitrary weights we would be able to study all odd primes.
	
	More precisely,~\cite[Prop. 3.8]{CastellaHsieh} gives
	\[
	\sL_f(\hat{\xi}) = \Omega_p^{2k + 2n} \cdot L^\alg(\tfrac{1}{2}, \pi_K\otimes\psi\xi) \cdot e_\frp(f,\psi\xi) \cdot \xi(\frN^{-1}) \cdot 2^{\#A(\psi) + 3}c_0\epsilon(f) \cdot u_K^2\sqrt{D_K},
	\]
	where
\[L^\alg(\tfrac{1}{2}, \pi_K\otimes\xi)=\frac{\Gamma(n+r)\Gamma(n-r+1)}{2(2\pi)^{2n+1}(\sqrt{D_K})^{2n}}\frac{L(\frac{1}{2},\pi_K\otimes\xi)}{\Omega_K^{4n}},
\]
    since $\Im(\theta)=\frac{\sqrt{D_K}}{2}$.
	Note that we can take $c=c_0=1$ for our application to the weight $2$ BSD formula (see for example~\cite[Theorem 5.1.3]{CGLS} where the summation is over $\sigma\in \Gal(H/K)$), so in particular $p\nmid c_0$ and\[
	e_\frp(f,\xi)=(1-a_p(f)p^{-r}\xi_{\ol \frp}(p)+\xi_{\ol \frp}(p^2)p^{-1})^2.
	\]
	Also notice that $u_K=1$ if $D_K<-4$, $A(\psi)=1$ since $(D_K, N)=1$. The global root number $\epsilon(f)=\pm1$ is ignored since it's a unit. Now as in~\cite[Theorem 2.1.1]{CGLS}, put $\eta^{-1}\xi$ in place of $\xi$ for a Hecke character $\eta$ of infinity type $(\frac{k}{2},-\frac{k}{2})$, and notice that $\eta(\frN)$ is also a unit since $p\nmid N$, we get the desired interpolation property.
\end{proof}
\begin{remark}\label[remark]{compperiods}
	We mention that the period $\Omega_K$ is denoted by $\Omega_\infty$ in~\cite{CGLS}. But as is pointed out in~\cite[Remark 2.3.2]{CGS}, the $\Omega_\infty$ defined in the next theorem (coming from Katz, de Shalit and Hida--Tilouine) is different from the $\Omega_K$ in this proposition (coming from~\cite{BDP13}) and they satisfy the relation $\Omega_\infty=2\pi i\cdot \Omega_K$.
\end{remark}

\subsubsection{Katz $p$-adic $L$-functions.}

\begin{theorem}
	There exists an element $\cL_f \in \Lambda^\nr$ characterized by the following interpolation property: For every character $\xi$ of $\Gamma$ crystalline at both $v$ and $\bar{v}$ and corresponding to a Hecke character of $K$ of infinity type $(n,-n)$ with $n\in\bZ_{>0}$ and $n\equiv 0$ (mod $p-1$), we have
	\begin{align*}
		\cL_{\theta}(\xi)=\frac{\Omega_p^{2n}}{\Omega_\infty^{2n}}\cdot 4\Gamma(n+\tfrac{k}{2})\cdot\frac{(2\pi i)^{n-\frac{k}{2}}}{\sqrt{D_K}^{n-\frac{k}{2}}}\cdot(1-\theta^{-1}(p)\xi^{-1}(v))\cdot(1-\theta(p)\xi(\bar{v})p^{-1})\\
		\times\prod_{\ell\mid C}(1-\theta(\ell)\xi(w)\ell^{-1})\cdot L(\theta_K\xi\mathbf{N}_K^\frac{k}{2},0).
	\end{align*}
\end{theorem}

\begin{proof}
See~\cite[Theorem 27]{Kri16}. Similarly as in~\cite[Theorem 2.1.2]{CGLS}, we take $\chi=\theta_K\xi\mathbf{N}_K^\frac{k}{2} ,k_1=n+\frac{k}{2},k_2=-n+\frac{k}{2}$ (note that the infinity type of the Hecke characters in these $2$ papers differ by a sign. Also the weight in~\cite{Kri16} is $k$ while the weight in~\cite{CGLS} is $2k$).
\end{proof}

\subsection{Comparison II: Analytic Iwasawa invariants}

\begin{theorem}\label{pLcongruence}
	$\cL_f\equiv(\mathscr{E}^{\iota}_{\phi,\psi})^2\cdot(\cL_\phi)^2 \pmod{p\Lambda^{\nr}}.$
\end{theorem}
\begin{proof}
	See~\cite[Theorem 2.2.1]{CGLS}. Note that when $\phi|_{G_K}\ne\mathbf{1}$, we always have full Eisenstein descent by~\cite{Kri16}. If $\phi|_{G_K}=\mathbf{1}$, we would only get partial Eisenstein descent in some cases. However, the proof in~\cite{CGLS} essentially only requires partial Eisenstein descent because only the constant term of $G$ (a certain Eisenstein series in \textit{loc.\ cit.}) can be different and the $p$-depletion of $G$ does not see the constant term. 
	
	We remark here that the formula in~\cite[Theorem 37]{Kri16} now becomes\begin{align*}
		\cL_G(\xi)=\frac{\Omega_p^{2n}}{\Omega_K^{2n}}\cdot\frac{\Gamma(n+\frac{k}{2})\phi^{-1}(\sqrt{-D_K})\xi(\bar{\mathfrak{t}})t}{\mathfrak{g}(\phi)}\cdot \frac{1}{(2\pi i)^{n+\frac{k}{2}}\sqrt{D_K}^{n-\frac{k}{2}}}\\
	\times\Xi_{\xi^{-1}\mathbf{N}_K^{-\frac{k}{2}}}(\phi,\psi\omega^{-1},N_+,N_-,N_0)\cdot L(\phi_K\xi\mathbf{N}_K^\frac{k}{2},0),
		\end{align*}
 where we take $\chi^{-1}=\xi\mathbf{N}_K^\frac{k}{2},\ \psi_1=\phi,\ \psi_2=\phi^{-1}=\psi\omega^{-1}$, and $j=n-\frac{k}{2}$.
 Here we're using $\mathbf{N}_K(\bar{\mathbf{t}})=1$ is a unit since $p\nmid c$ because $t\mid N$.
 
 From~\cref{compperiods}, the above formula becomes
 \begin{align*}
 	\cL_G(\xi)=\frac{\Omega_p^{2n}}{\Omega_\infty^{2n}}\cdot\frac{\Gamma(n+\frac{k}{2})\phi^{-1}(\sqrt{-D_K})\xi(\bar{\mathfrak{t}})t}{\mathfrak{g}(\phi)}\cdot \frac{(2\pi i)^{n-\frac{k}{2}}}{\sqrt{D_K}^{n-\frac{k}{2}}}\\
 	\times\Xi_{\xi^{-1}\mathbf{N}_K^{-\frac{k}{2}}}(\phi,\psi\omega^{-1},N_+,N_-,N_0)\cdot L(\phi_K\xi\mathbf{N}_K^\frac{k}{2},0),
 \end{align*}
which agrees with the formula in~\cite[eq. (2.7)]{CGLS}. 
\end{proof}

Combining the above inputs, as in~\cite{CGLS}, we obtain the following:
\begin{theorem}\label{anacong}
	Assume that $\ol\rho_f^{\mathrm{ss}}=\F(\phi)\oplus\F(\psi)$ as $G_\Q$-modules, with the characters $\phi$, $\psi$ labeled so that $p\nmid cond(\phi)$. Then $\mu(\cL_f)=0$ and\[
	\lambda(\cL_f)=\lambda(\cL_\phi)+\lambda(\cL_\psi)+\sum_{w\in S}\{\lambda(\cP_w(\phi))+\lambda(\cP_w(\psi))-\lambda(\cP_w(f))\}.\]
\end{theorem}

\begin{theorem}\label{anacomp}
	Assume that $\ol\rho_f^{ss}=\F(\psi)\oplus\F(\phi)$. Then $\mu(\cL_f)=\mu(\frX_f)=0$ and\[
	\lambda(\cL_f)=\lambda(\frX_f).\]
\end{theorem}
\begin{proof}
 The proof relies on the one-variable Iwasawa Main Conjectures for the characters which is already explained in~\cref{Rubin Hida}. We will add a little more detail. When $\phi|_{G_K}\notin\{\mathbf{1},\omega\}$, the proofs are identical to those in~\cite{CGLS}. Note also that in those cases the results on the algebraic side (see~\cref{algmain}) are also compatible. We explain how to modify the arguments to the case where $\psi|_{G_K}=\mathbf{1}$.
 
 It is essentially proved in~\cite{Rubin1991} that for any character $\theta$ of $G_K$, we have\[
 \Char(\H^1_{\cF_{\nr}}(K,M_\theta))=\Char(U_\infty/\bar{\mathscr{C}_\infty})^\theta,\] 
 where $U_\infty$ and $\mathscr{C}_\infty$ are the local units and elliptic units defined in \textit{op.\ cit.} As is mentioned in the remarks after Theorem 4.1 there, the further identification\[
 \Char(U_\infty/\bar{\mathscr{C}_\infty})^\theta=\Char(\cL_\theta)
 \] 
 is done in~\cite{CW78}. Notice that they need to assume (i) $p>3$; (ii) $p$ is not anomalous for their elliptic curve $E$; and (iii) $i\ne 0$ where $i$ is the power of $\chi$ acting on the isotypic components $(U_{\infty}/\bar{\mathscr{C}_\infty})^{(i)}$, with $\chi$ giving the canonical action of $\Gal(K(E[\pi])/K)$ on $E[\pi]$. Notice that our mod $\frp$ character $\theta$ is the reduction of $\chi$, where the trivial character $\mathbf{1}$ corresponds to $i=0$ and $\omega$ corresponds to $i=1$. 
 
 The above assumptions are not needed in a more general theorem~\cite[III.1.10]{deShalit} of de Shalit, but the result is different when $\theta|_{G_K}=\mathbf{1}$. We remark here that the $i=0$ case can also be studied directly using the idea of~\cite[Lemma 28]{Yager} where $(i_1,i_2)=(0,0)$, which also yields 
  \[ \Char(U_\infty/\bar{\mathscr{C}_\infty})^\mathbf{1}=\Char(\cL_\mathbf{1}\cdot T),
 \] 
 showing that $\lambda(\fX_{\mathbf{1}})=\lambda(\cL_\mathbf{1})+1$. Together with~\cref{algmain}, this yields the equation $\lambda(\fX_f)=\lambda(\cL_f)$, as expected.
\end{proof}

\begin{remark}\label[remark]{IMCwithtorsion}
	The above arguments imply that the standard Iwasawa Main Conjectures for the trivial character is different from others, namely the algebraic $\lambda$-invariant is bigger than the analytic $\lambda$-invariant by $1$. Actually, it is always predicted that some global torsion terms should appear (see, for example,~\cite{Greenbergmotive}). However, in our setting, even if $f$ has global torsion, it must be finite (from~\cref{H0 of V/T}\,(iii)) and will not affect the $\lambda$-invariants, so our use of Perrin-Riou's formula in section~\ref{sec:selmer-groups-of-rhof} is justified.
\end{remark}

\section{Proofs of the Iwasawa Main Conjectures}

We will directly adopt the notations and the arguments from~\cite{CGLS} to prove the main theorems of this paper, some of which are further generalized in~\cite{CGS}. More precisely, we will prove two Iwasawa Main Conjectures at a time. In the Kolyvagin system argument, we need to replace their assumption $E(K)[p]=0$ by $\H^0(K,\ol\rho_f)=0$ so that their arguments apply to arbitrary weight modular forms. By the arguments in the beginning of~\cref{sec:selmer-groups-of-rhof}, we are allowed to make this assumption. We will list the intermediate results we need and comment on what should and should not be changed.

Let $f \in S_k(\Gamma_0(N))^\new$ be a newform and let $\Z[f]$ be the coefficient ring of $f$ with field of fractions $\Q(f)$. Let $F$ be a finite extension of the completion of $\Q(f)$ at a chosen prime above $p\nmid 2N$ with ring of integers $\cO$. Let $\pi \in \cO$ be a uniformizer and $\F \defeq \cO/\pi$ be the residue field of $\cO$. Throughout this section, assume that
\begin{equation*}\tag{h.1}\label{h1}
\H^0(K,\ol\rho_f)=0.
\end{equation*}

As before, $\Gamma=\Gal(K_\infty/K)$ denote the Galois group of the anticyclotomic $\bZ_p$-extension of $K$. We let $\alpha\colon \Gamma\to R^\times$ be a character with values in the ring of integers $R$ of a finite extension $\Phi/\bQ_p$. Let\[r \defeq \rk_{\cO}R.\]
Let $\rho_f\colon G_\Q\to \GL_2(F)$ be the unique semisimple $p$-adic Galois representation unramified outside $pN$ and such that $\rho_f(\Frob_\ell)$ has characteristic polynomial\[T^2-a_\ell \ell^{1-\frac{k}{2}}T+\ell\] for all $\ell\nmid pN$. 

Let $T$ be an Galois stable lattice and consider the $G_K$-modules\[
T_\alpha \defeq T\otimes_\cO R(\alpha),\ \  V_\alpha \defeq T\otimes_\cO\Phi, \ \ A_\alpha \defeq T\otimes_R\Phi/R\cong V_\alpha/T_\alpha.\]

Let $\gamma\in\Gamma$ be a topological generator, and set $C_\alpha\coloneq\begin{cases} v_p(\alpha(\gamma)-\alpha^{-1}(\gamma)) & \alpha\ne\alpha^{-1},\\
	0 &  \alpha=\alpha^{-1}.
\end{cases}$

One last thing we need to explain is how to define the filtration on $V_f,T_f$ and $A_f$. Recall that in the case of elliptic curves one defines $\Fil_w^+(T_pE)$ as the kernel of the reduction map $T_pE\to T_p\tilde{E}$ at a place $w$ of $K$ above $p$. For $p$-ordinary modular forms (good ordinary or multiplicative), one can define $\Fil_w^+(T_f)$ on any Galois stable lattice $T_f$ of $V_f$ (without self-dual twist) by the characterization that $\Fil^-(T_f)\defeq T_f/\Fil^+(T_f)$ is the unique quotient of $T_f$ that is unramified. However, note that after a self-dual twist for weight $k>2$, the $G_{\bQ_p}$ action on $\Fil^-(T_f)$ (by abuse of notation) will be given by an unramified character multiplied by a non-trivial power of the cyclotomic character $\chi$. 

One then defines the filtration on $T_\alpha$, $V_\alpha$, $A_\alpha$, $\bT$ and $M_f$ (or $\bA$) in the same way as they do in~\cite{CGLS}, and use these in the definition of the ordinary Selmer groups.

The first remarkable result from~\cite{CGLS} is the following theorem generalizing~\cite[Theorem 2.2.10]{How2004}.
\begin{theorem}\label{main}
	Suppose $\alpha\neq 1$ and that there is a Kolyvagin system $\kappa_\alpha=\{\kappa_{\alpha,n}\}_{n\in\mathcal{N}}\in\KS(T_\alpha,\mathcal{F}_{\ord},\mathcal{L}_f)$ with $\kappa_1$ being non-torsion. Then $\H^1_{\cF_{\ord}}(K,T_\alpha)$ has rank one, and there is a finite $R$-module $M_\alpha$ such that
	\begin{equation*}
		\H^1_{\cF_{\ord}}(K,A_\alpha)\cong (\Phi/R)\oplus M_\alpha\oplus M_\alpha
	\end{equation*}
	with\begin{equation*}
		\length_R(M_\alpha)\leq\length_R(\H^1_{\cF_{\ord}}(K,T_\alpha)/R\cdot \kappa_{\alpha,1})+E_\alpha
	\end{equation*}
	for some constant $E_\alpha\in\bZ_{\geq 0}$ depending only on $C_\alpha, T_f$, and $\rk_{\cO}(R)$.
\end{theorem}

In~\cite[section 6]{CGS}, the above result was generalized so that the error term can be made independent of $\alpha$, at the expense of requiring $\alpha$ to be `sufficiently close to $1$'. More precisely, they have

\begin{theorem}\label{mainCGS}
	Assume $\alpha\equiv 1 \pmod{\varpi^m}$. Then there exist non-negative integers $\mathcal{M}$ and $\mathcal{E}$ depending only on $T_pE$ and $\rk_{\bZ_p}(R)$ such that if $m\ge \mathcal{M}$ and if there is a Kolyvagin system $\kappa_\alpha=\{\kappa_{\alpha,n}\}_{n\in\mathcal{N}}\in\KS(T_\alpha,\mathcal{F}_{\ord},\mathcal{L}_f)$ with $\kappa_1$ being non-torsion. Then $\H^1_{\cF_{\ord}}(K,T_\alpha)$ has rank one, and there is a finite $R$-module $M_\alpha$ such that
	\begin{equation*}
		\H^1_{\cF_{\ord}}(K,A_\alpha)\cong (\Phi/R)\oplus M_\alpha\oplus M_\alpha
	\end{equation*}
	with\begin{equation*}
		\length_R(M_\alpha)\leq\length_R(\H^1_{\cF_{\ord}}(K,T_\alpha)/R\cdot \kappa_{\alpha,1})+\mathcal{E}
	\end{equation*}
	and the `error term' $\mathcal{E}$ is independent of $m$.
\end{theorem}

We mention that~\cref{main} (in combination with our~\cref{algmain} and~\cref{anacomp}) is sufficient as an ingredient for us to get a desired BSD formula in rank $1$, but together with~\cref{mainCGS} it leads to a stronger result on the anticyclotomic Main Conjecture, in the sense that the divisibility in~\cite[Theorem 3.4.1]{CGLS} holds more generally without the need to invert $\gamma-1$ (see~\cite[Theorem 6.5.1]{CGS}).   

Thanks to the highly axiomatized arguments from~\cite{CGLS}, \cite{CGS}, we only need to explain how to get the error term $C_1$ in \textit{op.\ cit.}, as everything else can be directly extended to our more general setting. 

According to the proof of~\cite[Proposition 3.6]{CGLS}, we only need to find one non-trivial scalar matrix in $\cO^\times\cap \im(\rho_f|_{G_{K_\infty}})$.

\begin{lemma}\label[lemma]{C1}	The intersection $U=\cO^\times\cap \im(\rho_f|_{G_{K_\infty}})$ contains infinitely many scalar matrices. Here by abuse of notation $\cO^\times$ denotes $\cO^\times\cdot I_2$, the invertible scalar matrices in $\GL_2(\cO)$.
\end{lemma}
\begin{proof}
	Note that by the discussion above~\cite[Lemma 3.3.3]{CGLS}, we have\[
	\rho_f(G_\bQ)=\rho_f(G_{K_\infty}).\]Since $\rho_f$ is a Hodge--Tate representation, Bogomolov in~\cite{Bogomolov} showed the image (possibly after embedding $\GL_2(F)$ in $\GL_{2n}(\bQ_p)$ for some $n$) is open in its algebraic envelope. Since $\rho_f$ is irreducible by~\cite{RibetNebentypus}, so is its natural projection in $\PGL_2(F)$. We also know $\im(\rho_f)$ must be infinite. Actually, there will be infinitely many determinants since $\det\rho_f=\chi^{k-1}$ from \textit{op.\ cit.} Now by the classification of all algebraic subgroups of $\GL_2(F)$ (see~\cite{algsubgp}, whose results work over any algebraically closed field of characteristic $0$), we see that the algebraic envelope of $\im(\rho_f)$ must be $F^\times$ times a minimal group (in Theorem $4$ of \textit{op.\ cit.}, only (2)(a)(b) have infinitely many determinants, yet they are reducible), therefore it must contain infinitely many scalar matrices. So does the image.
\end{proof}
\begin{remark}
	It is mentioned in~\cite{CGLS} that $f$ has no CM by $K$ under their assumptions. When $f$ has no CM, there is a more straightforward proof of this openness result, using Momose's results (see for example~\cite{Ribet1985}) generalizing the open image results of Serre.
\end{remark}
For $U=\cO^\times\cap \im(\rho_f)$ as in~\cref{C1}, let\[
C_1\defeq\min\{v_p(u-1):u\in U\}.\]
Since $U$ is an open subgroup of $\cO^\times$, $0\leq C_1<\infty$.\\

The rest of the proof of~\cref{main} and \cref{mainCGS} can be directly borrowed from the corresponding papers. We next explain how to apply Iwasawa theory to prove the Main Conjectures. First we study a consequence of~\cref{main}, \cref{mainCGS}.

\begin{theorem}\label{HeegMC}
	Suppose there is a Kolyvagin system $\kappa\in\KS(\bT,\cF_\Lambda,\sL_f)$ where $\kappa_1$ is non-torsion. Then $\H^1_{\cF_\Lambda}(K,\bT)$ has $\Lambda$-rank one, and there is a finitely generated torsion $\Lambda$-module $M$ such that
	\begin{enumerate}
		\item $\mathcal{X}\sim\Lambda \oplus M\oplus M$,
		\item $\Char_\Lambda(M)$ divides $\Char_\Lambda(\H^1_\cF(K,\bT)/\Lambda \kappa_1)$ in $\Lambda$.
	\end{enumerate}
\end{theorem}
\begin{proof}
	With~\cref{main} and \cref{mainCGS}, this theorem follows exactly as in~\cite{CGLS} and \cite{CGS}, using specializations at height one primes of $\Lambda$. We mention that in the aforementioned papers the divisibilities hold only after inverting $p$, which is sufficient to their applications, but the computation works for the prime $p$ too, as is readily explained in~\cite[Theorem 2.2.10]{How2004}.
\end{proof}

The next result allows us to construct a non-trivial Kolyvagin system.

\begin{theorem}\label{Koly}
	Assume $f$ has weight $2r$ where $r$ is odd. There exists a Kolyvagin system $\kappa^\Hg \in \KS(\bT,\cF_\Lambda,\sL_f)$ such that $\kappa^\Hg_1 \in \H^1_{\cF_\Lambda}(K,\bT)$ is non-torsion.
\end{theorem}
\begin{proof}
Let $\tilde{\bT}=\tilde{T}\otimes \Lambda$ where $\tilde{T}$ is the canonical Galois stable lattice in the self-dual Galois representation $V_f$ constructed in~\cite[Section 3]{Nek92} (see also the beginning of~\cite[Section 4.2]{CastellaHsieh}). Note that this $\tilde{T}$ is in general different from our chosen $T$ so~\eqref{h1} may not hold for $\tilde{T}$, though they do have the same residual representation. In what follows, we use a tilde to denote relevant modules associated to the canonical lattice. We will follow the construction in~\cite{CastellaHsieh} and~\cite{LV} to show that one can find a non-trivial Kolyvagin system for $\tilde{\bT}$, and then use it to get a non-trivial Kolyvagin system on our chosen $\bT$.
	
The construction of Generalized Heegner cycles for $\tilde{\bT}$ and how they form a Kolyvagin system are essentially given in~\cite[Section~4]{LV}. However, we need to further take into consideration some local uncertainty coming from potential local torsion. We explain how to modify their construction to our situation, following~\cite{CGLS}. As in \textit{op.\ cit.}, we do not assume $p\nmid h_K$ or the `big image' result. The first change appears in~\cite[Theorem 2.4]{LV}, where the vanishing of $\H^0(K_\infty,\tilde{W})$ now follows from our~\cref{H0 of V/T}(iii) (whose proof did not use $\F(\omega)$ is a subrepresentation of $\rho_f$ so it works for $\tilde{T}$ too). Note that for the construction of Kolyvagin systems, the axioms (H.0)--(H.5) are not needed. Their Lemma 4.2 still works if we first replace $p$ with a generator $\pi$ of $\frp$, but this then implies the original result for $p$. Finally, the map induced by restriction\[
\H^1(K[n],\tilde{\bT}/I_n\tilde{\bT})^{\cG(n)}\leftarrow\H^1(K,\tilde{\bT}/I_n\tilde{\bT})\]
may not be an isomorphism since we do not assume~\eqref{h1}. However, we claim that it has kernel and cokernel bounded independent of $n$.

Note that we have the identifications $\H^1(K,\tilde{\bT}/I_n\tilde{\bT}) \iso \varinjlim_i \H^1(K_i, V_f/\tilde{T}[I_n])$ and $\H^1(K[n], \tilde{\bT}/I_n\tilde{\bT}) \iso \varinjlim_i \H^1(K_i[n], V_f/\tilde{T}[I_n])$ by Shapiro's lemma.

As in the proof of~\cite[Proposition~7.17]{DograLeFourn}, there is a canonical isomorphism $\Gal(K_i(V_f/\tilde{T}[I_n])/K_i) \isoto \Gal(\Q(V_f/\tilde{T}[I_n])/\Q)$ (using that $K_i$ and $\Q(V_f/\tilde{T}[I_n])$ are linearly disjoint because of their ramification). Hence the fixed submodule $\H^0(K_i[n], V_f/\tilde{T}[I_n])$ under these groups are contained in $\H^0(K_\infty,V_f/\tilde{T}[I_n]) \subseteq \cO/\pi^N$, which is uniformly bounded independent of $n$ and $i$ by~\cref{H0 of V/T}.

Hence in the Hochschild--Serre spectral sequence \[E_2^{p,q} = \H^p(K_i[n]/K_i, \H^q(K_i[n], V_f/\tilde{T}[I_n])) \Rightarrow E^{p+q} = \H^{p+q}(K_i,V_f/\tilde{T}[I_n]),\] the $E_2^{p,0}$ are of uniformly bounded exponent. Letting $p = 1,2$, the $5$-term exact sequence gives us the desired result for the kernel and cokernel of the restriction.

Thus there is a power $p^N$ for some $N\in\bZ_{\geq 0}$ independent of $n$ (one can take $N=0$ if~\eqref{h1} holds for $\tilde{T}$, for similar reason to that in~\cite[Theorem 4.1.1]{CGLS} combined with~\cref{splitGkinfty}) such that $p^N\tilde{\kappa}_n\in\H^1(K[n],\tilde{\bT}/I_n\tilde{\bT})^{\cG(n)}$ has a unique image $\kappa_n$ in $\H^1(K,\tilde{\bT}/I_n\tilde{\bT})$. Here the $\tilde{\kappa}_n$ still denotes the class in~\cite[eq.~(4.2)]{LV}. The construction in~\cite[Section 4]{LV} now almost gives a Kolyvagin system in our setting, and here we explain some modifications of their arguments. Note that if their arguments work for their $\kappa_n$, they also work for our $\kappa_n$ (except for the differences we explain below) since multiplying by a $p$-power does not affect the local behaviors. 

First, we give a different reasoning for the right vertical map $\H^1(K_v,\tilde{\bT}_n^-)\to\H^1(K[n]_w,\tilde{\bT}_n^-)$ in the diagram (22) in \textit{loc.\ cit.}\ to be injective. In fact, we will directly adopt the ideas of Howard in~\cite[Lemma 2.3.4, Case (iii)]{How2004}. One can identify $\H^1(K_v,\tilde{\bT}_n^-)=\H^1(K_v,\varprojlim_m \Ind^{K_m}_K(\tilde{T}^-/I_n))$ with $\varprojlim_m \H^1(K_v,\Ind^{K_m}_K(\tilde{T}^-/I_n))$ which by Shapiro's lemma is further identified with $\varprojlim_m \H^1(K_{m,v},\tilde{T}^-/I_n)$, where $K_m$ is the $m$-th layer in the anticyclotomic tower $K_\infty/K$ and by abuse of notation we denote a place in $K_m$ above $v$ still by $v$. Similarly, $\H^1(K[n]_w,\tilde{\bT}_n^-)$ can be identified with $\varprojlim_m \H^1(K_m[n]_w,\tilde{T}^-/I_n)$ for $w$ a prime of $K_m[n]$ above $v$. Now from the inflation-restriction exact sequence the kernel of the above map is identified with\[
\varprojlim_m \H^1(K_m[n]_w/K_{m,v},\H^0(K_m[n]_w,\tilde{T}^-/I_n)).\]We first claim that for all large $m$, $\Gal(K_m[n]_w/K_{m,v})\isom \Gal(K_\infty[m]_w/K_{\infty,v})$. This follows from the fact that for all large $m$, $v$ is totally ramified in $K_\infty$ and hence $K_\infty$ and $K_m[n]$ are linearly disjoint (at least when $p\nmid h_K$, otherwise we can increase $m$ to kill the finite difference). Now for any $1\ne n\in\cN$ (the case $n=1$ will be treated separately), we have $p\mid I_n$ by the choice of Kolyvagin primes (see e.g. the definition of $\cL_E$ before~\cite[Theorem 3.2.1]{CGLS}), so $\tilde{T}^-/I_n$ is finite and $\H^0(K_m[n]_w,\tilde{T}^-/I_n)$ must also stabilize for large $m$. In other words, for sufficiently large $m$, we have\[
\H^0(K_m[n]_w,\tilde{T}^-/I_n)=\H^0(K_{m+1}[n]_w,\tilde{T}^-/I_n)=\H^0(K_\infty[n]_w,\tilde{T}^-/I_n)
\]
Consequently, for $n\ne 1$ the above kernel can be identified with\[
\varprojlim_m \H^1(K_\infty[n]_w/K_{\infty,v},\H^0(K_\infty[n]_w,\tilde{T}^-/I_n)).\]
Since the inverse limit is with respect to the corestriction on the above $H^0$ groups (hence is identified with the norm map), and $G_{K_m[n]_w}$ acts on $\H^0(K_{m+1}[n]_w,\tilde{T}^-/I_n)$ trivially by the above identification, we can interpret the inverse limit as one with respect to multiplication by $p$ on the fixed module \[\H^1(K_\infty[n]_w/K_{\infty,v},\H^0(K_\infty[n]_w,\tilde{T}^-/I_n)).\]
Since the module is finite as a consequence of the finiteness of both $\Gal(K_\infty[n]_w/K_{\infty,v})$ and $\H^0(K_\infty[n]_w,\tilde{T}^-/I_n)$, the inverse limit is trivial. 

Second, the cokernel of the trace map $\H^0(K[n]_w,\tilde{\bA}^-)\to\H^0(K_v,\tilde{\bA}^-)$ in~\cite[Lemma 4.12]{LV} can fail to be trivial (their `$\bA$' is our `$M_f$'). In other words, their $\kappa_n$ may not lie in the desired Selmer group $\H^1_{\cF_\Lambda(n)}(K,\tilde{\bT}/I_n\tilde{\bT})$. For example, when $\H^0(K_v,\tilde{A}^-)\ne 0$ (e.g. when $\F(\tilde{\mathbf{1}})$ is a non-split subrepresentation of $\ol\rho_f$ or when the sequence $0\to\F(\tilde{\omega})\to \ol\rho_f\to\F(\tilde{\mathbf{1}})\to 0$ is locally split, or more generally in the `anomalous case'), by local Euler characteristic formula one can show that $\H^1(K_v,\tilde{A}^-)[\frp]$ and hence $\H^1(K_v,\tilde{A}^-)\ne 0$. This non-vanishing phenomenon can occur even in the residually irreducible case (but it can be avoided if one assumes $a_p\ne 1\pmod{\frp}$). 

The way to fix it is to use a fixed $p$-power independent of $n$ to kill the cokernel so that, after multiplying the classes $\kappa_n$ by this $p$-power, they belong to the Selmer groups and form a Kolyvagin system. We must show the existence of such $p$-powers. To bound the cokernel of the trace map, it suffices to bound the codomain $\H^0(K_v,\tilde{\bA}^-)$. 

We argue that $\H^0(K_v,\tilde{\bA}^-)$ is finite similarly as what we do in~\cref{H0 of V/T}. When the weight of $f$ is $2r>2$, our self-dual twist has the effect that the action of $G_{\bQ_p}$ on $V^-$ is an unramified character multiplied by a non-trivial power of the cyclotomic character. Since $K_\infty$ is the anticyclotomic $\bZ_p$ extension, $\H^0(K_{\infty,v},V^-)=0$ and hence $\H^0(K_v,\tilde{\bA}^-)\iso \H^0(K_{\infty,v},\tilde{A}^-)$ is finite. When $f$ has weight $2$, the finiteness follows from passing to abelian varieties.

Thus we have shown that $\H^0(K_v,\tilde{\bA}^-)$ is finite (necessarily independent from $n$) and so is the cokernel of the trace map. Hence there is an integer $t\in\bZ_{\geq 0}$ such that $p^t$ annihilates the cokernel. It is then not hard to see that $p^t\kappa_n\in\H^1_{\cF_\Lambda(n)}(K,\tilde{\bT}/I_n\tilde{\bT})$ in the notation of~\cite{LV} for all $1\ne n\in\cN$. For $n=1$, the same result follows from the fact that the image of $\cH_\infty$ in $\H^1(K_v,\tilde{\bT})$ lies in $\H^1_\ord(K_v,\tilde{\bT})$ (see Lemma 4.12 in \textit{loc. cit.}). Moreover, the Kolyvagin system relation is also preserved by multiplying every class with a fixed $p$-power. We define $\kappa^{\widetilde{\Hg}}_n$ to be the modification of $p^t\kappa_n$ (one can take $t=0$ if the trace map is indeed surjective. For example, this is the case if one assumes $\phi,\psi|_{G_p}\ne\mathbf{1},\omega$) in~\cite[Theorem 4.13]{LV}. Here we assume $r$ is odd so the factor $(-1)^{r-1}$ in~\cite[(K2)]{CastellaHsieh} disappears and the factors on both sides differ at most by a sign. Then $\kappa^{\widetilde{\Hg}}$ is a Kolyvagin system for $\tilde{T}$.

That the class is non-trivial follows from~\cite[Proposition~3.9, Theorem~4.9]{CastellaHsieh}. Indeed, the left hand side of the equation in Theorem 4.9 in \textit{op.\ cit.}\ is nonzero by Proposition~3.9 there, so the right hand side is nonzero. Therefore the class $z_f\ne 0$. But this is just $\cores_{K[1]/K}\beta_0[1]$ in~\cite[Proposition 4.9]{LV}. Therefore from its construction it's easy to see the first Kolyvagin class $\tilde{\kappa}_1$ is non-trivial. Since $\kappa_1$ is the image of $p^N\tilde{\kappa}_1$ in $\H^1(K,\tilde{\bT})$, $\kappa_1$ and $\kappa^{\widetilde{\Hg}}_1=p^t\kappa_1$ are also non-trivial.

We now explain how to construct a Kolyvagin system on our chosen $T$ which satisfies~\eqref{h1}. Let $T'=T\cap \tilde{T}$ which has finite index in both $T$ and $\tilde{T}$. Choose an $m\in\bZ$ with $p^m\tilde{T}\subset T'$. Now the map $\tilde{T}\to p^m\tilde{T}\to T'$ induces a map $\tilde{\bT}\to\bT'$. Now the classes $\kappa^{\widetilde{\Hg}}_n\in\H^1_{\cF(n)}(K,\tilde{\bT}/I_n\tilde{\bT})$ are mapped to their images $\kappa^{\Hg'}_n\in\H^1(K,\bT'/I_n\bT')$ and it is not hard to check that in fact $\kappa^{\Hg'}_n\in\H^1_{\cF(n)}(K,\bT'/I_n\bT')$ by functoriality. Moreover, the Kolyvagin system relation is also preserved, i.e., $\kappa^{\Hg'}_n$ is a Kolyvagin system for $T'$. Now the map $T'\to T$ maps $\kappa^{\Hg'}_n$ to $\kappa^{\Hg}_n\in\H^1_{\cF(n)}(K,\bT/I_n\bT)$ which is a Kolyvagin system for $T$.

Finally, that $\kappa^\Hg_1$ is non-torsion follows from the fact that $\H^1_{\cF_\Lambda}(K,\bT)$ is torsion free (see for example~\cite[Theorem 2.2.9]{How2004} where the condition $\H^0(K_\infty,\ol\rho_f)=0$ follows from~\eqref{h1} and~\cref{splitGkinfty}).\end{proof}

\begin{remark}\label[remark]{rodd}
	The above theorem is the only place where the assumption that $r$ is odd is used. It is not necessary if the $B_1$ in~\cite[eq. (7.2)]{CastellaHsieh} is zero, e.g., when $\rho_f^*$ is surjective.
\end{remark}

We are ready to prove the two Iwasawa Main Conjectures similar to those in~\cite{CGLS}. Here in the statement one needs to work with a seemingly different Kolyvagin class `$\kappa_\infty$' coming from the equivalence between the Main conjectures, but as is noted in~\cite[Remark 4.1.3]{CGLS}, $\kappa_\infty$ and $\kappa_1$ generate the same $\Lambda$-submodule in $\H^1_{\cF_\Lambda}(K,\bT)$.

\begin{theorem}\label[theorem]{IMC}
	Assume $f$ has weight $2r$ where $r$ is odd. Assume that $p=v\ol v$ splits in $K$ and $\H^0(K,\ol\rho_f)=0$ Then the following statements hold:\begin{enumerate}
		\item[(IMC1)] Both $\H^1_{\cF_\Lambda}(K,\bT)$ and $\cX=\H^1_{\cF_\Lambda}(K,M_f)^\vee$ have $\Lambda$-rank one, and the equality \begin{equation*}
			\Char_\Lambda(\cX_{\tors})= \Char_\Lambda(\H^1_{\cF_\Lambda}(K,\bT)/\Lambda\kappa_{\infty})^2
		\end{equation*}
	holds in $\Lambda_{ac}$.
		\item[(IMC2)] Both $\H^1_{\cF_{\nr}}(K,\bT)$ and $\fX_f=\H^1_{\cF_{\nr}}(K,M_f)^\vee$ are $\Lambda$-torsion, and the equality\begin{equation*}
			\Char_\Lambda(\fX_f)\Lambda^{\nr}=(\cL_f)
		\end{equation*}
		holds in $\Lambda^{\nr}$.
	\end{enumerate}
\end{theorem}
\begin{proof}
	We first treat the $t=N=0$ case in~\cref{Koly}. We begin by noting that (IMC1) and (IMC2) are equivalent; in fact, an either side divisibility in one of them is equivalent to a corresponding divisibility in the other (see~\cite[Prop.~4.2.1]{CGLS} and~\cite[Thm.~5.2]{BCK21} combined with~\cref{splitGkinfty}. See also~\cite[Appendix A]{Cas17}). Here we should mention that the equivalence is done with the Greenberg Selmer groups, but they generate the same characteristic ideals as the unramified Selmer group do. \cref{HeegMC} shows one divisibility of (IMC1), which translates to a divisibility of (IMC2). Using the equality between the algebraic and analytic Iwasawa invariants (see~\cref{algmain} and~\cref{anacomp}), the divisibility in (IMC2) is turned into an equality. Hence (IMC1) also holds.
	
	Now if $t>0$ or $N>0$, the proof in~\cite[Appendix A]{Cas17} would yield a new equivalence involving a fixed $p$-power:\begin{enumerate}
		\item[(IMC1')] Both $\H^1_{\cF_\Lambda}(K,\bT)$ and $\cX=\H^1_{\cF_\Lambda}(K,M_f)^\vee$ have $\Lambda$-rank one, and the divisibility \begin{equation*}
			\Char_\Lambda(\cX_{\tors})\supset \Char_\Lambda(\H^1_{\cF_\Lambda}(K,\bT)/\Lambda\cdot(p^{t+N}\kappa_{\infty}))^2
		\end{equation*}
		holds in $\Lambda_{ac}$.
		\item[(IMC2')] Both $\H^1_{\cF_{\nr}}(K,\bT)$ and $\fX_f=\H^1_{\cF_{\nr}}(K,M_f)^\vee$ are $\Lambda$-torsion, and the divisibility\begin{equation*}
			\Char_\Lambda(\fX_f)\Lambda^{\nr}\supset(p^{t+N}\cL_f)
		\end{equation*}
		holds in $\Lambda^{\nr}$.
	\end{enumerate}
(the same result holds for the opposite divisibilities, but we will not need it). In our situation,~\cref{HeegMC} applied to the Kolyvagin system $\kappa^\Hg_n=p^t\kappa_n$ (again noting the relation between $\kappa_1$ and $\kappa_\infty$) shows the divisibility in (IMC1'). Hence we obtain (IMC2') from the equivalence. However, since we know both $\fX_f$ and $\cL_f$ have vanishing $\mu$-invariants and equal $\lambda$-invariants (all computed independently), we obtain from (IMC2') exactly the same equality in (IMC2). Now from the equivalence between (IMC1) and (IMC2), we obtain (IMC1) as before.
\end{proof}
\begin{remark}\label{allowtor}
	As is explained in the beginning of~\cref{sec:selmer-groups-of-rhof}, the above theorem still holds without assuming $\H^0(K,\ol\rho_f)=0$.
\end{remark}

As a corollary of our anticyclotomic Main Conjectures, we also obtain the following Mazur's Main Conjecture for elliptic curves unconditionally in the Eisenstein case using the main results of~\cite{CGS}. Let $\Lambda_\bQ\defeq\bZ_p\llbracket \Gal(\bQ^\infty/\bQ)\rrbracket$ be the cyclotomic Iwasawa algebra over $\bQ$. Let $\fX_\ord(E/\bQ_\infty)$ be the Pontryagin dual of the $p$-primary Selmer group $\Sel_{p^\infty}(E/\bQ_\infty)$ for $E$ and let $\cL_p^\MSD(E/\bQ)$ be the Mazur--Swinnerton-Dyer $p$-adic $L$-function.

\begin{theorem}
	Let $E/\bQ$ be an elliptic curve, and let $p>2$ be a prime of good reduction for $E$. Suppose that $p$ is Eisenstein. Then $\fX_\ord(E/\bQ_\infty)$ is $\Lambda_\bQ$-cotorsion with\[
	\Char_{\Lambda_\bQ}(\fX_\ord(E/\bQ_\infty))=\cL_p^\MSD(E/\bQ),
	\]and hence Mazur's Main Conjecture holds.
\end{theorem}
\begin{proof}
	This is the main result of~\cite{CGS}. Their technical assumption on the local behaviors of the characters in the semisimplification of $E[p]$ can be removed if one replaces the appeal to~\cite[Theorem 4.2.2, Corollary 4.2.3]{CGLS} by our~\cref{IMC} and~\cref{allowtor} (specialized to weight $2$ case).
\end{proof}

\section{Proofs of the $p$-converse and the $p$-part of BSD formula}\label{4}
In this section, we discuss several applications of our results in the previous sections, whose proofs are direct generalizations of those in~\cite[§\,5]{CGLS}. We will only consider the weight $2$ cases.
\subsection{The $p$-converse theorem}
\begin{theorem}\label[theorem]{p converse}
	Let $A/\Q$ be a simple RM abelian variety associated to a newform $f$ and $\frP \mid p > 2$ a prime ideal of its endomorphism ring $\cO=\End(A)$ of good reduction such that $\rho_{A,\frP}$ is reducible.
	
	Then
	\[
		\corank_\cO \Sel_{\frP^\infty}(A/\Q) = r \in \{0,1\} \implies \rk_\cO A(\Q) = r_\an(f) = r,
	\]
	and $\Sha(A/\Q)[\frP^\infty]$ is finite.
\end{theorem}
\begin{proof}
	The proof of~\cite[Theorem 5.2.1]{CGLS} works if one replaces the reference [Mon96, Theorem~1.5] there by~\cite[Theorem~C]{NekovarTame} and replaces~\cite[Corollary 4.2.3]{CGLS} by our~\cref{IMC}\,(IMC1).
\end{proof}

The $p$-converse theorem has the following applications to Goldfeld's conjecture in quadratic twist families with a $3$-isogeny where $3$ is a prime of good or bad multiplicative reduction. We thank Ari Shnidman for communicating and explaining them to us.
\begin{remark} \label[remark]{Goldfeld conjecture}
	\begin{enumerate}[(i)]
		\item Better proportions of quadratic twists of (algebraic and analytic) rank~$1$ in~\cite[Theorem~2.5]{BKLOS} for a fixed elliptic curve, in particular a lower bound of $\frac{3}{4}\cdot\frac{5}{12}=\frac{5}{16}=31.25\,\%$ in the most advantageous cases (for example, when it's the curve having Cremona label $19a3$). Note that the proportion claimed in~\cite[Remark~5.2.4]{CGLS} is only for a subfamily (corresponding to good reduction and with their assumptions on the isogeny character) of the twists. Indeed, for the twists $E_d$ of the curve $E=19a3$, the results in \textit{loc.\ cit.} only apply when $d\equiv 2 \pmod{3}$ whereas our result also covers the cases when $d\equiv 1 \pmod{3}$, covering at least $\frac{3}{8}+\frac{3}{8}=\frac{3}{4}$ of the proportion of the $3$-Selmer rank $1$ twists. The twists with $d\equiv 0 \pmod{3}$ correspond to bad additive reduction and are not covered by our $p$-converse theorems. However, potentially good ordinary primes are studied in our subsequent work~\cite{KY24} and combined with our work, we can get the full lower bound $\frac{5}{12}=41.67\%$ of rank $1$ twists for the above curve $19a3$.

		\item Our result seems not to improve the bounds for rank~$0$ in~\cite{BKLOS} in the situation of~(i), but it allows one to quantify how often the BSD rank conjecture holds. Namely, their results give us that at least $\frac{1}{4}$ of the twists have $3$-Selmer rank $0$, hence at least $\frac{3}{4}\cdot \frac{1}{4}=\frac{3}{16}$ of them have analytic rank $0$. Together with~(i), $\frac{3}{16} + \frac{5}{16} = \frac{1}{2}$ of the twists satisfy the BSD rank conjecture.	
		\item Similar improvements apply to quadratic twist families of abelian varieties $A$ of $\GL_2$-type over $\Q$ with a $3$-isogeny, as in~\cite[Theorem~1.4, Corollary~7.2]{BruinFlynnShnidman} and~\cite[Theorem~1.5]{ShnidmanIMRN}. For example, in the context of the latter result (where the abelian varieties all have good reduction at $3$), our~\cref{p converse} gives an unconditional lower bound of $\frac{5}{16}$ for the proportion of twists with rank $1$ over the RM ring.
	\end{enumerate}
\end{remark}

\subsection{The $p$-part of BSD formula}\label{pBSD}

\begin{theorem} \label[theorem]{BSD E}
	Let $E/\Q$ be an elliptic curve and $p > 2$ a prime of good reduction. Assume that $E$ admits a cyclic $p$-isogeny with kernel $C=\F_p(\phi)$ for some character $\phi\colon G_\bQ\to \F_p^\times$ (equivalently, $E[p]$ is reducible). Assume that $r_\an(E) \in \{0,1\}$. 
	
	Then the $p$-part of the BSD formula holds for $E/\Q$, i.e.,
	\[
	\ord_p\Big(\frac{L^{*}(E/\Q,1)}{\Omega_E \cdot \Reg_{E/\Q}}\Big) = \ord_p\Big(\frac{\Tam(E/\Q)\#\Sha(E/\Q)[p^\infty]}{(\#E(\Q)_\tors)^2}\Big).
	\]
	Here, $L^*(E/\Q,1)$ denotes the leading Taylor coefficient of $L(E/\Q,s)$ at $s = 1$.
\end{theorem}
\begin{proof}
	This follows from~\cite[Theorem 5.3.1]{CGLS}, where the use of Theorem 4.2.2 in \textit{op.\ cit.}\ is replaced by our~\cref{IMC}\,(IMC2). We remark that the assumption that the character is either ramified at $p$ and even, or unramified at $p$ and odd coming from~\cite{GV00} is now removed by~\cite{CGS} in both the rank $0$ and rank $1$ formulae, and the assumption that $\phi|_{G_p}\ne\mathbf{1},\omega$ is removed by us.
	
	In particular, we take $k=2$, $r=j=0$ and $\chi=\bN_K$ in~\cite[Theorem 5.13]{BDP13}. Note that $L_p(f,\bN_K)$ in \textit{loc.\ cit.}\ corresponds $\cL_E(0)$ because it corresponds to the value $L(f/K,\bN_K^{-1},0)=L(f/K,\mathbf{1},1)=L(f/K,1)$ (see~\cite[section 4]{BDP13}).
	
	Note that to make the arguments in~\cite[Theorem 5.3.1]{CGLS} work, we need an anticyclotomic control theorem which allows torsion submodules. This is done in Appendix~\ref{appB}. The formula in~\cite[Theorem 5.1.1]{CGLS} then becomes
	\begin{equation*}
			\#\bZ_p/f_{E}(0)=\#\Sha(E/K)[p^\infty]\cdot\Bigg(\frac{\#(\bZ_p/(\frac{1-a_p+p}{p})\cdot\log_{\omega_E}P)}{[E(K)_{/\tors}:\bZ\cdot P]_p\cdot\#(E(K)_{\tors})_p}\Bigg)^2 \times\prod_w c_w(E/K)_p,
	\end{equation*}
where $x_p\coloneq p^{\ord_p(x)}$ denotes the $p$-part of a rational number $x$. Alternatively, one could continue to work with an isogenous curve which has no torsion, and use the invariance of BSD formulae to get the same results with torsion.
We remark that in the control theorem from~\cite{JSW2017}, their anticyclotomic Selmer group is actually the `Greenberg Selmer group' for $f$ (which agrees with the unramified Selmer group in~\cite{CGLS}). However, since $\ker(\res_{M_f})$ is finite by~\cref{kerresf}, these two Selmer groups generate the same characteristic ideals. 

Together with the above formula coming from the control theorem and the fact that $E(K)[p^\infty]=E(\Q)[p^\infty]\oplus E^K(\Q)[p^\infty]$ (because $p>2$), equation (5.7) in~\cite{CGLS} becomes\begin{align*}
&\ord_p\Big(\frac{L'(E,1)}{\Reg(E/\Q)\cdot\Omega_E\cdot\prod_l c_l(E/\Q)}\Big)-\ord_p(\#\Sha(E/\bQ))+2\ord_p(E(\bQ)_\tors)\\
&=-\Big(\ord_p\Big(\frac{L(E^K,1)}{\Omega_{E^K}\cdot\prod_l c_l(E^K/\Q)}\Big)-\ord_p(\#\Sha(E^K/\Q))+2\ord_p(E^K(\Q)_\tors)\Big)
\end{align*}
where the right-hand side vanishes by~\cite[Theorem 5.1.4]{CGLS}, a result of Greenberg--Vatsal.
\end{proof}

\section{Proof of the Iwasawa Main Conjecture with multiplicative reduction}\label{5}
Let $f \in S_2(\Gamma_0(N))$ be a newform of weight $2$ with coefficient field $\Q(f)$. Let $F$ be a finite extension of the completion of $\Q(f)$ at a chosen prime above $p\ne 2$ and let $\cO$ be its ring of integers. Let $\frp$ be the maximal ideal of $\cO$ and let $\F=\cO/\frp$ be the residue field of $\cO$. Let $\Lambda=\cO\llbracket T \rrbracket$ and $\Lambda^{\nr}=\Lambda\hat{\otimes}_{\bZ_p}\bZ_p^{\nr}$. Let $\ol \rho_f$ be the residual representation attached to $f$ and assume $\ol \rho_f$ is reducible. Then there are characters $\phi,\psi$ such that there is an exact sequence
\begin{equation}\label{badextension}
0\to \F(\phi)\to \ol\rho_f\to\F(\psi)\to 0.
\end{equation} Further assume that $p\mid\mid N$, so $f$ has bad multiplicative reduction at $p$. In this section, we prove the Iwasawa Main Conjectures for such forms using Hida families, following the ideas in~\cite{Skinner}. 

\subsection{Hida arguments}\label{5.1}

As in the good reduction case, define the unramified Selmer group for $f$ by {\small\[
	\H^1_{\cF_{\nr}}(K,M_f)= \ker\Big\{\H^1(K^\Sigma/K, M_f) \xrightarrow{\res_{\nr,f}} \prod_{w\in\Sigma, w \nmid p}\H^1(I_w,M_f)^{G_w/I_w} \times \H^1(I_{\ol v}, M_f)^{G_{\ol v}/I_{\ol v}}\Big\}.
	\]} 

 For a finite set $\Sigma$ of places of $K$ which contains $\infty$ and the primes above divisors of $N$, and such that the finite places in $\Sigma$ are all split in $K$, let $S=\Sigma\setminus \{v,\bar{v},\infty\}$ and define the $S$-imprimitive Selmer group for $f$ by
 \[
 \H^1_{\cF^S_{\nr}}(K,M_f)= \ker\Big\{\H^1(K^\Sigma/K, M_f) \xrightarrow{\res^S_{\nr,f}}  \H^1(I_{\ol v}, M_f)^{G_{\ol v}/I_{\ol v}}\Big\}.
 \]
 Let $\fX_f:=\H^1_{\cF_{\nr}}(K,M_f)^\vee$ and $\fX_f^S:=\H^1_{\cF^S_{\nr}}(K,M_f)^\vee$ denote the dual Selmer groups. 
 
 According to~\cite[Section 3.1]{Skinner}, replacing $F$ by a finite extension if necessary, for each integer $m>0$ we get
 \begin{enumerate}[(a)]
 	\item[(a)] a newform $f_m\in S_{k_m}(\Gamma_0(N))$ with $\bQ(f_m)\subset F,\ k_m>2$ and $k_m\equiv2 \pmod{p-1}$ and such that $a_p(f_m)\in\cO^\times$;
 	\item[(b)] an equality of ideals $(\cL^S_f,\frp^m)=(\cL^S_{f_m},\frp^m)\subseteq\Lambda^{\nr}$;
 	\item[(c)] $\ol\rho_f\cong \ol\rho_{f_m}$ is reducible. In particular, the characters appearing in their semisimplifications are the same (both globally and locally);
 	\item[(d)] $\fX^S_f(f_m)$ is a torsion $\Lambda$-module and Char$(\fX^S_{fm})\Lambda^\nr=(\cL^S_{f_m})\subseteq\Lambda^{\nr}$;
 	\item[(e)] $\fX^S_{f_m,\free}$ has no nonzero finite $\Lambda$-submodules, so Fitt$(\fX^S_{f_m,\free})=\Char(\fX^S_{f_m,\free})$. Here $\fX^S_{f_m,\free}$ denotes the $\Lambda$-free part of $\fX^S_{f_m}$ and Fitt denotes the $\Lambda$-Fitting ideal. The same is true for $f$.
 \end{enumerate}
(a) and (c) follow from Hida theory. Similarly as in \textit{op.\ cit.}, (b) comes from a $2$-variable $p$-adic $L$-function (adding a `weight' variable) for the Hida families constructed in~\cite[Theorem 2.11]{Cas20} which works for Eisenstein primes. (d) follows from~\cref{IMC}\,(IMC2) for good ordinary forms $f_m$ in the previous sections, the surjectivity of global-to-local maps (\cref{remark on Selmer groups} and~\cref{surj for Self}) and the equality of local factors appearing in the definition of the imprimitive Selmer groups and $p$-adic $L$-functions. Note that by~\cref{Koly} we only know the Main Conjectures for a subfamily with weights $k_m\equiv 2\pmod{4}$. But such subfamily is always infinite so the following limiting argument applies. (e) comes from basic properties of Fitting ideals.
\begin{lemma}
	The $\mu$-invariants of $\fX^S_f$ and $\fX^S_{f_m}$ are $0$.
\end{lemma}
\begin{proof}
	This is essentially proved in~\cref{imprimlambda} since the proof only relies on the extension~\eqref{badextension}. Note that the characters can be arbitrary.
\end{proof}

By (c) above, for every $f_m$, we also have an extension\[
0\to \F(\phi)\to \ol\rho_{f_m}\to\F(\psi)\to 0.
\]
\cref{algmain} then implies that $\lambda(\fX^S_f)=\lambda(\fX^S_{f_m})$. On the other hand, since the $\mu$-invariants are $0$, $\lambda(\fX^S_{f})$ (resp.\ $\lambda(\fX^S_{f_m})$) determines $\fX^S_{f,\free}/\frp^m$ (resp.\ $\fX^S_{f_m,\free}/\frp^m$). Thus there is an equality of ideals\[
(\Fitt(\fX^S_{f,\free}),\frp^m)=(\Fitt(\fX^S_{f_m,\free}),\frp^m).\] Then following the arguments in~\cite[Section 3.1]{Skinner}, we have the following result.
\begin{lemma}
	$\Char(\fX^S_f)\Lambda^{\nr}=(\cL^S_f)\subseteq\Lambda^{\nr}$.
\end{lemma}
\begin{proof}
	We have \begin{align*}
		(\Char(\fX^S_f)\Lambda^\nr,\frp^m)&=(\Char(\fX^S_{f,\free})\Lambda^\nr,\frp^m)\\
		&=(\Fitt(\fX^S_{f,\free})\Lambda^\nr,\frp^m)\\
		&=(\Fitt(\fX^S_{f_m,\free})\Lambda^\nr,\frp^m)\\
		&=(\Char(\fX^S_{f_m,\free})\Lambda^\nr,\frp^m)\\
		&=(\Char(\fX^S_{f_m})\Lambda^\nr,\frp^m)\\
		&=(\cL^S_{f_m},\frp^m)\\
		&=(\cL^S_f,\frp^m).
	\end{align*}
The result then follows from taking limit with respect to $m$ as in~\cite[section 3.1]{Skinner}.
\end{proof}
\begin{theorem}\label{imc mult}
	$\Char(\fX_f)\Lambda^{\nr}=(\cL_f)\subseteq\Lambda^{\nr}.$
\end{theorem}
\begin{proof}
	Similar to the good reduction case, this result follows from the comparison between primitive and imprimitive Selmer groups (see~\cref{algmain}, whose proof still works in the multiplicative reduction setting) and the comparison between primitive and imprimitive $p$-adic $L$-functions (see~\cref{anacong}). We mention that the same arguments as in~\cite[Prop. 37]{Kri16} (which has an assumption $p\nmid N$) still work since Kriz's computation of the congruence is with the $p$-depletion which is identified with the $((p^2)^{(0)})$-stabilization, which allows $p$ in the level.
\end{proof}

\subsection{A $p$-converse theorem}
Similar as in the good reduction case, we also obtain a $p$-converse theorem for multiplicative reduction. Via an explicit reciprocity law, one can argue as in~\cite{Cas24} to show that the BDP Main Conjecture proved in the last section is equivalent to a Heegner Point Main Conjecture which would then yield a $p$-converse theorem.

As in~\cite{Cas24}, one can consider a Heegner class $\bz_\infty^*\in \H^1_{\cF_\Lambda}(K,\bT)$ which is either a standard $\Lambda$-adic Heegner class $\bz_\infty$ (denoted as $\kappa_\infty$ by us in section $3$) in the non-split multiplicative case, or a derived Heegner class $\bz_\infty'$ defined by the relation $\bz_\infty=(\gamma-1)\bz_\infty'$ in the split multiplicative case. Here $\gamma$ is a topological generator of $\Gal(K_\infty/K)$, and this $(\gamma-1)$ factor reflects a trivial zero of the $p$-adic $L$-function. We have the following Heegner Point Main Conjecture for multiplicative primes (the definitions are the same as those for good ordinary primes).

\begin{theorem}\label{multHg}
	Let $f \in S_2(\Gamma_0(N))$ be a newform of weight $2$ and assume that $p\mid\mid N$ is an Eisenstein prime for $f$. Then both $\H^1_{\cF_\Lambda}(K,\bT)$ and $\cX=\H^1_{\cF_\Lambda}(K,M_f)^\vee$ have $\Lambda$-rank one, and the equality \begin{equation*}
		\Char_\Lambda(\cX_{\tors})= \Char_\Lambda(\H^1_{\cF_\Lambda}(K,\bT)/\Lambda\bz_{\infty}^*)^2
	\end{equation*} holds.
\end{theorem}
\begin{proof}
		This follows from the explicit reciprocity law~\cite[Theorem 2.2, Corollary 2.3]{Cas24} for both split and non-split cases, which can be applied to prove the equivalence of this Heegner Point Main Conjecture to~\cref{imc mult}. In showing the equivalence using~\cite[Proposition 3.2]{Cas24} (also~\cite[Theorem 5.1]{BCK21}), we could again first assume $\H^0(K,\ol\rho_f)=0$ (see also~\cref{splitGkinfty}) and deduce the general case using a combination of Ribet's lemma and the invariance of the Main Conjectures.
		
		It should be mentioned that in this situation we do not need to consider the $t>0$ case from~\cref{Koly}, since that is only needed to get one divisibility from the Kolyvagin system argument but not needed for the equivalence. 
\end{proof}

\begin{corollary}\label[corollary]{p converse mult}
	Let $A/\Q$ be a simple RM abelian variety associated to a newform $f$ and $\frP \mid p > 2$ a prime ideal of its endomorphism ring $\cO=\End(A)$ of multiplicative reduction such that $\rho_{A,\frP}$ is reducible.
	
	Then
	\[
	\corank_\cO \Sel_{\frP^\infty}(A/\Q) = r \in \{0,1\} \implies \rk_\cO A(\Q) = r_\an(f) = r,
	\]
	and $\Sha(A/\Q)[\frP^\infty]$ is finite.
\end{corollary}
\begin{proof}
This follows from the proof of~\cite[Theorem 1.1]{Cas24} for multiplicative primes and a control theorem of Greenberg (see~\cite[Theorem 3.4]{Cas24}), where the appeal to Theorem 1.3 from \textit{loc. cit.} is replaced by~\cref{multHg}. Note that their arguments also work in the Eisenstein case.
\end{proof}

\appendix
\section{A Kolyvagin system argument} \label{sec:appendix}

As in~\cite{How2004}, by a coefficient ring $R$ we mean a complete, Noetherian, local ring with finite residue field of characteristic $p$. If $L$ is a perfect field, we denote by $\Mod_{R,L}$ the category of finitely generated $R$-modules equipped with continuous, linear actions of $G_L$, assumed to be unramified outside of a finite set of primes in the case where $L$ is a global field. 

In this subsection, $R^{(k)}$ is assumed to be an Artinian principal idea ring of length $k$. Denote its residue field by $\F$. Let $T^{(k)}$ be an object of $\Mod_{R^{(k)},G_K}$. Fix a generator $\pi^{(k)}$ of the maximal ideal $\fm^{(k)}$ of $R^{(k)}$. We will ignore the indexes and simply write $\pi$ and $\fm$ when the $k$ are clear from the contexts. Let $T^{(k)*}=\Hom(T^{(k)},R(1))$ and fix a Selmer structure $\cF$ on $T^{(k)}$. Let $\cF^*$ denote the dual Selmer structure on $T^{(k)*}$. In all what follows we assume that $(T^{(k)},\cF)$ and $(T^{(k)*},\cF^*)$ satisfy the hypotheses $\H.0$ and $\H.3$--$\H.5$ in \textit{op.\ cit.} §\,1.3. In particular, we do \emph{not} assume $\H^0(K,\ol\rho_f)=0$ (assumption~(h1) in~\cite[§\,3.2]{CGLS}), and $\H.2$ is not necessary for our purpose.

The main result of this appendix is~\cref{square}, which follows from~\cref{even V} and the arguments in~\cite{How2004}.

We will make the following assumption about stabilizing $\H^0$ groups throughout this section:
\begin{assumption}\label[assumption]{stablizing H0}
	There is an absolute constant $N$ such that $\H^0(K,T^{(k)})\isom R/\pi^{k}$ when $k\leq N$, and $\H^0(K,T^{(k)})\isom R/\pi^{N}$ for all $k>N$.
\end{assumption}
Note that by a version of~\cref{H0 of V/T}\,(i) in the case where $\H^0(K,\ol\rho_f)\ne0$, $\H^0(K,V_f/T_f)$ is still finite (since $\H^0(K,V_f)=0$ because $\rho_f$ is irreducible), this assumption is satisfied when $T^{(k)}=T_f/\pi^k T_f$ where $\pi$ is the uniformizer of $R$. For the exposition, we often denote $\H^i(K,-)$ by $\H^i(-)$. 

Following~\cite[Lemma~1.3.3]{How2004} and~\cite[Lemma~3.5.4]{MR04}, for $i\leq k$, we would like to compare $\H^1_\cF(K,T^{(k)}/\fm^iT^{(k)})$ and $\H^1_\cF(K,T^{(k)}[\fm^i])$ with $\H^1_\cF(K,T^{(k)})[\fm^i]$. Note that the first two groups are isomorphic (see~\cite[Remark~1.1.4]{How2004}), we will use $T^{(k)}/\fm^iT^{(k)}$ and $T^{(k)}[\fm^i]$ interchangeably without further explanation in cohomology groups. In~\cite{How2004}, the vanishing of several $\H^0$ groups forces the natural maps between the above groups to be isomorphisms. Since we allow non-trivial $\H^0$ groups, we need to study the kernels and cokernels more carefully.
\begin{lemma}\label[lemma]{K and C for H1}
	
	Let $k \geq i$. The injection $\pi^{k-i}\colon T^{(k)}/\fm^iT^{(k)}\hookrightarrow T^{(k)}$ induces morphisms
	\begin{align*}
		[\pi^{k-i}]\colon &(\H^1(K,T^{(k)}[\fm^i])\isoto)\ \H^1(K,T^{(k)}/\fm^iT^{(k)})\to \H^1(K,T^{(k)})[\fm^i]\\
		[\pi^{k-i}]\colon &(\H^1_\cF(K,T^{(k)}[\fm^i])\isoto)\ \H^1_\cF(K,T^{(k)}/\fm^iT^{(k)})\to \H^1_\cF(K,T^{(k)})[\fm^i]
	\end{align*}
	with finite kernel $K^{(i)}$ and cokernel $C^{(i)}$ and finite kernel $K_\cF^{(i)}$ and cokernel $C_\cF^{(i)}$, respectively.
	
	One has $\vert K^{(i)}\vert=\vert K_\cF^{(i)}\vert=\vert C^{(i)}\vert=\frac{\vert\H^0(K,T^{(k)}/\fm^iT^{(k)})\vert\vert\H^0(K,T^{(k)}/\fm^{k-i}T^{(k)})\vert}{\vert\H^0(K,T^{(k)})\vert}\geq \vert C_\cF^{(k)}\vert$. 
\end{lemma}
\begin{proof}
	We want to point it out that this is essentially already proved in~\cite[Lemma~3.5.4]{MR04}. Cohomology of the following two exact sequences
	\begin{enumerate}
		\item $0\to T^{(k)}/\fm^iT^{(k)}\xrightarrow{\pi^{k-i}}T^{(k)}\xrightarrow{p_1} T^{(k)}/\fm^{k-i}T^{(k)}\to 0$
		\item $0\to T^{(k)}/\fm^{k-i}T^{(k)}\xrightarrow{\pi^i}T^{(k)}\xrightarrow{p_2} T^{(k)}/\fm^iT^{(k)}\to 0$
	\end{enumerate}
	gives $K^{(i)}\iso\frac{\H^0(K,T^{(k)}/\fm^{k-i}T^{(k)})}{p_1(\H^0(K,T^{(k)}))}$ and $ C^{(i)}\iso\frac{\H^0(K,T^{(k)}/\fm^iT^{(k)})}{ p_2(\H^0(K,T^{(k)}))}$. (In particular, they are finite, so we can speak about their order.) But $\vert p_1(\H^0(K,T^{(k)}))\vert=\frac{\vert\H^0(T^{(k)})\vert}{\vert\H^0(T^{(k)}/\fm^iT^{(k)})\vert}$ and $\vert p_2(\H^0(K,T^{(k)}))\vert=\frac{\vert\H^0(T^{(k)})\vert}{\vert\H^0(T^{(k)}/\fm^{k-i}T^{(k)})\vert}$, so $\vert K^{(i)}\vert=\vert C^{(i)}\vert=\frac{\vert\H^0(K,T^{(k)}/\fm^iT^{(k)})\vert\vert\H^0(K,T^{(k)}/\fm^{k-i}T^{(k)})\vert}{\vert\H^0(K,T^{(k)})\vert}$.
	
	To pass to Selmer groups, consider the following diagram (omitting $K$ in the cohomology groups) with the ``map defining the Selmer groups'' the global-to-local map:\\
	{\small\begin{tikzcd}
			0 \ar[r]& K_\cF^{(i)} \ar[r]\ar[d]& K^{(i)} \ar[r]\ar[d]& \ker(f) \ar[d]& \\
			0 \ar[r]& \H^1_\cF(T^{(k)}[\fm^i]) \ar[r]\ar[d,"f_\cF"]& \H^1(T^{(k)}[\fm^i]) \ar[r]\ar[d,"f_1"]& \im(\text{map\ defining\ }\H^1_\cF(T^{(k)}[\fm^i]))\ar[r]\ar[d,"f"]& 0\\
			0 \ar[r]& \H^1_\cF(T^{(k)})[\fm^i] \ar[r]\ar[d]& \H^1(T^{(k)})[\fm^i] \ar[r]\ar[d]& \im(\text{map\ defining\ } \H^1_\cF(T^{(k)}))[\fm^i]\ar[d] \\
			0 \ar[r,dashed]&C_\cF^{(i)} \ar[r]&C^{(i)}\ar[r] &\coker(f)
	\end{tikzcd}}
	What is proved in~\cite[Lemma~3.5.4]{MR04} is equivalent to $\ker(f)=0$. Thus by the snake lemma $K_\cF^{(i)}=K^{(i)}$ and  $C_\cF^{(i)} \inj C^{(i)}$. The claim follows. 
\end{proof}

\begin{remark}
	In the situations considered by the aforementioned authors (i.e., if (h1) holds), we have $K^{(i)}=C^{(i)}=0$. Then it easily follows that $K_\cF^{(i)}=C_\cF^{(i)}=0$, too. In our situation, three of them are equal and can be described explicitly, but $C_\cF^{(i)}$ is only bounded from above. It seems very hard to determine $C_\cF^{(i)}$ completely, but as we will see eventually, it is enough to control its growth using induction. 
\end{remark}

\subsection{Proof of a structure theorem of finite level Selmer groups}
To ease notation, let $(T,\mathcal{F},\mathcal{L})$ denote the Selmer triple $(T_\alpha,\mathcal{F}_{\ord},\mathcal{L}_f)$, and let $\rho=\rho_\alpha$. For any $k\geq 0$, let\[
R^{(k)}=R/\mathfrak{m}^kR,\  T^{(k)}=T/\fm^kT,\ \cL^{(k)}=\{\ell\in\cL:I_\ell\subset p^k\bZ_p\},\]
and let $\mathcal{N}^{(k)}$ be the set of square-free products of primes $\ell\in\cL^{(k)}$. \\

Recall that we have the self-duality $T\iso T^*$ and $(T^{(k)})^*\iso T^{(k)}$. \\

\begin{theorem}\label{square}
	For every $n\in\mathcal{N}^{(k)}$ where $k\gg0$, there is an $R^{(k)}$-module $M^{(k)}(n)$ and an integer $\epsilon$ such that there is a pseudo-isomorphism with kernel and cokernel bounded independent of $k$ (which can be computed explicitly):
	\begin{equation} \label{eq:structure of Selmer}
		\H^1_{\cF(n)}(K,T^{(k)})\sim (R^{(k)})^\epsilon\oplus M^{(k)}(n)\oplus M^{(k)}(n).
	\end{equation}
	Moreover, $\epsilon$ can be taken to be $\epsilon\in\{0,1\}$ and is independent of all $k$ for which $\H^1_\cF(K,T^{(k)})$ is not free, and $n$.
	
\end{theorem}
\begin{proof}
	We follow the ideas in~\cite[sections 1.4--1.5]{How2004}. As in Theorem~1.4.2 in \textit{op.\ cit.}, we still define
	\[
	V_s=\frac{\cH[\frm^s]}{\frm\mathcal{H}[\frm^{s+1}]},\quad  W_s=\frac{\cH[\frm]}{\frm^s\mathcal{H}[\frm^{s+1}]},\]
	where $\cH=\H^1_\cF(K,T^{(k)})$ is a finitely generated $R^{(k)}$-module and $1\leq s\leq k$.

	We also define
	\[
	P_s=\frac{\H^1_\cF(K,T^{(k)}/\fm^sT^{(k)})}{\proj(\H^1_\cF(K,T^{(k)}/\fm^{s+1}T^{(k)}))},\quad Q_s=\frac{\H^1_{\cF^*}(K,(T^{(k)})^*[\fm])}{\pi^s\H^1_{\cF^*}(K,(T^{(k)})^*[\fm^{s+1}])}.\]
	Here $\proj$ denotes the canonical projection and $Q_s$ can be identified with $\frac{\H^1_{\cF}(K,T^{(k)}[\fm])}{\pi^s\H^1_{\cF}(K,T^{(k)}[\fm^{s+1}])}$ by self-duality.
	To further ease notation, let $\cH(T[\fm^i])$ denote $\H_\cF^1(K,T^{(k)}[\fm^i])$. We may then identity $P_s$ with $\frac{\cH(T[\fm^s])}{\fm\cH(T[\fm^{s+1}])}$ and rewrite $Q_s$ as $\frac{\cH(T[\fm])}{\fm^s\cH(T[\fm^{s+1}])}$. Note that these isomorphisms may not be canonical, but as will be clear momentarily, we only care about the sizes of these groups.

	The Cassels--Tate pairing (see~\cite[Proposition~1.4.1, Theorem~1.4.2]{How2004}) gives a nondegenerate pairing of finite dimensional $R^{(k)}/\fm$-vector spaces\[(\ ,\ )_{s,1}\colon P_s\times Q_s\to R^{(k)}[\fm].\] However, in our case, $P_s$ ({resp.} $Q_s$) is not necessarily isomorphic to $V_s$ ({resp.} $W_s$). Moreover, $K_s\defeq\ker(P_s\to Q_s)$ is not the same as $P_{s-1}$. 
	What the arguments in~\cite[Theorem 1.4.2]{How2004} show for our case is that for each $s<k$, $\frac{P_s}{K_s}$ is even dimensional as $R^{(k)}/\fm$-vector spaces. We claim that $\frac{V_s}{V_{s-1}}$ (which is no longer identified with $\frac{P_s}{K_s}$) is also even dimensional for all but maybe two $s$'s bounded (independent of $k$). This will be proved in the following series of lemmas, where we compare the dimensions of $\frac{P_s}{K_s}$ and $\frac{V_s}{V_{s-1}}$ using direct computations.
	
	More precisely, we will assume $\cH$ is not free as an $R^{(k)}$-module (otherwise it automatically has the desired structure~\eqref{eq:structure of Selmer} with $M^{(k)}(n) = 0$ and $\epsilon = 1$), and show that $\frac{V_s}{V_{s-1}}$ is even dimensional for all $1\leq s<k$ except at $N$ and some $1\leq N_0\leq N+1$ independent of $s$ to be defined later, where $N$ is given as in~\cref{stablizing H0}. When $N=N_0$, all $\frac{V_s}{V_{s-1}}$ will be even. See~\cref{even V} for the precise statement. Then by the structure theorem for finitely generated $R^{(k)}$-modules, there is a finitely generated $R^{(k)}$-module $M^{(k)}$ such that 
	\begin{equation*}
		\cH\isom (R^{(k)})^\epsilon \oplus M^{(k)} \oplus M^{(k)} \oplus E^{(k)},
	\end{equation*} where the ``error term'' $E^{(k)}=E^{(k_1)}\oplus E^{(k_2)}$ is a direct sum of two (possibly isomorphic) cyclic $R^{(k)}$-modules of lengths between $1$ and $N+1$. Since this proof (and those of the following lemmas) still works if we replace $\cF$ by any $\cF(n)$, we also get
	\begin{equation*}
		\H_{\cF(n)}^1(K,T^{(k)})\isom (R^{(k)})^{\epsilon_k} \oplus M^{(k)}(n) \oplus M^{(k)}(n) \oplus E^{(k)}(n)
	\end{equation*}
	for some $R^{(k)}$-module $M^{(k)}$ where $E^{(k)}(n)$ again consists of two cyclic summands of lengths between $1$ and $N+1$.

	That $\epsilon$ is independent of $n\in\cN$ is due to the ``parity lemma'' of~\cite[Lemma~1.5.3]{How2004}, which says $\rho(n)\defeq\text{dim}_{R^{(k)}/\fm}\H^1_{\cF(n)}(K,\bar{T}) \pmod{2}$ is independent of $n\in\cN$. The same argument can be modified to work in our situation. Indeed, we will see that $\H^1_{\cF(n)}(K,\bar{T})$ and $\H^1_{\cF(n)}(K,T^{(k)})[\fm]$ will have the same size if $\H^1_{\cF(n)}(K,T^{(k)})$ is not free, and differ by $\oplus \F$ otherwise. Then if $\H^1_{\cF(n)}(K,T^{(k)})$ is not free, we have\[
	\epsilon+2\text{dim}_{R^{(k)}/\fm}M(n)^{(k)}[\fm]+E^{(k)}(n)[\fm]=\text{dim}_{R^{(k)}/\fm}\H_{\cF(n)}^1(K,T^{(k)})[\fm]=\rho(n),\]
	so $\epsilon\equiv \rho(n) \pmod{2}$, and if $\H^1_{\cF(n)}(K,T^{(k)})$ is free, $\epsilon\equiv \rho(n)+1 \pmod{2}$.
\end{proof}

\begin{lemma}
	For a fixed $k\gg0$ and $s<k$,\begin{equation*}
		K_s\defeq\ker(P_s \to Q_s)\iso \frac{\cH(T[\fm^s])[\fm^{s-1}]}{\fm(\cH(T[\fm^{s+1}])[\fm^s])}.
	\end{equation*}
\end{lemma}
\begin{proof}
	Recall that $\ker(P_s\to Q_s)$ fits in the exact sequence\[
	0\to \ker(P_s\to Q_s)\to \frac{\cH(T[\fm^s])}{\fm\cH(T[\fm^{s+1}])}\xrightarrow{\fm^{s-1}} \frac{\cH(T[\fm])}{\fm^{s}\cH(T[\fm^{s+1}])},\]
	so we can write $\ker(P_s\to Q_s)=\frac{(\fm^{s-1})^{-1}(\fm^s\cH(T[\fm^{s+1}]))\cap \cH(T[\fm^s])}{\fm\cH(T[\fm^{s+1}])}$.\\
	We claim that the (well-defined) natural map\[\cH(T[\fm^s])[\fm^{s-1}]\to \frac{(\fm^{s-1})^{-1}(\fm^s\cH(T[\fm^{s+1}]))\cap \cH(T[\fm^s])}{\fm\cH(T[\fm^{s+1}])}\] induced by quotient after inclusion is surjective.\\ Indeed, let $c\in (\fm^{s-1})^{-1}(\fm^s\cH(T[\fm^{s+1}]))\cap \cH(T[\fm^s])$, then $\fm^{s-1} c=\fm^s c_0$ for some $c_0\in\fm^s\cH(T[\fm^{s+1}])$. Let $a=\fm c_0\in\fm\cH(T[\fm^{s+1}])\subset \cH(T[\fm^s])$, then $\fm^{s-1}(c-a)=0$ so $c-a\in\cH(T[\fm^s])[\fm^{s-1}]$ and $c-a$ maps to the class of $c$.
	
	We now compute the kernel of the above map. By definition, the kernel is
	\begin{align*}
		&\{\text{classes\ of\ }c\in\cH(T[\fm^s]):\fm^{s-1}c=0, c=\fm a, a\in\cH(T[\fm^{s+1}])\}\\
		=&\{\text{classes\ of\ c}\in\cH(T[\fm^s]):\fm^{s-1}c=0, c=\fm a ,\fm^s a=0, a\in\cH(T[\fm^{s+1}])\}\\
		=&\{\text{classes\ of\ c}\in\cH(T[\fm^s]):\fm^{s-1}c=0, c=\fm a, a\in\cH(T[\fm^{s+1}])[\fm^s]\}\\
		=&\fm( \cH(T[\fm^{s+1}])[\fm^s]),
	\end{align*}
	hence the lemma.
\end{proof}

\begin{lemma}\label[lemma]{comparing even}
	
	For a fixed $k\gg0$ and all $s<k$, {\footnotesize\begin{equation*}
			\frac{\vert\frac{ V_s}{ V_{s-1}}\vert}{\vert\frac{P_s}{ K_s}\vert}=(\vert K_\cF^{(s)}\vert)^{-2}\times\frac{\vert K_s^{s+1}\vert}{\vert K_{s-1}^s\vert}\times \frac{\vert C_\cF^{(s)} \vert}{\vert C_\cF^{(s+1)}\vert}/\vert C_s^{s+1}\vert\times\Bigg(\frac{\vert C_\cF^{(s-1)} \vert}{\vert C_\cF^{s}\vert}/\vert C_{s-1}^s\vert\Bigg)^{-1}\times\vert K_\cF^{(s-1)}\vert\times\vert K_\cF^{(s+1)}\vert,
	\end{equation*}}
	where $K_s^{s+1},C_s^{s+1}$ (and similarly $K_{s-1}^s,C_{s-1}^s$) are the kernel and cokernel of the morphism\[
	\H^1_\cF(T^{(s+1)}[\fm^s])\to\H^1_\cF(T^{(s+1)})[\fm^s].\]

\end{lemma}
\begin{proof}
	Denote by $N(P_s):=\cH(T[\fm^s])$ and $D(P_s):=\fm\cH(T[\fm^{s+1}])$ the numerator and denominator of $P_s$. Similarly define $N(K_s),D(K_s),N(V_{s-1}),D(V_{s-1}),N(V_s),D(V_s)$. We will study the following $4$ maps separately.
	
	\noindent\textbf{$N(P_s)\to N(V_s)$:} This is the map\[0\to K_\cF^{(s)} \to \cH(T[\fm^s])\xrightarrow{N_{P\to V}} \cH(T)[\fm^s]\to C_\cF^{(s)}\to 0.\]
	So $\frac{\vert N(P_s)\vert}{\vert N(V_s)\vert}=\frac{\vert K_\cF^{(s)}\vert}{\vert C_\cF^{(s)}\vert}$.
	
	\noindent\textbf{$D(P_s)\to D(V_s)$:} This is the map\[ \fm\cH(T[\fm^{s+1}])\xrightarrow{D_{P\to V}} \fm(\cH(T)[\fm^{s+1}]).\]
	To study the kernel and cokernel of this map, we study the following diagram\\
	\begin{tikzcd}
		& \bar{K} \ar[d]& K_\cF^{(s+1)} \ar[d]& \ker(D_{P\to V}) \ar[d]& \\
		0\ar[r] & \cH(T[\fm^{s+1}])[\fm] \ar[r]\ar[d]& \cH(T[\fm^{s+1}]) \ar[r,"\fm_{(1)}"]\ar[d]& \im(\fm_{(1)})\ar[r]\ar[d,"D_{P\to V}"] &0\\
		0\ar[r]& \cH(T)[\fm] \ar[r]\ar[d]& \cH(T)[\fm^{s+1}] \ar[r,"\fm_{(2)}"]\ar[d]&\im(\fm_{(2)})\ar[r]\ar[d]&0 \\
		& \bar{C} &C_\cF^{(s+1)}& \coker{D_{P\to V}}
	\end{tikzcd}\\
	From the snake lemma, we see that $\frac{\vert D(P_s)\vert}{\vert D(V_s)\vert}=\frac{\vert \ker(D_{P\to V})\vert}{\vert \coker(D_{P\to V})\vert}=\frac{\vert K_\cF^{(s+1)}\vert}{\vert C_\cF^{(s+1)}\vert}\frac{\vert \bar{C}\vert}{\vert\bar{K}\vert}$.\\
	
	\noindent\textbf{$N(K_s)\to N(V_{s-1})$:} This is the map\[\cH(T[\fm^s])[\fm^{s-1}]\xrightarrow{N_{K\to V}} \cH(T)[\fm^{s-1}],\]
	To study the kernel and cokernel of this map, we need to consider the following diagram with $2$ familiar maps:\\
	\begin{tikzcd}
		0\ar[r] & K_{s-1}^{s}\ar[r]& \cH(T[\fm^{s}][\fm^{s-1}]) \ar[r]\ar[d,"="]& \cH(T[\frm^{s}])[\fm^{s-1}]\ar[r]\ar[d,"N_{K\to V}"] &C_{s-1}^{s}\ar[r]\ar[d]&0\\
		0\ar[r]& K_\cF^{(s-1)} \ar[r]& \cH(T)[\fm^{s-1}] \ar[r]&\cH(T)[\fm^{s-1}]\ar[r]&C_\cF^{(s-1)}\ar[r]&0 \\
	\end{tikzcd}\\
	Comparing the dimensions, we see that $\frac{\vert N(K_s)\vert}{\vert N(V_{s-1})\vert}=\frac{\vert K_\cF^{(s-1)}\vert}{\vert K_{s-1}^s\vert}\frac{\vert C_{s-1}^s\vert}{\vert C_\cF^{(s-1)}\vert}$.
	
	\noindent\textbf{$D(K_s)\to D(V_{s-1})$:} This is the map\[ \fm(\cH(T[\fm^{s+1}])[\fm^s])\xrightarrow{D_{K\to V}} \fm(\cH(T)[\fm^{s}]).\]
	To study the kernel and cokernel of this map, we study the following diagram\\
	\begin{tikzcd}
		& \bar{K} \ar[d]& \hat{K} \ar[d]& \ker(D_{K\to V}) \ar[d]& \\
		0\ar[r] & \cH(T[\fm^{s+1}])[\fm] \ar[r]\ar[d]& \cH(T[\fm^{s+1}])[\fm^s] \ar[r,"\fm_{(1)}"]\ar[d]& \im(\fm_{(1)})\ar[r]\ar[d,"D_{K\to V}"] &0\\
		0\ar[r]& \cH(T)[\fm] \ar[r]\ar[d]& \cH(T)[\fm^{s+1}][\fm^s] \ar[r,"\fm_{(2)}"]\ar[d]&\im(\fm_{(2)})\ar[r]\ar[d]&0 \\
		& \bar{C} &\hat{C}& \coker{D_{K\to V}}
	\end{tikzcd}\\
	Note that the first column also appeared in the second case and the middle column is already studied in the previous case (except we had $s-1$ in place of $s$): we have $\frac{\vert \hat{K}\vert}{\vert\hat{C}\vert}=\frac{\vert K_\cF^{(s)}\vert}{\vert K_s^{s+1}\vert}\frac{\vert C_s^{s+1}\vert}{\vert C_\cF^{(s)}\vert}$.\\
	Now the snake lemma tells us that $\frac{\vert D(K_s)\vert}{\vert D(V_{s-1})\vert}=\frac{\vert \ker(D_{K\to V})\vert}{\vert \coker(D_{K\to V})\vert}=\frac{\vert \hat{K}\vert}{\vert \hat{C}\vert}\frac{\vert \bar{C}\vert}{\vert\bar{K}\vert}=\frac{\vert K_\cF^{(s)}\vert}{\vert K_s^{s+1}\vert}\frac{\vert C_s^{s+1}\vert}{\vert C_\cF^{(s)}\vert}\frac{\vert \bar{C}\vert}{\vert\bar{K}\vert}$.
	
	Now \small{\begin{align*}
			\frac{\vert\frac{ V_s}{ V_{s-1}}\vert}{\vert\frac{P_s}{ K_s}\vert}=&\frac{\vert N(V_s)\vert}{\vert N(P_s)\vert} \frac{\vert D(V_{s-1})\vert}{\vert D(K_s)\vert}\frac{\vert D(P_s)\vert}{\vert D(V_s)\vert}\frac{\vert N(K_s)\vert}{\vert N(V_{s-1})\vert}\\
			=&\frac{\vert C_\cF^{(s)}\vert}{\vert K_\cF^{(s)}\vert}\frac{\vert K_s^{s+1}\vert}{\vert K_\cF^{(s)}\vert}\frac{\vert C_\cF^{(s)}\vert}{\vert C_s^{s+1}\vert}\frac{\vert \bar{K}\vert}{\vert\bar{C}\vert}\frac{\vert K_\cF^{(s+1)}\vert}{\vert C_\cF^{(s+1)}\vert}\frac{\vert \bar{C}\vert}{\vert\bar{K}\vert}\frac{\vert K_\cF^{(s-1)}\vert}{\vert K_{s-1}^s\vert}\frac{\vert C_{s-1}^s\vert}{\vert C_\cF^{(s-1)}\vert}\\
			=&(\vert K_\cF^{(s)}\vert)^{-2}\times\frac{\vert K_s^{s+1}\vert}{\vert K_{s-1}^s\vert}\times \frac{\vert C_\cF^{(s)} \vert}{\vert C_\cF^{(s+1)}\vert}/\vert C_s^{s+1}\vert\times(\frac{\vert C_\cF^{(s-1)} \vert}{\vert C_\cF^{s}\vert}/\vert C_{s-1}^s\vert)^{-1}\times\vert K_\cF^{(s-1)}\vert\times\vert K_\cF^{(s+1)}\vert.
	\end{align*}}
\end{proof}

Despite that we do not have full understanding of each $C_\cF^{(i)}$, we are able to determine the parity of the dimension of the right hand side by comparing these cokernels inductively, provided that $k$ is sufficiently large.
\begin{lemma}\label[lemma]{even V}
	For a fixed $k\gg0$, if $\H^1_\cF(K,T^{(k)})$ is not $R^{(k)}$-free, then there is an $N_0$ with $1\leq N_0\leq N+1$ such that $\frac{\vert K_s^{s+1}\vert}{\vert K_{s-1}^s\vert}\times\frac{\vert C_\cF^{(s)} \vert}{\vert C_\cF^{(s+1)}\vert}/\vert C_s^{s+1}\vert\times\Bigg(\frac{\vert C_\cF^{(s-1)} \vert}{\vert C_\cF^{s}\vert}/\vert C_{s-1}^s\vert\Bigg)^{-1}\times\vert K_\cF^{(s-1)}\vert\times\vert K_\cF^{(s+1)}\vert$ is an even power of $|\F|$ when $s\neq N_0,N$. When $s=N_0$ or $N$ when $N\neq N_0$, the product is an odd power of $|\F|$, and when $s=N_0=N$, the product is an even power. Consequently, by~\cref{comparing even}, $\frac{V_s}{V_{s-1}}$ is even-dimensional for all $s\neq N_0, N$. When $N\neq N_0$, it is odd-dimensional for $s=N,N_0$. When $N=N_0$, it is even-dimensional for $s=N=N_0$ and hence for all $1\leq s<k$.
\end{lemma}
\begin{proof}
	Taking cohomology of the short exact sequence
	\begin{equation*}
		0\to T^{(k)}[\fm^s] \to T^{(k)}[\fm^{s+1}] \xrightarrow{\fm^s} T^{(k)}[\fm]\to 0
	\end{equation*} yields an exact sequence (omitting $K$ from the cohomology groups)
	\begin{equation*}
		0\to \H^0(T^{(k)}[\fm])/\fm^s \to \H^1(T^{(k)}[\fm^s]) \to \H^1(T^{(k)}[\fm^{s+1}])\to \H^1(T^{(k)}[\fm])\to \ldots
	\end{equation*}

	Note that $T^{(k)}[\fm^s]\iso T^{(k)}/\fm^s\iso T^{(s)}/\fm^s\iso T^{(s)}$, we have $\vert \H^0(T^{(k)}[\fm])/\fm^s\vert=\frac{\vert \H^0(T^{(k)}[\fm])\vert\vert \H^0(T^{(k)}[\fm^s])\vert}{\vert \H^0(T^{(k)}[\fm^{s+1}])\vert}=\frac{\vert \H^0(\bar{T})\vert\vert \H^0(T^{(s)})\vert}{\vert \H^0(T^{(s+1)})\vert}$.
	
	We assume $k>2N+1$ and we consider several cases depending on how big $s$ is and how close it is to $k$. When $1\leq s< N$, the above quotient is trivial by computing dimensions. When $s\geq N$, $\vert \H^0(T^{(s)})\vert=\vert \H^0(T^{(s+1)})\vert$ so the above quotient is (non-canonically) $\H^0(K,\bar{T})\isom\F$. 
	
	Now consider the diagram\\
	{\small\begin{tikzcd}
			& & K^{(s)} \ar[d]& K^{(s+1)} \ar[d]& K^{(1)} \ar[d]& \\
			0 \ar[r]&\frac{\vert \H^0(\bar{T})\vert\vert \H^0(T^{(s)})\vert}{\vert \H^0(T^{(s+1)})\vert}\ar[r]& \H^1(T^{(k)}[\fm^s]) \ar[r]\ar[d,"f_1^{s}"]& \H^1(T^{(k)}[\fm^{s+1}]) \ar[r]\ar[d,"f_1^{s+1}"]& \H^1(T^{(k)}[\fm])\ar[d,"f_1^{1}"]\\
			&0 \ar[r]& \H^1(T^{(k)})[\fm^s] \ar[r]\ar[d]& \H^1(T^{(k)})[\fm^{s+1}] \ar[r]\ar[d]&\H^1(T^{(k)})[\fm]\ar[d] \\
			& &C^{(s)} &C^{(s+1)}&C^{(1)}
	\end{tikzcd}}\\
	When $1\leq s<N$ (so $k-s>N$),  by~\cref{K and C for H1} and~\cref{stablizing H0}, $\vert K^{(s)}\vert=\vert C^{(s)}\vert=\frac{\vert\H^0(K,T^{(s)})\vert\vert\H^0(K,T^{(k-s)})\vert}{\vert\H^0(K,T^{(k)})\vert}=\vert\H^0(K,T^{(s)})\vert=|\F|^s$ and $\vert K^{(s+1)}\vert=\vert C^{(s+1)}\vert=|\F|^{s+1}$. Note also that $K^{(1)}=C^{(1)}=\F$.
	
	To pass to the Selmer groups, we have the following diagram (noting $K_\cF^{(i)}=K^{(i)}$ for all $i$):\\
	\begin{tikzcd}
		& R/\pi^s \ar[d]& R/\pi^{s+1} \ar[d]& \F \ar[d]& \\
		0\ar[r] & \H^1_\cF(T^{(k)}[\fm^s]) \ar[r]\ar[d,"f_\cF^{s}"]& \H^1_\cF(T^{(k)}[\fm^{s+1}]) \ar[r,"\fm^s_{(1)}"]\ar[d,"f_\cF^{s+1}"]& \H^1_\cF(T^{(k)}[\fm])\ar[d,"f_\cF^{1}"]\\
		0\ar[r]& \H^1_\cF(T^{(k)})[\fm^s] \ar[r]\ar[d]& \H^1_\cF(T^{(k)})[\fm^{s+1}] \ar[r,"\fm^s_{(2)}"]\ar[d]&\H^1_\cF(T^{(k)})[\fm]\ar[d] \\
		&C_\cF^{(s)} &C_\cF^{(s+1)}&C_\cF^{(1)}(\subset C^{(1)}\iso \F)
	\end{tikzcd}\\
	By the structure theorem of finitely generated $R^{(k)}$-modules, the map $\fm^s_{(2)}$ is surjective if and only if $\H^1_\cF(T^{(k)})$ has a summand with height less than or equal to $s$. Let $N_0\defeq\min\{s\leq k:\fm^s_{(2)}\text{is not surjective}\}$. Then it is clear that if $s\geq N_0$, $\fm^s_{(2)}$ will also be non-surjective.
	
	Note that when $N_0=k$, $\H^1_\cF(T^{(k)})$ must be a free $R^{(k)}$-module. We claim that if $N_0\neq k$, then $N_0\leq N+1$. Namely, the lowest height of all non-free summands of $\H^1_\cF(T^{(k)})$ must have height less than or equal to $N$. In other words, assuming $\H^1_\cF(T^{(k)})$ is not a free $R^{(k)}$-module, as soon as $s\geq N+1$, $\fm^s_{(2)}$ will never be surjective.
	
	Assume by contradiction that $N_0>N+1$. Then for all $1\leq s\leq N+1$, $\fm^s_{(2)}$ is surjective. Consider the following diagram when $s<N$\\
	\begin{tikzcd}
		0\ar[r]	& R/\pi^s \ar[r]\ar[d]& R/\pi^{s+1}\ar[r] \ar[d]& K_s(\subset \F) \ar[d]& \\
		0\ar[r] & \H^1_\cF(T^{(k)}[\fm^s]) \ar[r]\ar[d,"f_\cF^{s}"]& \H^1_\cF(T^{(k)}[\fm^{s+1}]) \ar[r,"\fm^s_{(1)}"]\ar[d,"f_\cF^{s+1}"]& \im(\fm^s_{(1)})\ar[r]\ar[d,"f_\cF^{1'}"]& 0\\
		0\ar[r]& \H^1_\cF(T^{(k)})[\fm^s] \ar[r]\ar[d]& \H^1_\cF(T^{(k)})[\fm^{s+1}] \ar[r,"\fm^s_{(2)}"]\ar[d]&\H^1_\cF(T^{(k)})[\fm]\ar[r]\ar[d]&0 \\
		0\ar[r,dashed]	&C_\cF^{(s)} \ar[r]&C_\cF^{(s+1)}\ar[r]&C_s(\subset C^{(1)}\iso \F)\ar[r]&0
	\end{tikzcd}\\
	Since $K_s$ contains $\coker(R/\pi^s\to R/\pi^{s+1})$, $K_s$ is at least $\F$, so $K_s=\F$. Similarly, when $s=N$ or $N+1$, the first row becomes \[0\to \H^0(\bar{T})\isom\F\to \H^0(T^{(k-s)}) \isom R/\pi^N \to \H^0(T^{(k-s-1)})\isom R/\pi^N \to K_s,\] and $K_s$ must be $\F$. Therefore the snake lemma shows that $C_\cF^{(s)}\to C_\cF^{(s+1)}$ is injective with cokernel $C_s$. We consider $2$ subcases.
	
	\textbf{Case I: $\fm^s_{(1)}$ is never surjective for $1\leq s\leq N+1$.} In this case, consider the rightmost vertical map $f^{1'}_\cF$. If we replace its domain $\im(\fm^s_{(1)})$ by $\H^1_\cF(T^{(k)}[\fm])$ which is strictly bigger in dimension, since the codomain is fixed and the kernel cannot be any bigger, its cokernel must change from $C_s$ to something strictly smaller. Therefore $C_s\neq0$ and $C_s=\F$. Then for all $1\leq s\leq N+1$, the dimension of $C_\cF^{(s)}$ increases whenever $s$ increases, so $\vert C_\cF^{(N+2)}\vert \geq \vert \F^{N+1} \vert >\vert K_\cF^{(N+2)}\vert$. But $C_\cF^{(s)}\subset K_\cF^{(s)}$ by~\cref{K and C for H1}, contradiction.
	
	\textbf{Case II: $\fm^s_{(1)}$ is surjective for some $1\leq s\leq N+1$.} In this case, the rightmost vertical map $f_\cF^{1'}$ agrees with the map $f_\cF^1\colon \H^1_\cF(T^{(k)}[\fm])\to \H^1_\cF(T^{(k)})[\fm]$ and is independent of $s$. Therefore if for one such $s$ the cokernel of $f_\cF^{1'}$ (so also that of $f_\cF^{1}$) is $\F$, then the cokernel is $\F$ for all such $s$. Therefore $C_\cF^{(s+1)}$ also has one more dimension than $C_\cF^{(s)}$ when $\fm^s_{(1)}$ is surjective. But for other $s$, the same is true as in Case I, so again $\vert C_\cF^{(N+2)}\vert$ would be too big. Thus the cokernel of $f_\cF^{1}$ must be $0$. This in turn means all $\fm^s_{(2)}$ must be surjective because otherwise if we replace  $\H^1_\cF(T^{(k)})[\fm]$ by $\im(\fm^s_{(2)})$, the cokernel of $f_\cF^{1}$ must be smaller (since $f_\cF^{1}$ has fixed domain and kernel), which is impossible. But in this case $N_0=k$ and $\H^1_\cF(T^{(k)})$ is free.
	
	So far we showed that either $\H^1_\cF(T^{(k)})$ is free or it has a summand with height lower than $N+1$, and as a consequence we may assume $\fm^s_{(2)}$ is not surjective for all $s>N$ which is independent of $k$. Moreover, in this case $f_\cF^1$ will have kernel and cokernel both equal to $\F$, so in what follows, we will keep this assumption.
	
	We first study the dimensions appearing in this lemma when $k>s>N$.  By~\cref{K and C for H1} and~\cref{stablizing H0}, $\vert K^{(s)}\vert=\vert C^{(s)}\vert=\frac{\vert\H^0(K,T^{(s)})\vert\vert\H^0(K,T^{(k-s)}\vert}{\vert\H^0(K,T^{(k)})\vert}=\vert\H^0(K,T^{(k-s)})\vert$ and $\vert K^{(s+1)}\vert=\vert C^{(s+1)}\vert=\vert\H^0(K,T^{(k-s-1)})\vert$. Note also that $K^{(1)}=C^{(1)}=\F$.
	
	To pass to the Selmer groups, we have the following diagram(noting $K_\cF^{(i)}=K^{(i)}$ for all $i$):\\
	{\small\begin{tikzcd}
			0\ar[r]	& \H^0(\bar{T})=\F\ar[r] & \H^0(T^{(k-s)}) \ar[r]\ar[d]& \H^0(T^{(k-s-1)}) \ar[r]\ar[d]& \F \ar[d]& \\
			0 \ar[r]&\F\ar[r]& \H^1_\cF(T^{(k)}[\fm^s]) \ar[r]\ar[d,"f_\cF^{s}"]& \H^1_\cF(T^{(k)}[\fm^{s+1}]) \ar[r,"\fm^s_{(1)}"]\ar[d,"f_\cF^{s+1}"]& \H^1_\cF(T^{(k)}[\fm])\ar[d,"f_\cF^{1}"]\\
			&0 \ar[r]& \H^1_\cF(T^{(k)})[\fm^s] \ar[r]\ar[d]& \H^1_\cF(T^{(k)})[\fm^{s+1}] \ar[r,"\fm^s_{(2)}"]\ar[d]&\H^1_\cF(T^{(k)})[\fm]\ar[d] \\
			& &C_\cF^{(s)} &C_\cF^{(s+1)}&C_\cF^{(1)}(\subset C^{(1)}\iso \F)
	\end{tikzcd}}\\
	To apply snake lemma, we consider the following diagram where the first two rows may or may not be changed depending on whether $\fm^s_{(1)}$ is surjective:\\
	{\small\begin{tikzcd}
			0\ar[r]	& \F\ar[r] & \H^0(T^{(k-s)}) \ar[r]\ar[d]& \H^0(T^{(k-s-1)}) \ar[r]\ar[d]& K(\subset \F) \ar[d]& \\
			0 \ar[r]&\F\ar[r]& \H^1_\cF(T^{(k)}[\fm^s]) \ar[r]\ar[d,"f_\cF^{s}"]& \H^1_\cF(T^{(k)}[\fm^{s+1}]) \ar[r,"\fm^s_{(1)}"]\ar[d,"f_\cF^{s+1}"]& \im(\fm^s_{(1)})\ar[d,"f_\cF^{1'}"]\ar[r]&0\\
			&0 \ar[r]& \H^1_\cF(T^{(k)})[\fm^s] \ar[r]\ar[d]& \H^1_\cF(T^{(k)})[\fm^{s+1}] \ar[r,"\fm^s_{(2)}"]\ar[d]&\H^1_\cF(T^{(k)})[\fm]\ar[d] \\
			& &C_\cF^{(s)}\ar[r] &C_\cF^{(s+1)}\ar[r]&C_\cF^{(1)}(\subset C^{(1)}\iso \F)
	\end{tikzcd}}\\
	In the rightmost column, if we replace $\H^1_\cF(T^{(k)})[\fm]$ by $\im(\fm^s_{(2)})$ which is strictly smaller, the only possibility is $C_\cF^{(1)}$ gets strictly smaller because the domain and kernel of $f_\cF^{1'}$ are determined by other columns(by snake lemma) which do not change. This in turn means $C_\cF^{(1)}$ must be $\F$ to begin with, so if we do consider $\im(\fm^s_{(2)})$ we get the following diagram:\\
	{\small\begin{tikzcd}
			0\ar[r]	& \F\ar[r] & \H^0(T^{(k-s)}) \ar[r]\ar[d]& \H^0(T^{(k-s-1)}) \ar[r]\ar[d]& K(\subset \F) \ar[d]& \\
			0 \ar[r]&\F\ar[r]& \H^1_\cF(T^{(k)}[\fm^s]) \ar[r]\ar[d,"f_\cF^{s}"]& \H^1_\cF(T^{(k)}[\fm^{s+1}]) \ar[r,"\fm^s_{(1)}"]\ar[d,"f_\cF^{s+1}"]& \im(\fm^s_{(1)})\ar[d,"f_\cF^{1''}"]\ar[r]&0\\
			&0 \ar[r]& \H^1_\cF(T^{(k)})[\fm^s] \ar[r]\ar[d]& \H^1_\cF(T^{(k)})[\fm^{s+1}] \ar[r,"\fm^s_{(2)}"]\ar[d]&\im(\fm^s_{(2)})\ar[r]\ar[d] &0\\
			& &C_\cF^{(s)}\ar[r] &C_\cF^{(s+1)}\ar[r]&0
	\end{tikzcd}}\\
	
	When $k=s+1$, we have a special diagram
	\\
	{\small\begin{tikzcd}
			& & & 0 \ar[r]\ar[d]& K'(\subset\F) \ar[d]& \\
			&& \H^1_\cF(T^{(s+1)}[\fm^s]) \ar[r]\ar[d,"f_\cF^{s'}"]& \H^1_\cF(T^{(s+1)}) \ar[r,"\fm^s_{(1)}"]\ar[d,"="]& \im(\fm^s_{(1)})\ar[r]\ar[d,"f_\cF^{1'''}"]&0\\
			&0 \ar[r]& \H^1_\cF(T^{(s+1)})[\fm^s] \ar[r]\ar[d]& \H^1_\cF(T^{(s+1)})\ar[r,"\fm^s_{(2')}"]\ar[d]&\im(\fm^s_{(2')})\ar[r]\ar[d]&0 \\
			& &C_s^{s+1}\ar[r] &0\ar[r]&0
	\end{tikzcd}}\\
	from which we see $C_s^{s+1}=K'$ by the snake lemma.
	
	Our aim is to compare $\vert C_s^{s+1}\vert$ with $\vert C_\cF^{(s)}/C_\cF^{(s+1)}\vert=\vert K\vert$. 
	Depending on whether $k-s\leq N$ there are $2$ different possibilities:\\
	\textbf{Case i: $s>N$ and $k-s\leq N$.} In this case $\H^0(T^{(k-s)})\isom R/\pi^{k-s}$ and $\H^0(T^{(k-s-1)})\isom R/\pi^{k-s-1}$ so from snake lemma we get an exact sequence
	\[0\to K\to C_\cF^{(s)}\to C_\cF^{(s+1)}\to 0.\]
	Notice that the two $\fm^s_{(1)}$ in the two diagrams are indeed the same map because the rows are identical, since $T^{(k)}[\fm^i]\iso T^{(i)}$ for any $i\leq k$. Now if $\fm^s_{(1)}$ is surjective, then $K=K'=K^{(1)}=\F$. If if $\fm^s_{(1)}$ is not surjective, then in the last column if we replace $\im(\fm^s_{(1)})$ by $\H^1_\cF(T^{(k)}[\fm]) (=\H^1_\cF(T^{(s+1)}[\fm]))$ which is strictly bigger, it must be that $K$ (and $K'$) becomes strictly bigger because the codomain is fixed and the cokernel cannot be smaller. This in turn means $K=K'=0$.
	
	Thus in case i, we always have $K=K'$ and $\frac{\vert C_\cF^{(s)} \vert}{\vert C_\cF^{(s+1)}\vert}/\vert C_s^{s+1}\vert=0$. Moreover, in case i, $K_\cF^{(s-1)}\isom R/\pi^{n-s+1}$ and $K_\cF^{(s+1)}\isom R/\pi^{n-s-1}$ so $\vert K_\cF^{(s-1)}\vert\vert K_\cF^{(s+1)}\vert$ is an even power of $\F$.\\
	\textbf{Case ii: $s>N$ and $k-s>N$.} In this case $\H^0(T^{(k-s)})=\H^0(T^{(k-s-1)})\isom R/\pi^N$, so we have an exact row\[
	0\to \F\to R/\pi^N \to R/\pi^N\  \to K(\subset\F),\] which forces $K=\F$. Furthermore, by the snake lemma this means $C_\cF^{(s)}=C_\cF^{(s+1)}$. This also implies $\fm^s_{(1)}$ is surjective, so $C_s^{s+1}=K'=\F$. \\
	
	Thus in case ii, $K=K'=\F$ and $\frac{\vert C_\cF^{(s)} \vert}{\vert C_\cF^{(s+1)}\vert}/\vert C_s^{s+1}\vert=|\F|$.
	
	Finally, note that $K_s^{s+1}=\F$ for all $s\geq N$ and $K_{s-1}^s=0$ if $s<N$ by~\cref{K and C for H1} and~\cref{stablizing H0}.

	To sum up, when $k$ is at least $2N+2$, as $s$ grows from $N+2$ (so $s-1>N$) to $k$, $k-s$ is first greater than $N$. For such $s$, $\frac{\vert C_\cF^{(s)} \vert}{\vert C_\cF^{(s+1)}\vert}/\vert C_s^{s+1}\vert=|\F|=\frac{\vert C_\cF^{(s-1)} \vert}{\vert C_\cF^{(s)}\vert}/\vert C_{s-1}^{s}\vert$, and $K_\cF^{(s-1)}\isom R/\pi^N=K_\cF^{(s+1)}$; when $s=k-N$, $\frac{\vert C_\cF^{(s)} \vert}{\vert C_\cF^{(s+1)}\vert}/\vert C_s^{s+1}\vert=0$ and $\frac{\vert C_\cF^{(s-1)} \vert}{\vert C_\cF^{(s)}\vert}/\vert C_{s-1}^{s}\vert\isom\F$, and $K_\cF^{(s-1)}\isom R/\pi^N$, $K_\cF^{(s+1)}\isom R/\pi^{N-1}$; when $s>k-N$,  $\frac{\vert C_\cF^{(s)} \vert}{\vert C_\cF^{(s+1)}\vert}/\vert C_s^{s+1}\vert=1=\frac{\vert C_\cF^{(s-1)} \vert}{\vert C_\cF^{(s)}\vert}/\vert C_{s-1}^{s}\vert$, and $K_\cF^{(s-1)}\isom R/\pi^{k-s+1}$ and $K_\cF^{(s+1)}\isom R/\pi^{k-s-1}$. In all cases, \[\frac{\vert K_s^{s+1}\vert}{\vert K_{s-1}^s\vert}\times\frac{\vert C_\cF^{(s)} \vert}{\vert C_\cF^{(s+1)}\vert}/\vert C_s^{s+1}\vert\times\Bigg(\frac{\vert C_\cF^{(s-1)} \vert}{\vert C_\cF^{s}\vert}/\vert C_{s-1}^s\vert\Bigg)^{-1}\times\vert K_\cF^{(s-1)}\vert\times\vert K_\cF^{(s+1)}\vert\] is an even power of $|\F|$.\\

	Now the cases left are when $s\leq N$. When $\fm^s_{(2)}$ is not surjective, then similarly as in Case II, $\fm^s_{(1)}$ must be surjective. Then for such $s$, $C_s^{s+1}=|\F|$ if $s=N$ and $C_s^{s+1}=0$ if $s<N$. Also $C_\cF^{s}\iso C_\cF^{s+1}$, so $\frac{\vert C_\cF^{(s)} \vert}{\vert C_\cF^{(s+1)}\vert}/\vert C_s^{s+1}\vert=|\F|^{-1}$ if $s=N$ and $\frac{\vert C_\cF^{(s)} \vert}{\vert C_\cF^{(s+1)}\vert}/\vert C_s^{s+1}\vert=1$ if $s<N$.
	
	When $\fm^s_{(2)}$ is surjective, in fact, $\fm^s_{(1)}$ must be surjective, too. Note that for $s\leq N$ we have the following sequence\[
	0\to\F\to \im(\fm^s_{(1)})\xrightarrow{f^{1'}_\cF} \H^1_\cF(T^{(k)})[\fm]\to C_s(\subset C^{(1)}\iso \F)\to 0.\]
	
	If $\fm^S_{(1)}$ is not surjective, then enlarging $\im(\fm^s_{(1)})$ would shrink $C_s$, yielding an exact sequence\[
	0\to\F\to \H^1_\cF(T^{(k)}[\fm])\xrightarrow{f^{1}_\cF} \H^1_\cF(T^{(k)})[\fm]\to 0,\]
	which would mean $\fm^{s}_{(2)}$ are always surjective, a contradiction. 
	
	So as long as $\fm^s_{(s)}$ is surjective, $C_\cF^{s}$ is increasing steadily. And if $s+1\leq N$, then $C_s^{s+1}=0$ by~\cref{K and C for H1}. So $\frac{\vert C_\cF^{(s)} \vert}{\vert C_\cF^{(s+1)}\vert}/\vert C_s^{s+1}\vert=|\F|^{-1}$ if $s<N$. When $s=N$, since $\fm^N_{(1)}$ is surjective, $C_N^{(N+1)}=K'\isom\F$. So $\frac{\vert C_\cF^{(s)} \vert}{\vert C_\cF^{(s+1)}\vert}/\vert C_s^{s+1}\vert=1$.

	We summarize the behavior of the terms appearing in the lemma in the following table:
	{\small\begin{center}
		\begin{tabular}{| l | l | l | l | l | l |}
			\hline
			& $\frac{\vert C_\cF^{(s)}\vert}{\vert C_\cF^{(s+1)}\vert}/\vert C_{s}^{s+1}\vert$ & $\frac{\vert C_\cF^{(s-1)}\vert}{\vert C_\cF^{(s)}\vert}/\vert C_{s-1}^{s}\vert$ & $\vert K_\cF^{(s+1)}\vert$ & $\vert K_\cF^{(s-1)}\vert$ & $\frac{\vert K_s^{s+1}\vert}{\vert K_{s-1}^s\vert}$ \\ \hline
			\multirow{3}{*}{$N<S<k-N$}&\multirow{3}{*}{$|\F|$}  & $(s-1>N):|\F|$ & \multirow{3}{*}{$|\F|^N$} & \multirow{3}{*}{$|\F|^N$} & \multirow{3}{*}{$1$}\\
			& & $(s-1=N\geq N_0):|\F|$ & & &  \\
			& & $(s-1=N<N_0):1$   & & & \\ \hline
			$S=N<N_0$ & $1$ & $|\F|^{-1}$ & $|\F|^N$ 
			& $|\F|^{N-1}$   & $|\F|$ \\
			\hline
			$S=N=N_0$ & $|\F|^{-1}$ & $|\F|^{-1}$ & $|\F|^N$ 
			& $|\F|^{N-1}$  & $|\F|$ \\
			\hline
			$S=N>N_0$ & $|\F|^{-1}$ & $1$ & $|\F|^N$ 
			& $|\F|^{N-1}$ & $|\F|$\\
			\hline
			$N_0<S<N$ & $1$ & $1$ & $|\F|^{s+1}$
			& $|\F|^{s-1}$ & $1$\\
			\hline
			$S=N_0<N$ & $1$ & $|\F|^{-1}$ & $|\F|^{s+1}$
			& $|\F|^{s-1}$ & $1$\\
			\hline
			$S<N_0\leq N$ & $|\F|^{-1}$ & $|\F|^{-1}$ & $|\F|^{s+1}$
			& $|\F|^{s-1}$ & $1$\\
			\hline
				\end{tabular}
\end{center}}

	We're now able to determine the parity of the length of $\frac{\vert K_s^{s+1}\vert}{\vert K_{s-1}^s\vert}\times\frac{\vert C_\cF^{(s)} \vert}{\vert C_\cF^{(s+1)}\vert}/\vert C_s^{s+1}\vert\times(\frac{\vert C_\cF^{(s-1)} \vert}{\vert C_\cF^{s}\vert}/\vert C_{s-1}^s\vert)^{-1}\times\vert K_\cF^{(s-1)}\vert\times\vert K_\cF^{(s+1)}\vert$ as a power of $\vert\F\vert$ via direct computation in different cases, depending on the relation between $N_0$ and $N$.
	
	\textbf{Case 1: $N_0=N+1$.} In this case, when $s=N_0$ and $s=N$ we get an odd dimension, and even dimensions for all other $s\leq N+1$.
	
	\textbf{Case 2: $N_0=N$.} In this case, when get even dimensions for all $s\leq N+1$.
	
	\textbf{Case 3: $N_0=N-1$.} Same as in Case 1.
	
	\textbf{Case 4: $1<N_0<N-1$.} Same as in Case 1.
	
	\textbf{Case 5: $N_0=1$.} Same as in Case 1.
\end{proof}

\section{An anticyclotomic control theorem}\label{appB}
In this section we follow~\cite[section 3]{JSW2017} and~\cite[sections 3 and~4]{Greenberg1999} to prove an anticyclotomic control theorem with torsion, i.e., allowing $\H^0(K,W)\ne 0$. We use the notations from~\cite{JSW2017} and will also explain the difference between theirs and ours.

First, we make all the assumptions in~\cite[section 3.1]{JSW2017} except (irred$_K$) and (HT) (which fails if the weight of $f$ is higher than $2$). For our application to BSD formula in weight $2$ we may still assume (HT), but we will also explain how to remove it for higher weights. 

The first result we need from \textit{op.\ cit.} is the following:\begin{proposition}\label{Selfinite}
	One has\[
	\#\H^1_{\cF_{\ac}}(K,W)=\#\Sha_{\BK}(W/K)\cdot (\#\delta_v)^2,
	\]
	where $\delta_v=\coker\{\H^1_{\cF_{\BK}}(K,T)\xrightarrow{\loc_v}\H^1_f(K,W)/\H^1(K_v,T)_{\tors}\}$. In particular, $\H^1_{\cF_{\ac}}(K,W)$ has finite order.
\end{proposition}  

\begin{proof}
	See~\cite[Proposition 3.2.1]{JSW2017}. Note that by Remark 3.2.2 in \textit{op.\ cit.}, we do not need (irred$_K$) to get $W^*=T^\tau$.
	
	We mention that in the residually reducible setting, $\H^1_f(K,T)$ may no longer be a free $\cO$-module. In particular, their $\delta_v$ fits into the exact sequence\begin{equation}\label{deltav}
	0\to \H^1_f(K,T)_{\tors} \to \H^1_f(K,T) \xrightarrow{\loc_v} \H^1_f(K_v,T)/\H^1_f(K_v,T)_{\tors} \to \delta_v\to 0,
	\end{equation}
	where $\ker(\loc_v)$ may not be $0$. This observation will be helpful for our application to modular abelian varieties in~\cref{B3}.
\end{proof}

The main result of this appendix is the following anticyclotomic control theorem which allows torsion. 
\begin{theorem}[Anticyclotomic control theorem]\label{control}
	The $\Lambda$-module $X^\Sigma_{\ac}(M_f)$ is $\Lambda$-torsion, and if $f^\Sigma_{\ac}(T)$ is a generator of its characteristic $\Lambda$-ideal $\Char(X^\Sigma_{\ac}(M_f))$, then\[
	\#\cO/f^\Sigma_{\ac}(0)=\frac{\#\H^1_{\cF_{\ac}}(K,W)\cdot C^\Sigma(W)}{\#\H^0(K,W)\cdot \#\H^0(K,W)^\vee},\]
	where \[C^\Sigma(W)=\#\H^0(K_v,W)\cdot \#\H^0(K_{\ol v},W)\cdot\prod_{w\in S_p\setminus\Sigma, w\ split}\# \H^1_{\nr}(K_w,W)\cdot \prod_{w\in \Sigma}\#\H^1(K_w,W).\]
	
\end{theorem}
We also define $\mathcal{P}={\displaystyle \prod_{w\in S\setminus\Sigma}\frac{\H^1(K_w,W)}{\H^1_{\cF_{\ac}}(K_w,W)}}$ and $\mathcal{G}=\im\{\H^1(K^S/K,W)\xrightarrow{\loc_S}\mathcal{P}\}.$

We remark that their $S$ and $\Sigma$ are different from ours in section $1$. Their choice of $v,\ol v$ is also different from ours. For our application, the $\Sigma$ above would be taken as our $\Sigma-\{v,\ol v, \infty\}$ in the earlier sections and their $S$ will be chosen to be our $\Sigma\cup\{v,\ol v\}$ so that $X^\Sigma_{\ac}(M_f)$ is the dual of the imprimitive Greenberg Selmer group $\H^1_{\cF^S_{Gr}}(K,M_f)$. In particular, $\mathcal{P}=\frac{\H^1(K_v,W)}{\H^1(K_v,W)_{\div}}\times \H^1(K_{\ol v},W)$.
In the rest of this appendix, we explain how to adapt the proofs in~\cite[section 3]{JSW2017} to our setting, following~\cite{Greenberg1999}.

We begin with recalling~\cite[Lemma 4.2]{Greenberg1999} which says\begin{equation}\label{f(0)}
\#\cO/f^\Sigma_{\ac}(0)\sim \frac{\H^1_{\cF_{\ac}^\Sigma}(K,M_f)^\Gamma}{\H^1_{\cF_{\ac}^\Sigma}(K,M_f)_\Gamma},
\end{equation}
provided $\H^1_{\cF_{\ac}^\Sigma}(K,M_f)^\Gamma$ is finite. For now we assume this finiteness result so we have the above computation of $\#\cO/f^\Sigma_{\ac}(0)$. We will first study the numerator.

\subsection{Computing $\H^1_{\cF_{\ac}^\Sigma}(K,M_f)^\Gamma$}\label{B1}

As in~\cite[section 3.3.9]{JSW2017}, there is a commutative diagram\\
	\begin{tikzcd}
	0\ar[r] & \H^1_{\cF_{\ac}^\Sigma}(K,W)\ar[r]\ar[d,"s"]& \H^1(K^S/K,W) \ar[r,"\loc_S"]\ar[d,"h"]& \mathcal{P}\ar[d,"r"] \\
	0\ar[r]& \H^1_{\cF_{\ac}^\Sigma}(K,M_f)^\Gamma \ar[r]& \H^1(K^S/K,M_f)^\Gamma  \ar[r]&\mathcal{P}(M_f)^\Gamma
\end{tikzcd}\\
where $\mathcal{P}(M_f)={\displaystyle \prod_{w\in S\setminus\Sigma}\frac{\H^1(K_w,M_f)}{\H^1_{\cF_{\ac}}(K_w,M_f)}}$. In \textit{loc.\ cit.}, the first row is also exact on the right because $\loc_S$ is surjective, but this surjectivity may fail for us as we do not assume (irred$_K$). Therefore, we would need to work with a similar commutative diagram\\
\begin{tikzcd}
	0\ar[r] & \H^1_{\cF_{\ac}^\Sigma}(K,W)\ar[r]\ar[d,"s"]& \H^1(K^S/K,W) \ar[r,"\loc_S"]\ar[d,"h"]& \mathcal{G}\ar[r]\ar[d,"g"]&0 \\
	0\ar[r]& \H^1_{\cF_{\ac}^\Sigma}(K,M_f)^\Gamma \ar[r]& \H^1(K^S/K,M_f)^\Gamma  \ar[r]&\mathcal{G}(M_f)^\Gamma&
\end{tikzcd}\\
where $\mathcal{G}(M_f)=\im\{\H^1(K^S/K,M_f)\xrightarrow{\loc_S}\mathcal{P}(M_f)\}$.

As in \textit{loc.\ cit.}, $\coker(h)=0$ since $\H^1(K^S/K,M_f)^\Gamma=\H^1(K^S/K,M_f)[\gamma-1]$ and there is an exact sequence\[
0\to W\to M_f\xrightarrow{\gamma-1} M_f\to 0.\] 
As $\H^1_{\cF_{\ac}}(K,W)$ is finite from~\cref{Selfinite}, snake lemma then yields
\begin{equation}\label{WtoM}
\frac{\#\H^1_{\cF_{\ac}}(K,M_f)^\Gamma}{\#\H^1_{\cF_{\ac}}(K,W)}=\frac{\#\coker(s)}{\#\ker(s)}=\frac{\#\ker(g)}{\#\ker(h)},
\end{equation}
where $\ker(h)=\H^0(K,W)$ which no longer vanishes in our case. We will see shortly that $\ker(g)$ (and hence $\coker(s)$) is finite, which then gives the finiteness of $\H^1_{\cF_{\ac}^\Sigma}(K,M_f)^\Gamma$ we needed. It remains to compute $\ker(g)$.

\subsection{Computing $\ker(g)$}\label{B2}

We first claim that $\mathcal{G}(M_f)=\mathcal{P}(M_f)$, or in other words, the map $\H^1(K^S/K,M_f)\xrightarrow{\loc_S}\mathcal{P}(M_f)\coloneq{\displaystyle \prod_{w\in S\setminus\Sigma}\frac{\H^1(K_w,M_f)}{\H^1_{\cF_{\ac}}(K_w,M_f)}}=\H^1(K_{\ol v},M_f)$ is surjective. This is nothing but the global-to-local map defining the imprimitive Greenberg Selmer group, so it is surjective by~\cref{surj for Self}.

We now prove a lemma.

\begin{lemma}\label[lemma]{H1almostdiv}
	$\H^1(K^\Sigma/K,M_f)$ is almost divisible, i.e., its Pontryagin dual has no non-trivial finite submodules.
\end{lemma}
\begin{proof}
Similarly as in~\cref{Lambda invariants}, this follows from~\cite[Proposition 3, Proposition 5]{Greenberg1989}. Note that by~\cref{corankH1Mf}, $\H^2(K^\Sigma/K,M_f)$ must be cotorsion. Thus it's $0$ and the result follows.
	\end{proof} 

From the proof of~\cref{corankH1Mf}, we see that the $\cO$-corank of $\H^1(K_{\ol v},W)$ is $2$. Since $\corank_\cO\mathcal{P}=\corank_\cO \H^1(K_{\ol v},W)$, the same is true for $\mathcal{P}$. The arguments after Proposition 4.12 in~\cite{Greenberg1999} then shows that $\H^1(K^\Sigma/K,M_f)_\Gamma=0$. Now Theorem 4.7 in \textit{op.\ cit.} gives the equation
\begin{equation*}
\#\ker(g)=\#\ker(r)\cdot\#\H^1_{\cF_{\ac}^\Sigma}(K,M_f)_\Gamma\cdot\#(\mathcal{P}/\mathcal{G}).
\end{equation*}
In our case the surjectivity of the map $(\mathcal{P} \to \mathcal{P}(M_f)^\Gamma)=\big(\frac{\H^1(K_v,W)}{\H^1(K_v,W)_{\div}}\times \H^1(K_{\ol v},W) \to \H^1(K_{\ol v}, M_f)^\Gamma\big)$ follows from the $p$-cohomological dimension of $G(K_{\infty,\ol v}/K_{\ol v})$ being $1$.

Now by Cassels' Theorem (Proposition 4.13 in \textit{op.\ cit.}), since $\H^1_{\cF_{ac}}(K,W)$ is finite by~\cref{Selfinite} and $\H^0(K_w,W)$ is finite by~\cref{H0 of V/T}, we have $\mathcal{P}/\mathcal{G}\cong \H^0(K,W)^\vee$.

Combining this result with~\eqref{f(0)}, \eqref{WtoM} and the computation of $\ker(r)$ in~\cite[Proposition 3.3.7]{JSW2017} (the proof of case 3(b) is the only place where (HT) is used, but by~\cref{H0 of V/T}, we already have the finiteness of $M_f^{G_{K_v}}$ so (HT) is no longer needed), \cref{control} follows.

\subsection{Applications to newforms and modular abelian varieties $A_f$}\label{B3}
Most part of the computations are already done in~\cite[section 3.5]{JSW2017}. 

As is already noted in~\cref{Selfinite}, in our residually reducible setting, $A_f(K)\otimes_{\bZ[f]}\cO=\H^1_f(K,T)$ may no longer be a free $\cO$-module, and we need to replace $A_f(K)$ and $\H^1_f(K,T)$ in \textit{loc.\ cit.} by $A_f(K)_{/\tors}$ and $\H^1_f(K,T)_{/\tors}$ respectively. In particular, from the sequence~\eqref{deltav}, equation (3.5.a) in \textit{op.\ cit.} becomes\[
\#\delta_v=\#[A_f(K_v)_{/\tors}\otimes_{\bZ[f]\otimes\bZ_p}\cO:A_f(K)_{/\tors}\otimes_{\bZ[f]}\cO].
\]
This means in our application to elliptic curves, the above index becomes $[E(K)_{/\tors}:\bZ\cdot P]_p$, which does not include the torsion of $E(K)$. Indeed, as we will see immediately, $E(K)_{\tors}$ has already been considered separately in~\cref{B1} and~\cref{B2}.

It remains to compute the torsion terms $\H^0(K,W)$ and $\H^0(K,W)^\vee$. But these are identified with $(A_f(K)\otimes_{\bZ[f]}\cO)[\frp^\infty]$ and its dual respectively, both having size $\#((A_f(K)\otimes_{\bZ[f]}\cO)_{\tors})_\frp$. When $A_f$ is an elliptic curve, $\bZ[f]=\bZ$, $\cO=\bZ_p$ and the size is $\#(E(K)\otimes\bZ_p)[p^\infty]=\#(E(K)[p^\infty])=(\#E(K)_{\tors})_p$.

\bibliographystyle{amsalpha}
\bibliography{references}

\end{document}